\documentclass[oneside,11pt]{article}

\usepackage{blindtext}

\usepackage{amssymb}    
\usepackage{amsmath}    
\usepackage{bbm}        
\usepackage{hyperref}   
\usepackage{enumitem}   
\usepackage{xcolor}
\usepackage{dsfont}     
\usepackage{subcaption} 
\usepackage{natbib}     
\usepackage{mathrsfs}   
\usepackage{tikz}       
\usetikzlibrary{positioning,arrows.meta,fit,backgrounds}

\begin{document}

\title{Deep Neural Network Approximation of High-Dimensional Parabolic Integro-Differential Equations}

\author{ Marcin Baranek }



\maketitle

\newcommand{\jump}{c}           
\newcommand{\jumpset}{\mathcal{C}\at{q, K_\jump,L_\jump,\beta_{\jump},\nu}}
\newcommand{\jumpsettwo}{\mathcal{C}_2\at{q, K_\jump,L_\jump,\beta_{\jump},\nu}}
\newcommand{\bset}{\mathcal{B}\at{K_b,L_b,\alpha,\beta}}
\newcommand{\p}{2}

\newcommand{\R}{\mathbb{R}}     
\newcommand{\E}{\mathbb{E}}     
\newcommand{\Pra}{\mathbb{P}}
\newcommand{\Z}{\mathcal{Z}}
\newcommand{\dd}{\mathrm{d}}

\newcommand{\tx}{(t, x)}
\newcommand{\tX}{(t, X_t)}
\newcommand{\opL}{\mathcal{L}}    
\newcommand{\dismu}{\tilde{\mu}}  
\newcommand{\rrho}{\tilde{\rho}}

\newcommand{\doverdt}[1]{\frac{\partial #1}{\partial t}}      
\newcommand{\at}[1]{\left( #1\right)}
\newcommand{\norm}[1]{\left\Vert #1\right\Vert}
\newcommand{\modulus}[1]{\left\vert #1\right\vert}

\newcommand{\Var}[1]{\text{Var}\at{ #1}}

\newcommand{\phif}[1]{\phi_{ #1,\delta_{#1}}}               

\newcommand{\size}[1]{\mathrm{size}\at{ #1}}



\newtheorem{theorem}{Theorem}[section]
\newtheorem{lemma}[theorem]{Lemma}
\newtheorem{proposition}[theorem]{Proposition}
\newtheorem{corollary}[theorem]{Corollary}
\newtheorem{definition}[theorem]{Definition}
\newtheorem{assumption}[theorem]{Assumption}
\newtheorem{remark}[theorem]{Remark}
\newtheorem{example}[theorem]{Example}
\newtheorem{claim}[theorem]{Claim}
\newtheorem{conjecture}[theorem]{Conjecture}

\newenvironment{proof}[1][Proof]{%
  \par\noindent\textbf{#1.}\ }{\hfill\ensuremath{\square}\par}

\providecommand{\BlackBox}{\rule{1.5ex}{1.5ex}}

\begin{abstract}
In this article, we investigate the existence of a
deep neural network (DNN) realization capable of approximating solutions to
partial integro-differential equations in a weighted $L^2$ sense. Under
additional quantitative bounds on the local consistency and moment constants,
the resulting complexity estimates have controlled, at most polynomial,
dependence on the dimension.
Using the Feynman-Kac theorem, we express the solution in terms of stochastic differential equations (SDEs).
Based on mean-square and high-probability estimates for Euler--Monte Carlo
estimators, we establish that a realization achieving any prescribed accuracy
exists with nonzero probability, without requiring convergence of a training
algorithm.
We also illustrate the methodology with numerical experiments in dimensions one and ten.
\end{abstract}


\section{Introduction}
Let $T>0$ be a time horizon, independent of the dimension $d$.
We consider the following Cauchy problem for a parabolic partial integro-differential equation (PIDE) on the unbounded domain $[0,T]\times\R^d$
\begin{equation}
    \label{pde_problem}
    \begin{cases}
        -\doverdt{u}\tx = \opL_t u\tx + a\tx\cdot\nabla u\tx + b\tx\\
        u(T,x) =
        g(x).
    \end{cases}
\end{equation}
where $x\in\R^d$ and $t\in[0,T)$.
The operator $\opL$ has the following form
\begin{equation*}
\begin{aligned}
\opL_t u(x) &= \frac{1}{2}\Delta u\tx + \int_{\Z}\bigl[u\at{t, x + \jump(t,x,z)} - u\tx \\
&\qquad - \jump(t,x,z)\cdot\nabla u\tx\bigr]\,\nu(\dd z),
\end{aligned}
\end{equation*}
for $\Z=\R^d\setminus\{0\}$, and $\nu(\dd z)$ is a finite L\'evy measure on $\at{\Z, \mathcal{B}(\Z)}$.
By $\nabla$ we denote the gradient, and by $\Delta$ we denote the Laplacian
operator,
$\Delta u = \sum_{i=1}^{d} \frac{\partial^2 u}{\partial x_i^2}$.
Our objective is not to compute the solution separately at a prescribed list
of points. Instead, we seek one parameterized function
\[
\Phi_\vartheta:[0,T]\times\R^d\longrightarrow\R
\]
capable of approximating $u(t,x)$ in the weighted
$L^2(\mu_f)$ norm introduced later with the error at most $\varepsilon$. The precise
assumptions on the coefficients $u$, $a$, $b$, $c$, and $g$ are provided later in the article.
We propose a architecture for $\Phi_\vartheta$ and show that, under the stated assumptions, 
there exists a choice of parameters $\vartheta$ such that the approximation error is at most $\varepsilon$ with nonzero probability.
The network architecture is constructed using the stochastic representation of the solution and its Euler--Monte Carlo approximation.
The network after trainning can be evaluated at any new point $(t,x)$ by a forward evaluation, without rerunning the Monte Carlo procedure.
We quantify how the required accuracy $\varepsilon$ affects the number of parameters of the network.
Under assumptions about the dependence between the dimension $d$ and some constant,
the number of the network parameters is of at most polynomial growth in $d$.

This global functional representation is the major difference from methods that
store values on a spatial or space--time grid, or that must rerun an
iteration for each new evaluation point. For a fixed realization, inference
at any new $(t,x)$ requires only evaluation of the same finite computational
graph. No evaluation point has to be known in advance, and no costly
stochastic or iterative calculation needs to be repeated.
Thus one finite set of stored parameters and random inputs represents the
approximation on the entire domain. This functional viewpoint is related to
physics-informed neural networks~\cite{raissi_pinn} and deep-learning methods
for high-dimensional problems~\cite{pde_dnn_2,dnn_master_like_MC}.

Partial differential equations with an integral term in the operator $\opL$
appears in many applications, including jump-diffusion models in finance
\cite{application:finance1,application:finance2,application:finance3,applicationfinance4,application:finance5}.
They are used, for example, in the pricing of financial derivatives and in
models with discontinuous changes in asset prices. 
While the dimension is high, the
main challenge is to obtain an accurate approximation on the entire domain without a computational cost that grows exponentially
with the dimension.

Deep neural networks have already been used to approximate solutions of PDEs~\cite{pde_dnn_1,pde_dnn_2,pde_dnn_5,pde_dnn_3,pde_dnn_4,dnn_master_like_heat_eq,pde_dnn_6,dnn_master_like_BS,dnn_master_like_MC,dnn_master_like_Kolmogorov,baranek2026stpinnsdeeplearning}
and to solve selected problems in mathematical finance
\cite{finance_dnn_2,finance_dnn_3,finance_dnn_1}.
Two recent PINN studies are
particularly relevant here. The first one~\cite{pinn_heston_finance}
uses a physics-informed neural network to price options in the Heston model.
The second one~\cite{pinn_financial_pide} considers financial partial integro-differential equations and uses a PINN to approximate their solutions
for selected financial models.

These studies are primarily application-oriented: they construct a network
for a particular financial model and assess its numerical performance.
Our approach has a different objective. We prove, under explicit assumptions,
that a neural network with a specified architecture exists and achieves a
prescribed approximation accuracy in a weighted $L^2$ sense for the entire domain. 

\section{Preliminaries}
We treat a $d$-dimensional vector as a column vector.
By $\norm{\cdot}$ we mean the Euclidean norm for vectors and the
Frobenius norm for matrices. The notation $\mathscr C^{1,2}$ denotes
functions with continuous first-order time derivative and continuous
spatial derivatives up to order two.
For each fixed $d$, work on a complete probability space
$\at{\Omega,\Sigma,\Pra}$ carrying a $d$-dimensional standard Wiener
process $W$ and an independent Poisson random measure $\mu$ on
$[0,\infty)\times\Z$ with intensity $\dd t\,\nu_d(\dd z)$, where
$\nu_d(\Z)=\lambda(d)\in(0,\infty)$. The notation $\lambda(d)$ makes
the dependence of the jump intensity on the dimension explicit.
For fixed $d$ we also write $\nu=\nu_d$.
Let $\Sigma_t$ be the complete and right-continuous filtration generated by $W$ and $\mu$.
The notation $\dismu$ denotes the compensated
random measure,
\[
\dismu(\dd t,\dd z)=\mu(\dd t,\dd z)-\dd t\,\nu(\dd z).
\]
For simplicity, for $(s,x)\in[0,T]\times\R^d$, the $X^{s,x}$ denotes $X|X_s=x$,
analogusly $\E_{s,x}(\cdot) = \E(\cdot|X_s=x)$.
Whenever a jump integral is written with respect to the
compensated random measure, its integrand is predictable and is therefore
evaluated at the left limit $X_{t-}^{s,x}$. 

Throughout the article, by the Burkholder--Davis--Gundy (BDG) inequality we
mean the standard $r$-moment estimate, for $r\in\{2,6\}$: if $X$ is a
càdlàg local martingale with $X_0=0$, then (see, e.g.,~\cite{kunito_inequality})
\begin{equation}
    \label{eq:bdg_inequality}
\E\left[\sup_{0\leq t\leq T}\norm{X_t}^r\right]
\leq C_r\E\left([X]_T^{r/2}\right),
\end{equation}
where $[X]_T$ denotes the scalar optional quadratic variation.


We consider families of coefficients indexed by the dimension. For every
$d\in\mathbb N$, the constants displayed in the definitions below refer to
the $d$-dimensional member of the corresponding family.
In these assumptions, constants and regularity exponents are independent
of $d$ unless their dimension dependence is explicitly written.
The assumption on the dimension dependence of the constants may be skipped, 
but then the approximating network may have an exponential growth in the dimension $d$.

The final part of the paper contains numerical experiments illustrating the
theoretical framework. All examples
presented there are constructed so that their coefficients and data satisfy
the assumptions stated below.

For each dimension $d$, let
$a:[0,T]\times\R^d\to\R^d$ be a Borel-measurable function. 
We assume
\begin{enumerate}[label={(A.\arabic*)},leftmargin=2.5\parindent,align=left]
\item \label{A1} There are constants $c_{a,K},c_{a,L},
\gamma_{a,K},\gamma_{a,L}>0$, independent of $d$, such that
\begin{align*}
K_a(d)&\leq c_{a,K}(1+\log(1+d^{\gamma_{a,K}})),\\
L_a(d)&\leq \sqrt{c_{a,L}(1+\log(1+d^{\gamma_{a,L}}))}.
\end{align*}

\item \label{A2} for all $(t,x)\in[0,T]\times\R^d$
\begin{equation*}
\norm{a(t,x)}\leq K_a(d)(1+\norm{x}).
\end{equation*}

\item \label{A3} The function $a$ is globally Lipschitz in space and H\"older
continuous in time:
\begin{align*}
\norm{a(t,x)-a(t,y)}&\leq L_a(d)\norm{x-y},\\
\norm{a(t,x)-a(s,x)}&\leq L_a(d)|t-s|^{\alpha_a},
\end{align*}
for all $s,t\in[0,T]$ and $x,y\in\R^d$. Here,
$\alpha_a\in(0,1]$ is independent of $d$.
\end{enumerate}

Since the function
$r\mapsto r^q$ for $q\in\{2,6\}$ is convex on $[0,\infty)$, Jensen's inequality gives, for all $u,v\geq0$,
\[
\left(\frac{u+v}{2}\right)^q\leq\frac{u^q+v^q}{2},
\qquad\text{hence}\qquad
(u+v)^q\leq2^{q-1}(u^q+v^q).
\]
Applying this inequality with $u=1$ and $v=\norm{x}$ to the linear-growth
bound in~\ref{A2}, we obtain, for all $(t,x)\in[0,T]\times\R^d$,
\[
\norm{a(t,x)}^q
\leq K_a(d)^q(1+\norm{x})^q
\leq 2^{q-1}K_a(d)^q(1+\norm{x}^q).
\]
Thus, defining $K_{a,q}(d):=2^{1-1/q}K_a(d)$, we obtain the pointwise
$q$-power growth bound
\begin{equation}
\label{eq:drift_q_growth}
\norm{a(t,x)}^q\leq K_{a,q}(d)^q(1+\norm{x}^q),
\qquad K_{a,q}(d):=2^{1-1/q}K_a(d).
\end{equation}


Let
$b:[0,T]\times\R^d\rightarrow\R$ be a Borel-measurable function,
and $\alpha_b\in(0,1]$ independent of $d$.
We assume:

\begin{enumerate}[label={(B.\arabic*)},leftmargin=2.5\parindent,align=left]
\item \label{B1} There are constants $c_b,\gamma_b>0$, independent of $d$, such that
    \[
    K_b(d)\leq c_b\bigl(1+\log(1+d^{\gamma_b})\bigr).
    \]
\item \label{B2} The function $b$ has at most quadratic growth:
\[
\modulus{b(t,x)}\leq K_b(d)(1+\norm{x}^2)
\]
for all $(t,x)\in[0,T]\times\R^d$.
\item \label{B3} The function $b$ is H\"older continuous in time and globally
Lipschitz in space:
\begin{align*}
    \modulus{b(t,x)-b(s,x)}&\leq L_b(d)\modulus{s-t}^{\alpha_b},\\
    \modulus{b(t,x)-b(t,y)}&\leq L_b(d)\norm{x-y},
\end{align*}
    for all $s,t\in[0,T]$ and $x,y\in\R^d$. Moreover,
    $L_b(d)\leq c_{b,L}(1+\log(1+d^{\gamma_{b,L}}))$ for constants
    $c_{b,L},\gamma_{b,L}>0$ independent of $d$.
\item \label{B4} The function $b$ is bounded on $[0,T]\times\R^d$:
    \[
    \norm{b}_{\infty}:=\sup_{(t,x)\in[0,T]\times\R^d}\modulus{b(t,x)}
    \leq M_b(d)<\infty,
    \]
    where
    $M_b(d)\leq c_{b,M}(1+\log(1+d^{\gamma_{b,M}}))$ for constants
    $c_{b,M},\gamma_{b,M}>0$ independent of $d$.
\end{enumerate}
Let $g:\R^d\rightarrow\R$ be a Borel-measurable
function. We assume:

\begin{enumerate}[label={(G.\arabic*)},leftmargin=2.5\parindent,align=left]
\item \label{G1} There are constants $c_{g,K},c_{g,L},c_{g,M},
\gamma_{g,K},\gamma_{g,L},\gamma_{g,M}>0$, independent of $d$, such that
\begin{align*}
K_g(d)&\leq c_{g,K}(1+\log(1+d^{\gamma_{g,K}})),\\
L_g(d)&\leq c_{g,L}(1+\log(1+d^{\gamma_{g,L}})),\\
M_g(d)&\leq c_{g,M}(1+\log(1+d^{\gamma_{g,M}})).
\end{align*}
\item \label{G2} The function $g$ has at most quadratic growth:
\[
\modulus{g(x)}\leq K_g(d)(1+\norm{x}^2)
\]
for all $x\in\R^d$.
\item \label{G3} The function $g$ is globally Lipschitz:
\[
\modulus{g(x)-g(y)}\leq L_g(d)\norm{x-y}
\]
for all $x,y\in\R^d$.
\item \label{G4} The function $g$ is bounded on $\R^d$:
\[
\norm{g}_{\infty}:=\sup_{x\in\R^d}\modulus{g(x)}\leq M_g(d)<\infty.
\]
\end{enumerate}
Let
$c:[0,T]\times\R^d\times\Z\rightarrow\R^d$ be a Borel-measurable jump
coefficient, and
$\beta_\jump\in(0,1]$ be independent of $d$. 

\begin{enumerate}[label={(C.\arabic*)},leftmargin=2.5\parindent,align=left]
\item \label{C1} There are constants $c_{c,K},c_{c,L},
\gamma_{c,K},\gamma_{c,L}>0$, independent of $d$ and $q$, such that
\begin{align*}
K_\jump(d)&\leq c_{c,K}(1+\log(1+d^{\gamma_{c,K}})),\\
L_\jump(d)&\leq \sqrt{c_{c,L}(1+\log(1+d^{\gamma_{c,L}}))}.
\end{align*}
Here, $\beta_\jump\in(0,1]$ is independent of $d$.
\item \label{C2} The coefficient is $q$-integrable: for every $(s,x)\in[0,T]\times\R^d$,
\[
\E_{s,x}\int_s^T\int_{\Z}\norm{c(t,X_{t-}^{s,x},z)}^q\nu(\dd z)\dd t<\infty.
\]
\item \label{C3} At $(t,x)=(0,0)$, the coefficient satisfies
\[
\at{\int_{\Z}\norm{c(0,0,z)}^q\nu(\dd z)}^{1/q}\leq K_\jump(d),
\]
for $q\in\{2,6\}$.
\item \label{C4} The coefficient is globally Lipschitz in space in the
$L^q(\nu)$ norm for $q\in\{2,6\}$:
\[
\at{\int_{\Z}\norm{c(t,x_1,z)-c(t,x_2,z)}^q\nu(\dd z)}^{1/q}
\leq L_\jump(d)\norm{x_1-x_2}
\]
for all $t\in[0,T]$ and $x_1,x_2\in\R^d$.
\item \label{C5} The coefficient is H\"older continuous in time in the
$L^q(\nu)$ norm for $q\in\{2,6\}$, with a bound growing at most linearly in $x$:
\[
\at{\int_{\Z}\norm{c(t_1,x,z)-c(t_2,x,z)}^q\nu(\dd z)}^{1/q}
\leq L_\jump(d)(1+\norm{x})\modulus{t_1-t_2}^{\beta_\jump}
\]
for all $t_1,t_2\in[0,T]$ and $x\in\R^d$.
\end{enumerate}

For each dimension $d$, let $\nu_d$ denote the Lévy measure and let
$\lambda(d)=\nu_d(\Z)\in(0,\infty)$ denote its dimension-dependent intensity,
so that $\nu_d=\lambda(d)\rho_d$ for a probability measure $\rho_d$ on $\Z$.
Assume in addition that, for constants $c_\lambda,\gamma_\lambda,
\underline c_\lambda,\underline\gamma_\lambda>0$ independent of $d$,
\[
    \lambda(d)\leq c_\lambda(1+\log(1+d^{\gamma_\lambda})),
    \qquad
    \lambda(d)^{-1}\leq \underline c_\lambda(1+d)^{\underline\gamma_\lambda}.
\]

For some results we need a stronger spatial regularity assumption, 
therefore we introduce the following pointwise Lipschitz condition.
\begin{enumerate}[label={(C*.\arabic*)},leftmargin=2.5\parindent,align=left]
\item \label{Cstar1} There are a constants $c_{\ell_c},\gamma_{\ell_c}>0$, independent of $d$, such that 
\[
\ell_c(d)\leq c_{\ell_c}(1+d^{\gamma_{\ell_c}}),\qquad \ell_c(d)\lambda(d)\leq c_{\ell_c}\sqrt{1+\log(1+d^{\gamma_{\ell_c}})}
\]
and a measurable set $\mathcal Z_{\mathrm{Lip},d}
\subseteq\Z$ with $\nu(\Z\setminus\mathcal Z_{\mathrm{Lip},d})=0$ such that
\[
\norm{c(t,x,z)-c(t,y,z)}\leq\ell_c(d)\norm{x-y}
\]
for every $z\in\mathcal Z_{\mathrm{Lip},d}$, $t\in[0,T]$, and
$x,y\in\R^d$. The constant is independent of $t$, $z$, and all
discretization and sampling parameters; the same full-measure set works
for all $t$, $x$, and $y$.

\item \label{Cstar2} There are constants
$c_{c,L,6},\gamma_{c,L,6},c_{\nu,6},\gamma_{\nu,6}>0$,
independent of $d$, and constants $L_{c,6}(d),K_{\nu,6}(d)\geq0$ such that
\begin{align*}
L_{c,6}(d)^6
&\leq c_{c,L,6}\bigl(1+\log(1+d^{\gamma_{c,L,6}})\bigr),\\
K_{\nu,6}(d)
&\leq c_{\nu,6}(1+d)^{\gamma_{\nu,6}}.
\end{align*}
The jump coefficient is globally Lipschitz in space in $L^6(\nu_d)$:
\[
\left(\int_{\Z}\norm{c(t,x,z)-c(t,y,z)}^6\nu_d(\dd z)\right)^{1/6}
\leq L_{c,6}(d)\norm{x-y}
\]
for all $t\in[0,T]$ and $x,y\in\R^d$.
Moreover, the L\'evy measure has a finite sixth moment satisfying
\[
\left(\int_{\Z}\norm{z}^6\nu_d(\dd z)\right)^{1/6}
\leq K_{\nu,6}(d)<\infty.
\]
\end{enumerate}
We keep both \ref{C4} and \ref{Cstar1} in the article, since they are used in different.

We introduce an auxiliary inequality.
By Minkowski's inequality and conditions~\ref{C3}--\ref{C5},
\begin{align*}
\norm{c(t,x,\cdot)}_{L^q(\nu)}
&\leq\norm{c(t,x,\cdot)-c(t,0,\cdot)}_{L^q(\nu)}\\
&\qquad
 +\norm{c(t,0,\cdot)-c(0,0,\cdot)}_{L^q(\nu)}
 +\norm{c(0,0,\cdot)}_{L^q(\nu)}\\
&\leq L_\jump(d)\norm{x}+L_\jump(d)T^{\beta_\jump}+K_\jump(d).
\end{align*}
Consequently, for all $(t,x)\in[0,T]\times\R^d$,
\begin{equation}
    \label{jump_int_norm}
    \at{\int_{\Z}\norm{c(t,x,z)}^q\nu(\dd z)}^{1/q}
    \leq
    \max\{L_\jump(d),L_\jump(d) T^{\beta_\jump}+K_\jump(d)\}(1+\norm{x}).
\end{equation}

Since we work on an unbounded domain, we introduce a fixed probability density
$f_d$ on $[0,T]\times\R^d$.
We assume the $f_d:[0,T]\times\R^d\rightarrow[0,\infty)$ is Borel measurable and
\begin{enumerate}[label={(F.\arabic*)},leftmargin=2.5\parindent,align=left]
\item \label{F1}
    $\int_0^T\int_{\R^d}f_d(s,x)\,\dd x\,\dd s=1$ and
    for $q\in\{2,6\}$
    \begin{equation*}
    \int_0^T\int_{\R^d}(1+\norm{x}^q)f_d(s,x)\,\dd x\,\dd s
    \leq a_q(1+\log(1+d^{\gamma_{f,q}})),
    \end{equation*}
    where $a_q,\gamma_{f,q}>0$ are independent of $d$.
\end{enumerate}

For each dimension $d$, let $u:[0,T]\times\R^d\rightarrow\R$ be
Borel-measurable. We assume the following conditions
\begin{enumerate}[label={(U.\arabic*)},leftmargin=2.5\parindent,align=left]
\item \label{U1} The function $u$ satisfies the equation
\eqref{pde_problem} pointwise on $[0,T)\times\R^d$ and the terminal
condition $u(T,x)=g(x)$ for all $x\in\R^d$.
\item \label{U2}
\[
u\in\mathscr{C}^{1,2}\at{[0,T)\times\R^d}.
\]
\item \label{U3} The function $u$ has at most quadratic growth, uniformly
in time: there is a finite constant $K_u>0$, which may depend on $d$ but
not on $t$ or $x$, such that
\[
\max_{t\in[0,T]}\modulus{u(t,x)}\leq K_u(1+\norm{x}^2)
\]
for all $x\in\R^d$. No dimension-uniform bound on $K_u$ is imposed here.
\item \label{U4} The function $u$ is continuous on $[0,T]\times\R^d$.
\end{enumerate}

We next verify the Feynman--Kac representation under these assumptions
(see also~\cite{fayman_kac_Rd}). For each fixed $d$, conditions
\ref{A2}--\ref{A3} and \ref{C3}--\ref{C5} with $q=2$ give global
spatial Lipschitz and linear-growth bounds in the Euclidean and
$L^2(\nu_d)$ norms, respectively. On the filtered space specified above,
these bounds ensure existence and pathwise uniqueness of a strong solution
$X^{s,x}$ (see, e.g.,~\cite[Section~6.2]{kunito_inequality}) of
\begin{equation}
    \label{intro.x_t_definition}
    \begin{cases}
        \dd X_t = a(t,X_{t-})\dd t+\dd W_t
        + \int_{\Z} c(t, X_{t-}, z)\dismu(\dd t, \dd z),\\
        X_s = x\in\R^d.
    \end{cases}
\end{equation}

Under the Lipschitz and linear growth conditions \ref{A2}--\ref{A3} and \ref{C3}--\ref{C5}, 
there exists $C_X(d,T)>0$ (see~\cite[Sec.~6.2]{Applebaum2009}) such that
for every $(s,x)\in[0,T]\times\R^d$
\begin{equation}
\label{eq:exact_maximal_second_moment}
\E_{s,x}\left[\sup_{s\leq r\leq T}\norm{X_r^{s,x}}^2\right]
\leq C_X(d,T)(1+\norm{x}^2).
\end{equation}
This bound enables the verification of the Feynman--Kac representation: it ensures
that the right-hand side of the martingale domination is integrable, allowing
localization to be removed.

In \ref{U1} we assume that $u$ solves the PIDE in order to focus on the approximation
part. This assumption is not restrictive in our setting:
since the diffusion coefficient is the identity (so the operator is
uniformly parabolic), the L\'evy measure $\nu$ is finite, and the
coefficients are Lipschitz in space and H\"older continuous in time. 
The regularity of $u$ and the fact that it solves the
PIDE can then be shown analogously to 
\cite{GarroniMenaldi1992, ContTankov2004, BarlesBuckdahnPardoux1997}.
Thus the Feynman--Kac assumptions are satisfied.

The article is organized as follows. First, we establish
moment, strong-error, and quadrature estimates for the Euler--Monte Carlo
approximation of the process $X_t$. The estimates are stated with explicit
dimension-dependent constants.

Second, we use these estimates to construct a single measurable
network realization whose input is $(t,x)$ and whose output approximates
$u(t,x)$ on the entire domain in the weighted $L^2(\mu_f)$ sense.
We state the dependence of the discretization, Monte Carlo sample sizes, and
stored parameters on the target accuracy and confidence. The resulting
high-probability statement implies, in particular, that a realization with
the prescribed accuracy exists with nonzero probability, without requiring
gradient-based training.

Finally, we describe the training methodology and present numerical
experiments.

\section{Approximation algorithm}
In this section we introduce our approach to the approximation of the solution using DNNs.
First, we define discretization methods and analyze their error.
The second subsection describes the approximation done by deep neural networks.

\subsection{Monte Carlo consistency for Euler trajectories}
\label{subsec:mc_consistency}
We first establish an $L^1$ error estimate for the Monte Carlo trajectory
approximations, weighted by a probability density, and include an explicit
approximation of the compensator integral. We then use these estimates to
obtain the final approximation result in the weighted $L^2(\mu_f)$ norm.
We first state the analytical assumptions and define all constants used in the
error estimate. Throughout this subsection, the dimension $d$ is fixed. We
use $q=2$, including in \ref{C1}--\ref{C5} and \ref{F1}, and write
$\nu=\nu_d=\lambda(d)\rho_d$, $\rho=\rho_d$. Generic constants $C_T$ and
$C_{T,2}$ depend only on the fixed horizon $T$ and the moment order, not
on $d,h,s,x$, and may increase from one estimate to the next. The numerical
martingale constant $C_2$ is independent also of $T$. We define
\[
\mu_f(\dd s,\dd x):=f_d(s,x)\,\dd x\,\dd s.
\]
For $\theta=(s,x)$ we use the shorthand
\[
\int F(\theta)\,\mu_f(\dd\theta)
:=\int_0^T\int_{\R^d}F(s,x)f_d(s,x)\,\dd x\,\dd s.
\]

We use the following notation convention for the Euler processes. The exact
SDE solution is $X^{s,x}$. A bar denotes values on the shifted Euler grid,
an overline denotes the associated piecewise-constant interpolation, and a
tilde denotes the continuous-time Euler interpolation. Thus, for the
exact-compensator scheme, the grid values, continuous interpolation, and
piecewise-constant interpolation are denoted respectively by
$\bar X^{h;s,x}$, $\widetilde X^{h;s,x}$, and
$\overline X^{h;s,x}$. We suppress $(s,x)$ or write
$\theta=(s,x)$ only when the meaning is unambiguous.

For terminal-mean notation take $r=T$ and, for $\theta=(s,x)$, define
\[
m(\theta)=\E_\theta X_r.
\]
This convention does not restrict the error estimates below, which hold
for every $s\leq r\leq T$. Fix a positive integer $N_{\mathrm{Euler}}$ such
that $h=T/N_{\mathrm{Euler}}\leq\min\{1,T\}$, and put $t_j=jh$ for
$0\leq j\leq N_{\mathrm{Euler}}$. For $s<T$, define
\begin{align*}
&j(s)=\min\{j:t_j>s\},\qquad K(s)=N_{\mathrm{Euler}}-j(s)+1,
\qquad \tau_0(s)=s,\\
&\tau_k(s)=t_{j(s)+k-1}\quad(1\leq k\leq K(s)).
\end{align*}
Thus $\tau_{K(s)}(s)=T$. Start at $\bar X_s^{h;\theta}=x$ and, for
$0\leq k<K(s)$, set
\begin{equation}
\label{eq:euler_exact_compensator}
\begin{aligned}
    \bar X_{\tau_{k+1}(s)}^{h;\theta}
    ={}&\bar X_{\tau_k(s)}^{h;\theta}
    +\Delta\tau_k(s)a(\tau_k(s),\bar X_{\tau_k(s)}^{h;\theta})
    +\Delta W_k^s\\
    &+\sum_{n=1}^{N_k^s}c(\tau_k(s),\bar X_{\tau_k(s)}^{h;\theta},Y_{k,n}^s)
    -\Delta\tau_k(s) I_k^s(\bar X_{\tau_k(s)}^{h;\theta}).
\end{aligned}
\end{equation}
Here $\Delta\tau_k(s)=\tau_{k+1}(s)-\tau_k(s)$ and
\[
I_k^s(y):=\int_{\Z}c(\tau_k(s),y,z)\,\nu(\dd z).
\]
This integral is finite by Cauchy--Schwarz and $\nu(\Z)=\lambda(d)<\infty$.
Use the same $W$ and $\mu$ as in the exact SDE: namely,
$\Delta W_k^s=W_{\tau_{k+1}(s)}-W_{\tau_k(s)}$ and
$N_k^s=\mu((\tau_k(s),\tau_{k+1}(s)]\times\Z)$, with $Y_{k,n}^s$ the
marks of these atoms in chronological order. Conditional on the past,
the Brownian increment and the marked Poisson increment are independent,
with laws $\mathcal N(0,\Delta\tau_k(s)I_d)$ and Poisson count
$\operatorname{Poisson}(\lambda(d)\Delta\tau_k(s))$; conditional on the
count, the marks are independent with law $\rho_d$. Increments on disjoint
cells are independent.

Define $\pi_{h,s}(v)=\tau_k(s)$ for
$v\in[\tau_k(s),\tau_{k+1}(s))$ and $\pi_{h,s}(T)=T$.
In particular, $\pi_{h,s}(\tau_k(s))=\tau_k(s)$. The continuous-time
interpolation is
\begin{align}
\label{eq:euler_exact_compensator_continuous}
\widetilde X_t^{h;\theta}
={}&x+\int_s^t a\bigl(\pi_{h,s}(v),
\bar X_{\pi_{h,s}(v)}^{h;\theta}\bigr)\,\dd v
    +W_t-W_s \\
&+\int_s^t\int_{\Z}
 c\bigl(\pi_{h,s}(v),\bar X_{\pi_{h,s}(v)}^{h;\theta},z\bigr)\dismu(\dd v,\dd z).
\end{align}
The frozen Euler coefficient in Poisson integrands in this formula and
below is understood in its predictable version
\[
\sum_{k=0}^{K(s)-1}c(\tau_k(s),\bar X_{\tau_k(s)}^{h;\theta},z)
\mathbf 1_{(\tau_k(s),\tau_{k+1}(s)]}(v).
\]
The versions differ only at deterministic grid times, a set of zero
$\dd v\,\nu(\dd z)$ measure and almost surely containing no Poisson atoms.
The associated piecewise-constant interpolation is
\begin{equation*}
\overline X_v^{h;\theta}:=\bar X_{\pi_{h,s}(v)}^{h;\theta},
\qquad v\in[s,T].
\end{equation*}
In estimates involving a continuous-time process we write
$\widetilde X_v^{h;\theta}$ explicitly; values at grid points
always refer to $\bar X_{\tau_k(s)}^{h;\theta}$.
For $s=T$, set $K(T)=0$, $\tau_0(T)=T$, and $\pi_{h,T}(T)=T$;
all three Euler processes equal $x$ on the singleton interval $[T,T]$.

\begin{proposition}
\label{prop:dimension_explicit_euler_constant}
Assume \ref{A1}--\ref{A3} and \ref{C1}--\ref{C5}. We work throughout with
the second moment. Then there are constants $A_d<\infty$ and
$\alpha_{\mathrm E}=\min\{1/2,\alpha_a,\beta_\jump\}$ such that
\[
 \left(\E_\theta\norm{X_r-\widetilde X_r^{h;\theta}}^2\right)^{1/2}
 \leq A_d^{1/2}h^{\alpha_{\mathrm E}}(1+\norm{x}),
\]
for all $\theta=(s,x)$ and $0\le s\le r\le T$. Moreover, the dependence on
$d$ is at most polynomial: there are constants
$C,\kappa_A>0$, independent of $d$, such
that
\[
 A_d\le C(1+d)^{\kappa_A}.
\]

\begin{proof}
Let $e_t=X_t-\widetilde X_t^{h}$. We use the continuous-time representation of the
exact-compensator Euler scheme,
\begin{align*}
\widetilde X_t^{h}
={}&x+\int_s^t a\bigl(\pi_{h,s}(v),\bar X_{\pi_{h,s}(v)}^{h}\bigr)\,\dd v
    +W_t-W_s \\
&+\int_s^t\int_{\Z}
 c\bigl(\pi_{h,s}(v),\bar X_{\pi_{h,s}(v)}^{h},z\bigr)\dismu(\dd v,\dd z).
\end{align*}
The Brownian terms in the representations of $X_t$ and
$\widetilde X_t^{h}$
cancel. Subtracting the two representations therefore gives
\begin{align}
e_t={}&\int_s^t\Bigl[a(v,X_v)-a\bigl(\pi_{h,s}(v),\bar X_{\pi_{h,s}(v)}^{h}\bigr)\Bigr]\,\dd v\nonumber \\
&+\int_s^t\int_{\Z}
\Bigl[c(v,X_{v-},z)-c\bigl(\pi_{h,s}(v),\bar X_{\pi_{h,s}(v)}^{h},z\bigr)\Bigr]
\dismu(\dd v,\dd z).
\label{eq:euler_error_decomposition}
\end{align}
We split each coefficient difference into a propagation error and a local
discretization error. Define
\begin{align*}
E_a(v)&=a(v,X_v)-a(v,\bar X_{\pi_{h,s}(v)}^{h}),\\
R_a(v)&=a(v,\bar X_{\pi_{h,s}(v)}^{h})
    -a(\pi_{h,s}(v),\bar X_{\pi_{h,s}(v)}^{h}),\\
E_c(v,z)&=c(v,X_{v-},z)-c(v,\bar X_{\pi_{h,s}(v)}^{h},z),\\
R_c(v,z)&=c(v,\bar X_{\pi_{h,s}(v)}^{h},z)
    -c(\pi_{h,s}(v),\bar X_{\pi_{h,s}(v)}^{h},z).
\end{align*}
Then \eqref{eq:euler_error_decomposition} becomes
\begin{align}
e_t={}&\int_s^t\bigl[E_a(v)+R_a(v)\bigr] \dd v \\
&+\int_s^t\int_{\Z}\bigl[E_c(v,z)+R_c(v,z)\bigr]
\dismu(\dd v,\dd z).
\label{eq:euler_error_split}
\end{align}
Applying the inequality $(a+b+c+d)^2\leq4
(a^2+b^2+c^2+d^2)$ to the two integrals after separating $E_a$, $R_a$,
$E_c$, and $R_c$, we obtain the following bound. If needed, all maximal
estimates are first applied to stopped processes; the resulting bounds
are uniform in the localization level, which is removed after
Gr\"onwall by Fatou's lemma.
\begin{align}
\E\sup_{s\leq r\leq t}\norm{e_r}^2
&\leq 4\Bigg(
\E\left[\int_s^t\norm{E_a(v)}\,\dd v\right]^2 \nonumber\\
&\quad+\E\left[\int_s^t\norm{R_a(v)}\,\dd v\right]^2 \nonumber\\
&+\E\sup_{s\leq r\leq t}
\norm{\int_s^r\int_{\Z}E_c(v,z)\dismu(\dd v,\dd z)}^2 \nonumber\\
&+\E\sup_{s\leq r\leq t}
\norm{\int_s^r\int_{\Z}R_c(v,z)\dismu(\dd v,\dd z)}^2\Bigg).
\label{eq:euler_error_sum_bound}
\end{align}
For the drift propagation term, the spatial Lipschitz condition in \ref{A3}
and $X_v=\widetilde X_v^{h}+e_v$ give
\begin{align}
\norm{E_a(v)}
&\leq L_a(d)\norm{X_v-\bar X_{\pi_{h,s}(v)}^{h}} \nonumber\\
&\leq L_a(d)\left(\norm{e_v}
    +\norm{\widetilde X_v^{h}-\bar X_{\pi_{h,s}(v)}^{h}}\right).
\label{eq:euler_drift_propagation}
\end{align}
Put $\widehat K_\jump(d):=K_\jump(d)+T^{\beta_\jump}L_\jump(d)$.
By \ref{C3}--\ref{C5} and the triangle inequality,
\begin{align}
\int_{\Z}\norm{c(t,y,z)}^2\nu(\dd z)
&\leq 2\Bigl[\widehat K_\jump(d)^2
 +L_\jump(d)^2\norm{y}^2\Bigr].
\label{eq:euler_jump_growth_second_moment}
\end{align}

For the drift, \ref{A2} at zero and the spatial part of \ref{A3} give
\[
\norm{a(t,y)}\leq K_a(d)+L_a(d)\norm{y},\qquad
\norm{a(t,y)}^2\leq2K_a(d)^2+2L_a(d)^2\norm{y}^2.
\]
These bounds and the conditional Brownian and Poisson second-moment
identities imply finite second moments of all Euler grid values by induction
from $\bar X_s^h=x$. They also imply finite moments of the interpolation
on each cell, so the following isometries are justified. Let
\[
Y_v:=\sup_{s\leq w\leq v}\E\norm{\widetilde X_w^h}^2.
\]
This also bounds every grid second moment up to $v$, since the
interpolation agrees with the grid process there. For each $w\leq v$,
apply the four-term square inequality to
\eqref{eq:euler_exact_compensator_continuous}, Cauchy--Schwarz to its drift
integral, and the Brownian and compensated-Poisson isometries. Taking
the supremum over $w\in[s,v]$ gives the explicit estimate
\begin{align*}
    Y_v\leq
\bigl(4+8T^2K_a(d)^2+4Td+8T\widehat K_\jump(d)^2\bigr)(1+\norm{x}^2)\\
+\bigl(8TL_a(d)^2+8L_\jump(d)^2\bigr)\int_s^v Y_r\,\dd r.
\end{align*}
In particular,
\begin{equation}
Y_v\leq C_T\bigl(d+1+K_a(d)^2+\widehat K_\jump(d)^2\bigr)(1+\norm{x}^2)
+C_T\int_s^v\bigl(1+L_a(d)^2+L_\jump(d)^2\bigr)Y_r\,\dd r.
\label{eq:euler_moment_gronwall_prebound}
\end{equation}
where $C_T>0$ is independent of $d$, $h$, $s$, and $x$.

Gr\"onwall's lemma \cite{gronwal_bound} proves
\begin{equation}
\sup_{s\leq v\leq T}\E\norm{\widetilde X_v^h}^2
\vee\sup_k\E\norm{\bar X_{\tau_k(s)}^h}^2
\leq M_d(1+\norm{x}^2),
\label{eq:euler_interpolation_moment}
\end{equation}
where
\begin{equation}
M_d:=C_T\bigl[d+1+K_a(d)^2+\widehat K_\jump(d)^2\bigr]
\exp\!\Bigl(C_T\bigl[1+L_a(d)^2+L_\jump(d)^2\bigr]\Bigr),
\label{eq:euler_moment_constant_definition}
\end{equation}
with $C_T$ enlarged to absorb the time integration in Gr\"onwall's lemma.

At $v=T$ the interpolation increment is zero. Otherwise fix $v\in[s,T)$
and set $\pi:=\pi_{h,s}(v)$. Since $v$ and $\pi$
belong to the same Euler cell, the interpolation formula
\eqref{eq:euler_exact_compensator_continuous} gives the local increment
decomposition
\begin{align*}
\widetilde X_v^h-\bar X_\pi^h
={}&\int_\pi^v a\bigl(\pi,\bar X_\pi^h\bigr)\,\dd r
 +W_v-W_\pi\\
&+\int_\pi^v\int_{\Z}
c\bigl(\pi,\bar X_\pi^h,z\bigr)\dismu(\dd r,\dd z).
\end{align*}
We estimate the second moment of this increment by estimating these three
contributions separately. Cauchy--Schwarz in time controls the first term,
It\^o's isometry controls the second term,
 and the compensated-Poisson isometry controls the last term. Since
\(0\leq v-\pi\leq h\), the resulting bound is
\begin{align}
\E&\norm{\widetilde X_v^h-\bar X_{\pi_{h,s}(v)}^h}^2
\leq \nonumber\\
&C_T\Bigl[d+K_a(d)^2(1+M_d)
+\widehat K_\jump(d)^2+L_\jump(d)^2M_d\Bigr]
 h(1+\norm{x}^2).
\label{eq:euler_increment_explicit_proof}
\end{align}
More explicitly, the three terms in the square bracket come from the
following estimates:
\begin{align*}
\E\left\|\int_\pi^v a\bigl(\pi,\bar X_\pi^h\bigr)\,\dd r\right\|^2
&\leq C_TK_a(d)^2(1+M_d)h(1+\norm{x}^2),\\
\E\|W_v-W_\pi\|^2
&=d(v-\pi)\leq dh(1+\norm{x}^2),\\
\E\left\|\int_\pi^v\int_{\Z}
c\bigl(\pi,\bar X_\pi^h,z\bigr)\dismu(\dd r,\dd z)\right\|^2
&\leq C_T\bigl[\widehat K_\jump(d)^2+L_\jump(d)^2M_d\bigr]
 h(1+\norm{x}^2).
\end{align*}
Here the last inequality uses \eqref{eq:euler_jump_growth_second_moment},
while the bound for $M_d$ is given by \eqref{eq:euler_interpolation_moment}.
Since $\alpha_{\mathrm E}\leq1/2$ and $h\leq1$, this proves
\begin{equation}
\E\norm{\widetilde X_v^h-\bar X_{\pi_{h,s}(v)}^h}^2
\leq C_{\mathrm{inc}}(d)h^{2\alpha_{\mathrm E}}(1+\norm{x}^2),
\label{eq:euler_increment_dimension_bound}
\end{equation}
where
\[
C_{\mathrm{inc}}(d):=C_T\Bigl[d+K_a(d)^2(1+M_d)
 +\widehat K_\jump(d)^2+L_\jump(d)^2M_d\Bigr],
\]
and $C_T$ is independent of $d$, $h$, $s$, and $x$. Thus both majorants are
proved before the local-consistency constant is introduced. We now define
that constant by
\begin{equation}
\label{eq:euler_local_constant_definition}
C_{2,d}:=C_T\bigl(1+L_a(d)^2+L_\jump(d)^2\bigr)
\bigl(1+M_d+C_{\mathrm{inc}}(d)\bigr),
\end{equation}
where $C_T$ is independent of $d$. Its dimension dependence is therefore
obtained from the explicit estimates proved above, not imported from the
literature.

The constant $C_2$ appearing in the BDG estimates below is a numerical
second-moment Doob/BDG constant (one may take $C_2=4$), not a
drift-propagation constant. Cauchy--Schwarz for the drift also contributes
a factor at most $T$. After absorbing this factor and the numerical
square-inequality factors into $C_{T,2}$, set
\begin{equation}
\label{eq:euler_propagation_constant}
C_{\mathrm{prop}}(d):=C_{T,2}\bigl(1+L_a(d)^2+L_\jump(d)^2\bigr),
\end{equation}
where the representative $C_{T,2}$ in this definition is chosen large
enough to dominate all propagation factors below.
All dimension dependence in the error estimate is consequently contained in
the coefficient bounds and in the moment and increment constants displayed
above; the BDG inequality contributes no additional dependence on $d$.

Using Hölder's inequality in time and the explicit $L^2$ increment
estimate derived above for the continuous left-point interpolation, the part
containing the increment of $\widetilde X^{h}$ is a local consistency term.
Hence
\begin{align}
\E\left[\int_s^t\norm{E_a(v)}\,\dd v\right]^2
\leq C_{T,2}L_a(d)^2\int_s^t
\E\sup_{s\leq q\leq v}\norm{e_q}^2\,\dd v
+C_{2,d}h^{2\alpha_{\mathrm E}}(1+\norm{x}^2).
\label{eq:euler_drift_bound}
\end{align}
Before absorbing constants the propagation coefficient is
$2TL_a(d)^2$ and the local term is at most
$2T^2L_a(d)^2C_{\mathrm{inc}}(d)h^{2\alpha_{\mathrm E}}(1+\norm{x}^2)$.
The time regularity in \ref{A3} gives
$\norm{R_a(v)}\leq L_a(d)h^{\alpha_a}$, hence
\begin{align}
\E\left[\int_s^t\norm{R_a(v)}\,\dd v\right]^2
&\leq L_a(d)^2T^2h^{2\alpha_a}\nonumber\\
&\leq C_{2,d}h^{2\alpha_{\mathrm E}}(1+\norm{x}^2).
\label{eq:euler_drift_local_bound}
\end{align}
For the jump propagation term, the spatial condition in \ref{C4} yields
\begin{align}
\int_{\Z}\norm{E_c(v,z)}^2\nu(\dd z)
&\leq L_\jump(d)^2\norm{X_{v-}-\bar X_{\pi_{h,s}(v)}^{h}}^2 \\
&\leq 2L_\jump(d)^2\left(\norm{e_{v-}}^2
 +\norm{\widetilde X_{v-}^{h}-\bar X_{\pi_{h,s}(v)}^{h}}^2\right).
\label{eq:euler_jump_propagation}
\end{align}
The compensated-Poisson isometry gives
\[
\E\left\|\int_s^t\int_{\Z}E_c(v,z)\dismu(\dd v,\dd z)\right\|^2
=\E\int_s^t\int_{\Z}\norm{E_c(v,z)}^2\nu(\dd z)\dd v,
\]
and the $L^2$ Doob/BDG maximal inequality, with constant $C_2$, therefore
gives (absorbing the additional numerical factors into $C_{T,2}$)
\begin{align}
&\E\sup_{s\leq r\leq t}
\norm{\int_s^r\int_{\Z}E_c(v,z)\dismu(\dd v,\dd z)}^2 \\
&\qquad\leq C_{T,2}L_\jump(d)^2
\int_s^t\E\sup_{s\leq q\leq v}\norm{e_q}^2\,\dd v
+C_{2,d}h^{2\alpha_{\mathrm E}}(1+\norm{x}^2).
\label{eq:euler_jump_bound}
\end{align}
Here deterministic times carry no jumps almost surely, so the increment
estimate also applies to left limits for $\dd v$-almost every $v$.
Finally, \ref{C5} gives
$\int_{\Z}\norm{R_c(v,z)}^2\nu(\dd z)
\leq2L_\jump(d)^2\left(1+\norm{\bar X_{\pi_{h,s}(v)}^h}^2\right)h^{2\beta_\jump}$.
Thus the same maximal inequality yields
\begin{align}
&\E\sup_{s\leq r\leq t}
\norm{\int_s^r\int_{\Z}R_c(v,z)\dismu(\dd v,\dd z)}^2 \\
&\qquad\leq 2C_2L_\jump(d)^2h^{2\beta_\jump}
\int_s^t\bigl(1+\E\norm{\bar X_{\pi_{h,s}(v)}^h}^2\bigr)\,\dd v\nonumber\\
&\qquad\leq C_{2,d}h^{2\alpha_{\mathrm E}}(1+\norm{x}^2).
\label{eq:euler_jump_local_bound}
\end{align}
The definition \eqref{eq:euler_local_constant_definition} now shows that the
drift and jump propagation terms use $C_{\mathrm{inc}}(d)$, whereas the
time-local terms use $M_d$; hence it dominates all four local contributions
after enlarging $C_T$. Set $A_{d,\mathrm{loc}}:=C_T C_{2,d}$. We now derive the
dependence of these constants on $d$. Define, for a constant $C>0$ independent of $d$,
\[
\ell_d:=1+\log(1+d^{\Gamma}),
\]
where $\Gamma>0$ is chosen large enough to dominate the finitely many
logarithmic dimension factors appearing in the assumptions. Specifically,
\ref{A1} and \ref{C1} give
$K_a(d)^2+\widehat K_\jump(d)^2\leq C\ell_d^2$ and
$L_a(d)^2+L_\jump(d)^2\leq C\ell_d$. Thus
\[
M_d\leq C\bigl(d+1+\ell_d^2\bigr)\exp(C\ell_d).
\]
The explicit $d$ term comes from Brownian variance. Since
$\exp(C\ell_d)\leq C(1+d)^{C\Gamma}$, this is bounded by a fixed
polynomial in $d$. More explicitly, there is $\kappa_M>0$, independent
of $d$, such that
\begin{equation*}
    M_d\leq C(1+d)^{\kappa_M}.
\end{equation*}
Consequently, \eqref{eq:euler_increment_explicit_proof} yields
\begin{equation}
\label{eq:euler_increment_polynomial}
C_{\mathrm{inc}}(d)\leq C(1+d)\ell_d^2(1+M_d)
    \leq C(1+d)^{\kappa_{\mathrm{inc}}}
\end{equation}
for some $\kappa_{\mathrm{inc}}>0$ independent of $d$. Finally, the
definition \eqref{eq:euler_local_constant_definition} gives
\begin{equation}
\label{eq:euler_local_constant_polynomial}
C_{2,d}\leq C(1+d)^{\kappa_2},
\qquad
 A_{d,\mathrm{loc}}\leq C(1+d)^{\kappa_A},
\end{equation}
where $\kappa_2,\kappa_A>0$ and $C$ are independent of $d$. Thus the
dimension dependence of the local consistency constant is polynomial.

Combining \eqref{eq:euler_error_sum_bound}--\eqref{eq:euler_jump_local_bound}
and choosing $C_{T,2}$ in \eqref{eq:euler_propagation_constant} sufficiently
large gives
\begin{align}
 \E\sup_{s\leq r\leq t}\norm{e_r}^2
 &\leq A_{d,\mathrm{loc}}h^{2\alpha_{\mathrm E}}(1+\norm{x}^2) \\
 &\quad+C_{T,2}\int_s^t\bigl(1+L_a(d)^2+L_\jump(d)^2\bigr)
    \E\sup_{s\leq q\leq v}\norm{e_q}^2\,\dd v.
\label{eq:euler_error_integral_inequality}
\end{align}
The constant $A_{d,\mathrm{loc}}$ collects the local-consistency bounds and is independent
of $h$, $r$, and $(s,x)$. Grönwall's lemma therefore gives the explicit choice
\begin{equation}
\label{eq:euler_global_constant}
A_d:=A_{d,\mathrm{loc}}\exp\!\bigl(TC_{\mathrm{prop}}(d)\bigr).
\end{equation}
Indeed,
\[
 \E\sup_{s\leq v\leq r}\norm{e_v}^2
 \leq A_dh^{2\alpha_{\mathrm E}}(1+\norm{x}^2).
\]
The polynomial bound for $A_{d,\mathrm{loc}}$, together with
$L_a(d)^2+L_\jump(d)^2=O(\log(1+d))$, shows that
$A_d\leq C(1+d)^{\kappa_A}$ after possibly enlarging $C$ and
$\kappa_A$, independently of $d$.
It is the logarithmic bound on the squared Lipschitz constants that is
essential here; a generic polylogarithmic bound would not suffice.
Taking the square root finishes the proof.
\end{proof}
\end{proposition}

    As a direct consequence of Proposition~\ref{prop:dimension_explicit_euler_constant},
    we obtain an estimate for the difference of the conditional expectations.
    With $Z_\theta=X_r^{\theta}-\widetilde X_r^{h;\theta}$, the triangle
    inequality and Cauchy--Schwarz give
    \begin{align*}
    \norm{\E_\theta\widetilde X_r^{h;\theta}-\E_\theta X_r}
    &=\norm{\E_\theta Z_\theta}
    \leq \E_\theta\norm{Z_\theta}\\
    &\leq \left(\E_\theta\norm{Z_\theta}^2\right)^{1/2}
    \leq A_d^{1/2}h^{\alpha_{\mathrm E}}(1+\norm{x}).
    \end{align*}
    We now integrate this pointwise bound with respect to
    $\mu_f(\dd\theta)=f_d(s,x)\,\dd x\,\dd s$. Put
    \[
    F_{2,d}:=a_2\bigl(1+\log(1+d^{\gamma_{f,2}})\bigr).
    \]
    Since $\mu_f$ is a probability measure, Cauchy--Schwarz and \ref{F1}
    with $q=2$ give
    \begin{align*}
    \int (1+\norm{x})\,\mu_f(\dd\theta)
    &\leq\left[\int(1+\norm{x})^2\,\mu_f(\dd\theta)\right]^{1/2}\\
    &\leq\left[2\int(1+\norm{x}^2)\,\mu_f(\dd\theta)\right]^{1/2}
    \leq\sqrt{2F_{2,d}}.
    \end{align*}
    Therefore,
    \begin{equation}
    \label{eq:euler_discretization_weighted_error}
    \begin{aligned}
    \int \norm{\E_\theta\widetilde X_r^{h;\theta}-\E_\theta X_r}\,\mu_f(\dd\theta)
    &\le C_E(d)h^{\alpha_{\mathrm E}},\\
    C_E(d)&:=A_d^{1/2}\int_0^T\int_{\R^d}(1+\norm{x})f_d(s,x)\,\dd x\,\dd s\\
    &\leq A_d^{1/2}\sqrt{2F_{2,d}}
    \le C(1+d)^{\kappa_E},
    \end{aligned}
    \end{equation}
    where $C,\kappa_E>0$ are independent of $d$. The last inequality follows
    from Proposition~\ref{prop:dimension_explicit_euler_constant}, which gives
    $A_d\leq C(1+d)^{\kappa_A}$, and from the elementary fact that
    $1+\log(1+d^{\gamma_{f,2}})$ is bounded by a (possibly different) fixed
    power of $1+d$. Thus the weighted Euler discretization error has the same
    convergence order $h^{\alpha_{\mathrm E}}$, with a constant whose
    dependence on the dimension is at most polynomial.
    \begin{lemma}
    \label{lem:running_source_quadrature}
    Assume \ref{A1}--\ref{A3}, \ref{B1}--\ref{B3}, \ref{C2}--\ref{C5},
    and \ref{F1}. Define
    \begin{align*}
    Q_h^b(s,x)&=\sum_{k:\tau_k(s)\in[s,T)}\Delta\tau_k(s)
    b\bigl(\tau_k(s),\bar X_{\tau_k(s)}^{h;s,x}\bigr),\\
    I^b(s,x)&=\int_s^T b\bigl(v,X_v^{s,x}\bigr)\,\dd v.
    \end{align*}
    With $\alpha_Q=\min\{\alpha_{\mathrm E},\alpha_b\}$, there is a
    constant $C_Q(d)$, independent of $h$, $s$, and $x$, such that
    \begin{equation}
    \label{eq:running_source_quadrature}
    \left(\E_{s,x}\left|Q_h^b(s,x)-I^b(s,x)\right|^2\right)^{1/2}
    \leq C_Q(d)h^{\alpha_Q}(1+\norm{x}).
    \end{equation}
    \end{lemma}

    \begin{proof}
    Let $\pi_{h,s}(v)$ be the left endpoint of the Euler cell containing $v$.
    Since the cell lengths partition $[s,T]$,
    \begin{align*}
    Q_h^b(s,x)-I^b(s,x)
    =\int_s^T\bigl[&b\bigl(\pi_{h,s}(v),
    \bar X_{\pi_{h,s}(v)}^{h;s,x}\bigr)-b\bigl(v,X_v^{s,x}\bigr)\bigr]\,\dd v.
    \end{align*}
    For brevity, write $\pi=\pi_{h,s}(v)$,
    $\bar X_\pi=\bar X_{\pi_{h,s}(v)}^{h;s,x}$, and
    $X_\pi=X_{\pi_{h,s}(v)}^{s,x}$. We add and subtract
    $b(\pi,X_\pi)$ and $b(v,X_\pi)$ to obtain
    \begin{align*}
    D_v&:=b(\pi,\bar X_\pi)-b(v,X_v^{s,x})\\
    &=\underbrace{b(\pi,\bar X_\pi)-b(\pi,X_\pi)}_{D_v^{(1)}}
    +\underbrace{b(\pi,X_\pi)-b(v,X_\pi)}_{D_v^{(2)}}
    +\underbrace{b(v,X_\pi)-b(v,X_v^{s,x})}_{D_v^{(3)}}.
    \end{align*}
    By the spatial Lipschitz estimate in B.3,
    \begin{equation*}
    |D_v^{(1)}|\leq L_b(d)\norm{\bar X_\pi-X_\pi},
    \qquad
    |D_v^{(3)}|\leq L_b(d)\norm{X_\pi-X_v^{s,x}}.
    \end{equation*}
    By the time-Hölder estimate in B.3 and the fact that
    $|v-\pi|\leq h$,
    \begin{equation*}
    |D_v^{(2)}|\leq L_b(d)|v-\pi|^{\alpha_b}
    \leq L_b(d)h^{\alpha_b}.
    \end{equation*}
    Hence, by the triangle inequality in $L^2(\E_{s,x})$,
    \begin{align*}
    &\left(\E_{s,x}|D_v|^2\right)^{1/2}\\
    &\qquad\leq\left(\E_{s,x}|D_v^{(1)}|^2\right)^{1/2}
    +\left(\E_{s,x}|D_v^{(2)}|^2\right)^{1/2}
    +\left(\E_{s,x}|D_v^{(3)}|^2\right)^{1/2}\\
    &\qquad\leq L_b(d)\left[
    \left(\E_{s,x}\norm{\bar X_\pi-X_\pi}^2\right)^{1/2}
    +h^{\alpha_b}
    +\left(\E_{s,x}\norm{X_\pi-X_v^{s,x}}^2\right)^{1/2}
    \right].
    \end{align*}
    Since $\pi$ is an Euler grid point, the continuous interpolation agrees
    with the grid process there, that is,
    $\widetilde X_\pi^{h;s,x}=\bar X_\pi^{h;s,x}$. The fixed-$d$ estimate
    in the proof of Proposition~\ref{prop:dimension_explicit_euler_constant}
    uses only \ref{C2}--\ref{C5} and finite coefficient constants;
    \ref{C1} is used solely to obtain polynomial dimension dependence.
    Thus that proof applies under the present assumptions, with terminal
    time $r=\pi$, and gives
    \begin{equation*}
    \left(\E_{s,x}\norm{\bar X_\pi-X_\pi}^2\right)^{1/2}
    =\left(\E_{s,x}\norm{\widetilde X_\pi^{h;s,x}-X_\pi^{s,x}}^2\right)^{1/2}
    \leq A_d^{1/2}h^{\alpha_{\mathrm E}}(1+\norm{x}).
    \end{equation*}
    For completeness, the same second-moment argument applies to the exact
    SDE, with the coefficients evaluated at $X_v$ and $X_{v-}$. Using
    $\norm{a(t,y)}^2\leq2K_a(d)^2+2L_a(d)^2\norm{y}^2$ and
    \eqref{eq:euler_jump_growth_second_moment}, localization followed by
    Gr\"onwall's lemma and Fatou's lemma gives
    \begin{align*}
    \sup_{s\leq u\leq T}\E_{s,x}\norm{X_u^{s,x}}^2
    &\leq M_X(d)(1+\norm{x}^2),\\
    M_X(d)&:=C_T\bigl[d+1+K_a(d)^2+\widehat K_\jump(d)^2\bigr]
    e^{C_T(1+L_a(d)^2+L_\jump(d)^2)}.
    \end{align*}
    More precisely, before Gr\"onwall the supremum of second moments
    satisfies the same integral prebound as $Y_v$ above. This uses the
    spatial Lipschitz bound for the drift, not its coarser linear-growth
    constant in the exponent.
    For $s\leq w\leq v\leq T$ put $\delta=v-w$. The SDE, the three-term
    square inequality, and the two isometries now imply
    \begin{align*}
    \E_{s,x}\norm{X_v^{s,x}-X_w^{s,x}}^2
    &\leq3\delta\int_w^v\E_{s,x}\norm{a(r,X_r^{s,x})}^2\,\dd r
        +3d\delta\\
    &\quad+3\int_w^v\E_{s,x}\int_{\Z}
        \norm{c(r,X_{r-}^{s,x},z)}^2\nu(\dd z)\,\dd r\\
    &\leq C_X(d)^2\delta(1+\norm{x}^2),
    \end{align*}
    where, since $\delta\leq T$, one may take
    \[
    C_X(d)^2:=6T\bigl[K_a(d)^2+L_a(d)^2M_X(d)\bigr]+3d
        +6\bigl[\widehat K_\jump(d)^2+L_\jump(d)^2M_X(d)\bigr].
    \]
    Taking square roots and setting $w=\pi$ yields
    \begin{equation*}
    \left(\E_{s,x}\norm{X_\pi-X_v^{s,x}}^2\right)^{1/2}
    \leq C_X(d)|v-\pi|^{1/2}(1+\norm{x})
    \leq C_X(d)h^{1/2}(1+\norm{x}).
    \end{equation*}
    Minkowski's integral inequality therefore gives
    \begin{align*}
    \left(\E_{s,x}|Q_h^b(s,x)-I^b(s,x)|^2\right)^{1/2}
    &\leq\int_s^T\left(\E_{s,x}|D_v|^2\right)^{1/2}\,\dd v\\
    &\leq TL_b(d)\bigl[A_d^{1/2}h^{\alpha_{\mathrm E}}
        +h^{\alpha_b}+C_X(d)h^{1/2}\bigr](1+\norm{x})\\
    &\leq C_Q(d)h^{\alpha_Q}(1+\norm{x}),
    \end{align*}
    with $C_Q(d):=TL_b(d)(\sqrt{A_d}+1+C_X(d))$, because
    $\alpha_Q\leq\alpha_{\mathrm E}\leq1/2$, $\alpha_Q\leq\alpha_b$,
    and $h\leq1$. This proves \eqref{eq:running_source_quadrature}.
    If \ref{C1} is imposed as well, $M_X(d)$ and $C_X(d)$ are polynomially
    bounded by the same squared-Lipschitz argument used for $M_d$;
    together with \ref{B3} and the bound for $A_d$, this also gives a
    polynomial bound for $C_Q(d)$.
    \end{proof}

        As a consequence of Lemma~\ref{lem:running_source_quadrature}, we first
    square its pointwise estimate. Since both sides are nonnegative, this
    yields, for every $(s,x)$,
    \begin{equation*}
    \E_{s,x}|Q_h^b(s,x)-I^b(s,x)|^2
    \leq C_Q(d)^2h^{2\alpha_Q}(1+\norm{x})^2.
    \end{equation*}
    We may then integrate this pointwise inequality with respect to the
    nonnegative measure $\mu_f$. By monotonicity of the integral,
    \begin{equation*}
    \int\E_{s,x}\left|Q_h^b(s,x)-I^b(s,x)\right|^2\,\mu_f(\dd s,\dd x)
    \leq C_Q(d)^2h^{2\alpha_Q}\int(1+\norm{x})^2\,\mu_f(\dd s,\dd x).
    \end{equation*}
    By $(1+\norm{x})^2\leq2(1+\norm{x}^2)$ and \ref{F1} with $q=2$,
    \begin{align*}
    \int(1+\norm{x})^2\,\mu_f(\dd s,\dd x)
    &\leq2\int(1+\norm{x}^2)\,\mu_f(\dd s,\dd x)
    \leq2F_{2,d}<\infty.
    \end{align*}

    \begin{lemma}
    \label{F3}
    Assume \ref{A1}--\ref{A3}, \ref{C2}--\ref{C5}, and \ref{F1}. Then there is
    a constant $C_V(d,T)<\infty$, independent of $h$, $s$, and $x$, such that
    \begin{equation}
    \label{eq:F3_pointwise_bound}
    \sup_{s\leq t\leq T}
    \E_\theta\norm{\widetilde X_t^{h;\theta}}^2
    \leq C_V(d,T)(1+\norm{x}^2).
    \end{equation}
    In particular, there is a constant $V_f<\infty$, independent of $h$, such that
    \begin{equation}
    \label{eq:F3_bound}
    \int \E_\theta\norm{\widetilde X_r^{h;\theta}}^2\,\mu_f(\dd\theta)
    \leq V_f.
    \end{equation}

    \begin{proof}
    Fix $\theta=(s,x)$ and abbreviate
    \[
    \Delta_k:=\Delta\tau_k(s),\qquad
    \bar{X}_k:=\bar X_{\tau_k(s)}^{h;\theta},\qquad
    a_k:=a(\tau_k(s),\bar{X}_k).
    \]
    We put $\mathcal F_k:= \Sigma_{\tau_k(s)}$ and set
    \[
    J_k:=\sum_{n=1}^{N_k^s}c(\tau_k(s),\bar{X}_k,Y_{k,n}^s)
    -\Delta_k I_k^s(\bar{X}_k).
    \]
    Then
    \begin{equation*}
    \bar{X}_{k+1}=\bar{X}_k+\Delta_k a_k+\Delta W_k^s+J_k.
    \end{equation*}
    Conditional on $\mathcal F_k$, the increments $\Delta W_k^s$ and $J_k$
    are independent, have zero conditional mean, and satisfy
    \begin{align*}
    \E_\theta[J_k\mid\mathcal F_k]&=0,\\
    \E_\theta[\norm{J_k}^2\mid\mathcal F_k]
    &=\Delta_k\int_{\Z}\norm{c(\tau_k(s),\bar{X}_k,z)}^2\,\nu(\dd z).
    \end{align*}
    Thus, after conditioning, all mixed terms involving the Brownian or
    compensated-jump increment vanish, and
    \begin{align*}
    \E_\theta[\norm{X_{k+1}}^2\mid\mathcal F_k]
    &=\norm{\bar{X}_k+\Delta_k a_k}^2+d\Delta_k\\
    &\quad+\Delta_k\int_{\Z}\norm{c(\tau_k(s),\bar{X}_k,z)}^2\,\nu(\dd z).
    \end{align*}
    The inequality
    $2\langle y,z\rangle\leq\varepsilon\norm{y}^2+
    \norm{z}^2/\varepsilon$
    with $\varepsilon=\Delta_k$ gives
    \begin{align*}
    \norm{\bar{X}_k+\Delta_k a_k}^2
    &\leq (1+\Delta_k)\norm{\bar{X}_k}^2
    +\Delta_k(1+\Delta_k)\norm{a_k}^2.
    \end{align*}
    By \ref{A2}, $\norm{a_k}^2\leq2K_a(d)^2(1+\norm{\bar{X}_k}^2)$, while
    \eqref{eq:euler_jump_growth_second_moment} gives
    \begin{equation*}
    \int_{\Z}\norm{c(\tau_k(s),\bar{X}_k,z)}^2\,\nu(\dd z)
    \leq2\bigl[\widehat K_\jump(d)^2+L_\jump(d)^2\norm{\bar{X}_k}^2\bigr].
    \end{equation*}
    Since $\Delta_k\leq T$, there is a constant $C_0(d,T)$, independent of
    $h$, $s$, $k$, and $x$, such that
    \begin{equation*}
    \E_\theta[\norm{X_{k+1}}^2\mid\mathcal F_k]
    \leq(1+C_0(d,T)\Delta_k)\norm{\bar{X}_k}^2+C_0(d,T)\Delta_k.
    \end{equation*}
    Therefore, with $y_k:=\E_\theta\norm{\bar{X}_k}^2$ and $y_0=\norm{x}^2$,
    \begin{align}
        y_{k+1} + 1&\leq(1+C_0(d,T)\Delta_k)(y_k + 1)\nonumber\\
        \label{eq:F3_grid_recursion}
        &\leq(1+\norm{x}^2)\prod_{j<k}\left(1+C_0(d,T)\Delta_j\right).
    \end{align}
    Since $\sum_k\Delta_k\leq T$, discrete Grönwall gives
    \begin{equation}
    \label{eq:F3_grid_bound}
    \max_{k\leq K(s)}\E_\theta\norm{\bar X_{\tau_k(s)}^{h;\theta}}^2
    \leq C_1(d,T)(1+\norm{x}^2),
    \end{equation}
    with $C_1(d,T)=e^{C_0(d,T)T}$ independent of $h$ and $s$.

    It remains to pass from grid values to the continuous interpolation. For
    $t\in[\tau_k(s),\tau_{k+1}(s)]$, put $w=t-\tau_k(s)$. Then $0\leq w\leq h$
    and
    \begin{align*}
    \widetilde X_t^{h;\theta}
    &=\bar{X}_k+w a_k+W_t-W_{\tau_k(s)}\\
    &\quad+\int_{\tau_k(s)}^t\int_{\Z}
    c(\tau_k(s),\bar{X}_k,z)\,\dismu(\dd r,\dd z).
    \end{align*}
    Applying the same conditional calculation on an interval of length $w$,
    and then using \eqref{eq:F3_grid_bound}, yields
    \begin{align*}
    \E_\theta\norm{\widetilde X_t^{h;\theta}}^2
    &\leq(1+C_0(d,T)w)\E_\theta\norm{\bar{X}_k}^2+C_0(d,T)w\\
    &\leq C_V(d,T)(1+\norm{x}^2).
    \end{align*}
    Hence this estimate is uniform in $h$, $s$, and $t$. In particular, at
    $t=r=T$, F1 gives
    \begin{align*}
    \int\E_\theta\norm{\widetilde X_r^{h;\theta}}^2\,\mu_f(\dd\theta)
    &\leq C_V(d,T)\int(1+\norm{x}^2)\,\mu_f(\dd\theta)\\
    &\leq C_V(d,T)a_2\bigl(1+\log(1+d^{\gamma_{f,2}})\bigr)=:V_f<\infty.
    \end{align*}
    The constant $V_f$ is independent of $h$, proving
    \eqref{eq:F3_bound}.
    This coarse conditional argument asserts finiteness, not polynomial
    dimension dependence. Under \ref{C1}, the sharper bound
    \eqref{eq:euler_interpolation_moment} instead permits
    $V_f=M_dF_{2,d}$, which is polynomially bounded in $d$.
    \end{proof}
    \end{lemma}

    \begin{lemma}
    \label{F4}
        Assume \ref{C1}--\ref{C5}, \ref{F1}, and the polynomial upper bound on
        $\lambda(d)^{-1}$ stated in Section~2. Let $Z$ denote a random variable
        with law $\rho_d$. Then there is a constant $\Sigma_z<\infty$, independent
        of $h$, such that
    \begin{equation}
    \label{eq:F4_bound}
    \int_0^T\int_{\R^d}
    \sup_{t\in[0,T]}\operatorname{Var}_{\rho_d}(c(t,x,Z))f_d(s,x)\,\dd x\,\dd s\le\Sigma_z^2
    \end{equation}
        Moreover, there are constants $C,\kappa_z>0$, independent of $d$, such that
        \begin{equation}
        \label{eq:F4_dimension_bound}
        \Sigma_z\leq C(1+d)^{\kappa_z}.
        \end{equation}
    \end{lemma}

    \begin{proof}
    By \eqref{jump_int_norm} with $q=2$ and $\nu_d=\lambda(d)\rho_d$, set
    \[
    C_c(d):=\frac{\max\{L_\jump(d),K_\jump(d)+T^{\beta_\jump}L_\jump(d)\}}
    {\sqrt{\lambda(d)}}.
    \]
    Then
    \begin{equation*}
    \E_{\rho_d}\norm{c(t,x,Z)}^2
    \le C_c(d)^2(1+\norm{x})^2.
    \end{equation*}
    By \ref{C1} and the polynomial upper bound on $\lambda(d)^{-1}$,
    $C_c(d)$ grows at most polynomially in $d$.
    Since the variance of a vector is bounded by its second moment,
    \begin{align*}
    \operatorname{Var}_{\rho_d}(c(t,x,Z))
    &:=\E_{\rho_d}\norm{c(t,x,Z)-\E_{\rho_d}c(t,x,Z)}^2\\
    &\le \E_{\rho_d}\norm{c(t,x,Z)}^2\\
    &\le C_c(d)^2(1+\norm{x})^2.
    \end{align*}
    For each $x$, \ref{C5} implies continuity of
    $t\mapsto c(t,x,\cdot)$ in $L^2(\rho_d)$, hence continuity of this
    variance. Its supremum over $[0,T]$ equals the supremum over a fixed
    countable dense subset. Since $c$ is measurable, the latter is
    a measurable function of $x$. Taking that supremum and integrating gives
    \begin{align*}
    \int_0^T\int_{\R^d}
    &\sup_{t\in[0,T]}\operatorname{Var}_{\rho_d}(c(t,x,Z))
    f_d(s,x)\,\dd x\,\dd s\\
    &\le C_c(d)^2
    \int_0^T\int_{\R^d}(1+\norm{x})^2f_d(s,x)\,\dd x\,\dd s.
    \end{align*}
    The second-moment bound in \ref{F1} bounds the last integral by
    $2F_{2,d}$. Thus choose $\Sigma_z^2:=2C_c(d)^2F_{2,d}$, finite and
    polynomially bounded in $d$. This proves
    \eqref{eq:F4_bound} and the asserted dimension dependence.
    \end{proof}

\subsection{Approximation of $u$ in weighted $L^2$ norm}
\label{subsec:u_mc_weighted_l2}
Under \ref{U1}--\ref{U4}, the coefficient conditions
\ref{A2}--\ref{A3}, \ref{C2}--\ref{C5}, and
\ref{B2}, the Feynman--Kac identity is
\begin{equation}
\label{eq:u_decompositin}
u(s,x)=\E\left[g(X_T^{s,x})\right]
 +\E\left[\int_s^T b(v,X_v^{s,x})\,\dd v\right],
\qquad (s,x)\in[0,T]\times\R^d.
\end{equation}
This is the Feynman--Kac representation discussed in the preliminaries
under \ref{U1}--\ref{U4}.

The three Euler constructions use the same branch functional. For
$r\in\{0,\mathrm{fresh},\mathrm{emp}\}$, write
\[
\mathcal X_v^{r,m;s,x}:=
\begin{cases}
\bar X_v^{h,m;s,x},&r=0,\\
\bar X_v^{h,\mathrm{fresh},m;s,x},&r=\mathrm{fresh},\\
\bar X_v^{h,\mathrm{emp},m;s,x},&r=\mathrm{emp},
\end{cases}
\]
at shifted-grid times $v$, where $m$ indexes the trajectory branch.
The exact-compensator Euler process retains its undecorated bar notation;
the two sampled-compensator processes are defined in the constructions below.
Motivated by \eqref{eq:u_decompositin}, define
\begin{equation}
\label{eq:mc_branch_payoff}
G_m^r(s,x)=g\big(\mathcal X_T^{r,m;s,x}\big)
 +\sum_{k=0}^{K(s)-1}\Delta\tau_k(s)
 b\big(\tau_k(s),\mathcal X_{\tau_k(s)}^{r,m;s,x}\big).
\end{equation}
The sum is empty for $s=T$. When the construction fixes the scheme, we
write $G_m$ without a superscript for that choice of $G_m^r$.

\begingroup
\renewenvironment{proof}{\par\noindent{\bfseries Proof\ }}
	{\hfill\BlackBox\par\vspace{2mm}}

\paragraph{Euler scheme with fresh samples for the compensator.}
Use the shifted grid defined above, with
$h=T/N_{\mathrm{Euler}}\leq H:=\min\{1,T\}$, and let $M,M_z$ be positive
integers. On a common probability space take independent copies
$(W^{(m)},\mu^{(m)})$ of the trajectory noises, and batches
\[
Z_{k,m}^{\mathrm{comp}}=(z_{k,m}^1,\ldots,z_{k,m}^{M_z})
\sim\rho_d^{\otimes M_z},
\]
where $0\leq k < N_{\mathrm{Euler}}$ and $1\leq m \leq M$.
The batches are mutually independent across steps $k$ and branches $m$,
and independent of all trajectory noises. For each $(s,x)$, extract
$\Delta W_k^{s,(m)}$, $N_k^{s,(m)}$, and $Y_{k,n}^{s,(m)}$ from these
noises on $(\tau_k(s),\tau_{k+1}(s)]$, as in the exact-compensator
construction. In particular, the Brownian covariance and Poisson mean
are $\Delta\tau_k(s)I_d$ and $\lambda(d)\Delta\tau_k(s)$, respectively.
The trajectory jump marks are denoted collectively by $Z_m^{\mathrm{jump}}$.
Couple the exact SDE and both Euler schemes synchronously within each
branch, using the same trajectory noises.

For $0\leq k<K(s)$ define
\[
\widehat I_{k,m}^{s,M_z,\mathrm{fresh}}(y)
=\frac{\lambda(d)}{M_z}\sum_{\ell=1}^{M_z}
c(\tau_k(s),y,z_{k,m}^{\ell}).
\]
Write $h_k=\Delta\tau_k(s)$ and, for fixed $s,x,m$, use the abbreviations
\[
\bar X_k^f:=\bar X_{\tau_k(s)}^{h,\mathrm{fresh},m;s,x},
\qquad \bar X_k^0:=\bar X_{\tau_k(s)}^{h,m;s,x}.
\]
The superscripts $f$ and $0$ distinguish the fresh-sample and
exact-compensator grid processes; neither denotes the exact SDE.
Starting from $\bar X_0^f=x$, set
\begin{equation}
\label{eq:euler_with_fresh_Mz}
\begin{aligned}
\bar X_{k+1}^f={}&\bar X_k^f+h_k a(\tau_k(s),\bar X_k^f)
 +\Delta W_k^{s,(m)}\\
&+\sum_{n=1}^{N_k^{s,(m)}}c(\tau_k(s),\bar X_k^f,Y_{k,n}^{s,(m)})
 -h_k\widehat I_{k,m}^{s,M_z,\mathrm{fresh}}(\bar X_k^f).
\end{aligned}
\end{equation}
Define $\overline X_v^{h,\mathrm{fresh},m;s,x}=\bar X_k^f$ on
$[\tau_k(s),\tau_{k+1}(s))$, and set its value at $T$ to
$\bar X_{K(s)}^f$. For $s=T$, there are no updates and the process equals $x$.
The branch index is suppressed in estimates for a single branch.

For fixed deterministic $s$, let $\mathcal G_k^s$ be the completed
sigma-field generated by all trajectory noises up to $\tau_k(s)$ and by
all batches with step index $i<k$. Both grid states are
$\mathcal G_k^s$-measurable. The current batch is independent of this
sigma-field and of the current Brownian and marked Poisson increments.
For
\[
\xi_{k,m}:=\widehat I_{k,m}^{s,M_z,\mathrm{fresh}}(\bar X_k^f)
-I_k^s(\bar X_k^f),
\]
conditional independence gives
\begin{equation}
\label{eq:fresh_compensator_variance}
\begin{aligned}
\E[\xi_{k,m}\mid\mathcal G_k^s]&=0,\\
\E[\norm{\xi_{k,m}}^2\mid\mathcal G_k^s]
&=\frac{1}{M_z}\left[
\lambda(d)\int_{\Z}\norm{c(\tau_k(s),\bar X_k^f,z)}^2\nu(\dd z)
-\norm{I_k^s(\bar X_k^f)}^2\right]\\
&\leq\frac{\lambda(d)}{M_z}
\int_{\Z}\norm{c(\tau_k(s),\bar X_k^f,z)}^2\nu(\dd z).
\end{aligned}
\end{equation}
Thus the fresh estimator is conditionally unbiased at the random incoming
state. Independence across branches alone would not give this property
if each branch reused its batch over time.

Use the same noise paths and indexed batch array for all inputs $(s,x)$.
The finite recursion, Borel measurability of the coefficients, and
measurability of Brownian path evaluation and finite marked Poisson sums
then give jointly measurable grid states and branch functionals in
$(\omega,s,x)$. No independence between different inputs is required.
With the branch functional from \eqref{eq:mc_branch_payoff}, define
\[
\mathrm{MC}^{\mathrm{fresh}}_{h,M,M_z}(s,x)
:=\frac1M\sum_{m=1}^M G_m^{\mathrm{fresh}}(s,x).
\]
Under \ref{B4} and \ref{G4}, every branch functional is bounded in absolute
value by $M_g(d)+TM_b(d)$. Hence this estimator belongs to
$L^2(\mu_f)$ almost surely.

\begin{lemma}
\label{lem:fresh_euler_moments}
Assume \ref{A1}--\ref{A3} and \ref{C2}--\ref{C5}.
For the construction \eqref{eq:euler_with_fresh_Mz}, there is a finite
constant $C_{\mathrm{fresh}}(d)$, independent of $h,M_z,s,x$, such that
\begin{equation}
\label{eq:fresh_euler_moments}
\E\left[\sup_{s\leq v\leq T}
\norm{\overline X_v^{h,\mathrm{fresh};s,x}}^2\right]
\leq C_{\mathrm{fresh}}(d)(1+\norm{x}^2).
\end{equation}
If \ref{C1} and the stated intensity bound are imposed as well, this
constant can be chosen to grow at most polynomially in $d$.
\end{lemma}

\begin{proof}
Fix $s<T,x$ and a branch; suppress $m$ and write $\tau_k=\tau_k(s)$.
Use the previously defined
$\widehat K_\jump(d)=K_\jump(d)+T^{\beta_\jump}L_\jump(d)$.
For this proof, suppress $d$ in coefficient constants. By
\ref{A2}--\ref{A3} and \ref{C3}--\ref{C5},
\begin{align*}
\norm{a(t,y)}^2&\leq2K_a^2+2L_a^2\norm{y}^2,\\
\int_{\Z}\norm{c(t,y,z)}^2\nu(\dd z)
&\leq2\widehat K_\jump^2+2L_\jump^2\norm{y}^2.
\end{align*}
In particular, the time regularity assumption \ref{C5} controls the
coefficient at $(t,0)$ uniformly over $t$.
Set
\begin{align*}
\eta_k&=\sum_{n=1}^{N_k^{s,(m)}}c(\tau_k,\bar X_k^f,Y_{k,n}^{s,(m)})
-h_k I_k^s(\bar X_k^f),\\
U_k&=\Delta W_k^{s,(m)}+\eta_k-h_k\xi_k,
\qquad a_k=a(\tau_k,\bar X_k^f).
\end{align*}
Conditional on $\mathcal G_k^s$, the three terms in $U_k$ are independent
and centred. The Poisson isometry and
\eqref{eq:fresh_compensator_variance} give
\begin{align*}
\E[\norm{U_k}^2\mid\mathcal G_k^s]
&=d h_k+h_k\int_{\Z}\norm{c(\tau_k,\bar X_k^f,z)}^2\nu(\dd z)
+h_k^2\E[\norm{\xi_k}^2\mid\mathcal G_k^s],\\
\E[\norm{\bar X_{k+1}^f}^2\mid\mathcal G_k^s]
&=\norm{\bar X_k^f+h_k a_k}^2+
\E[\norm{U_k}^2\mid\mathcal G_k^s].
\end{align*}
These identities and the growth bounds also prove square integrability
by induction over the finite grid. Applying
$2\langle y,a(t,y)\rangle\leq\norm{y}^2+\norm{a(t,y)}^2$ yields,
with $q_k:=\E\norm{\bar X_k^f}^2$ and $M_z\geq1$,
\[
q_{k+1}\leq(1+b_{\mathrm{fr}}h_k+c_{\mathrm{fr}}h_k^2)q_k
+\gamma_{\mathrm{fr}}h_k,
\]
where
\begin{align*}
b_{\mathrm{fr}}&=1+2L_a^2+2L_\jump^2,
&c_{\mathrm{fr}}&=2L_a^2+2\lambda L_\jump^2,\\
\gamma_{\mathrm{fr}}
&=d+2(1+H)K_a^2+2(1+H\lambda)\widehat K_\jump^2.
\end{align*}
Put $r_{\mathrm{fr}}=b_{\mathrm{fr}}+\sqrt{2c_{\mathrm{fr}}}$.
For $v\geq0$, the first three terms of the exponential series show that
$1+b_{\mathrm{fr}}v+c_{\mathrm{fr}}v^2\leq e^{r_{\mathrm{fr}}v}$.
Iteration and $\sum_kh_k=T-s\leq T$ imply
\begin{equation}
\label{eq:fresh_grid_moments}
\max_{0\leq k\leq K(s)}q_k
\leq B_{\mathrm{fr}}(d)(1+\norm{x}^2),
\qquad B_{\mathrm{fr}}(d):=(1+\gamma_{\mathrm{fr}}T)e^{r_{\mathrm{fr}}T}.
\end{equation}

To control the expected maximum, define
$\mathcal M_k=\sum_{j<k}U_j$. This is a square-integrable discrete
martingale with respect to $(\mathcal G_k^s)_k$.
Define the growth constants
\[
\Gamma_a(d)=2K_a(d)^2,
\qquad
\Gamma_\jump(d)=2\max\{\widehat K_\jump(d),L_\jump(d)\}^2.
\]
The first is $K_{a,2}(d)^2$ from \eqref{eq:drift_q_growth}.
Cauchy--Schwarz, discrete Doob's inequality, and orthogonality of the
martingale differences give
\begin{align*}
\E\max_k\norm{\bar X_k^f}^2
&\leq3\norm{x}^2+3T\sum_jh_j\E\norm{a_j}^2
+12\sum_j\E\norm{U_j}^2\\
&\leq3\norm{x}^2+3T^2\Gamma_a(1+B_{\mathrm{fr}})(1+\norm{x}^2)\\
&\quad+12\left[dT+(1+H\lambda)\Gamma_\jump T
(1+B_{\mathrm{fr}})\right](1+\norm{x}^2).
\end{align*}
Thus one may take
\begin{align*}
C_{\mathrm{fresh}}(d)
={}&3+3T^2\Gamma_a(d)(1+B_{\mathrm{fr}}(d))\\
&+12\left[dT+(1+H\lambda(d))\Gamma_\jump(d)T
(1+B_{\mathrm{fr}}(d))\right].
\end{align*}
The supremum of the interpolation equals the maximum of all grid values,
including the terminal value, proving \eqref{eq:fresh_euler_moments}.
The case $s=T$ is immediate. Under the dimension bounds,
$r_{\mathrm{fr}}=O(1+\log(1+d))$ and $\gamma_{\mathrm{fr}}$ is
polynomially bounded. The displayed constants therefore have at most
polynomial growth in $d$ for fixed $T$.
\end{proof}

\begin{lemma}
\label{lem:fresh_compensator_stability}
Under the assumptions of Lemma~\ref{lem:fresh_euler_moments}, synchronously
couple the fresh and exact-compensator Euler schemes. There is a finite
constant $S_{\mathrm{fr}}(d)$, independent of $h,M_z,s,x$, such that
\begin{equation}
\label{eq:fresh_state_sup_stability}
\E\left[\sup_{s\leq v\leq T}
\norm{\overline X_v^{h,\mathrm{fresh};s,x}-\overline X_v^{h;s,x}}^2\right]
\leq S_{\mathrm{fr}}(d)\frac{h}{M_z}(1+\norm{x}^2).
\end{equation}
If also \ref{B3} and \ref{G3} hold, there is a constant
$P_{\mathrm{fr}}(d)$, with the same independence properties, such that
\begin{equation}
\label{eq:fresh_functional_stability}
\E|G^{\mathrm{fresh}}(s,x)-G^0(s,x)|^2
\leq P_{\mathrm{fr}}(d)\frac{h}{M_z}(1+\norm{x}^2).
\end{equation}
\end{lemma}

\begin{proof}
Fix a branch and suppress its index. Write $D_k=\bar X_k^f-\bar X_k^0$,
$\delta a_k=a(\tau_k,\bar X_k^f)-a(\tau_k,\bar X_k^0)$, and
$\delta c_k(z)=c(\tau_k,\bar X_k^f,z)-c(\tau_k,\bar X_k^0,z)$, for $\tau_k = \tau_k(s)$.
Set
\[
J_k^D=\sum_{n=1}^{N_k^{s,(m)}}\delta c_k(Y_{k,n}^{s,(m)})
-h_k\int_{\Z}\delta c_k(z)\nu(\dd z).
\]
Subtracting the recursions gives $D_0=0$ and
$D_{k+1}=D_k+h_k\delta a_k+J_k^D-h_k\xi_k$.
Both noise terms are conditionally centred and conditionally independent.
Consequently,
\begin{align*}
\E[\norm{D_{k+1}}^2\mid\mathcal G_k^s]
&=\norm{D_k+h_k\delta a_k}^2
+h_k\int_{\Z}\norm{\delta c_k(z)}^2\nu(\dd z)
+h_k^2\E[\norm{\xi_k}^2\mid\mathcal G_k^s]\\
&\leq(1+\beta_{\mathrm{fr}}(d)h_k)\norm{D_k}^2
+h_k^2\E[\norm{\xi_k}^2\mid\mathcal G_k^s],
\end{align*}
where $\beta_{\mathrm{fr}}(d)=2L_a(d)+HL_a(d)^2+L_\jump(d)^2$.
By \eqref{eq:fresh_compensator_variance} and \eqref{eq:fresh_grid_moments},
\[
\E\norm{\xi_k}^2\leq\frac{Q_{\mathrm{fr}}(d)}{M_z}(1+\norm{x}^2),
\qquad Q_{\mathrm{fr}}(d):=\lambda(d)\Gamma_\jump(d)(1+B_{\mathrm{fr}}(d)).
\]
Discrete Gr\"onwall and $\sum_kh_k^2\leq Th$ now give
\[
\max_k\E\norm{D_k}^2
\leq R_{\mathrm{fr}}(d)\frac h{M_z}(1+\norm{x}^2),
\qquad R_{\mathrm{fr}}(d):=TQ_{\mathrm{fr}}(d)e^{\beta_{\mathrm{fr}}(d)T}.
\]
To bound the expectation of the maximum, use the discrete martingale
$\mathcal M_k^D=\sum_{j<k}(J_j^D-h_j\xi_j)$, not merely the preceding
bound on individual expectations. Since
$D_k=\sum_{j<k}h_j\delta a_j+\mathcal M_k^D$, Cauchy--Schwarz and
Doob's inequality imply
\begin{align*}
\E\max_k\norm{D_k}^2
&\leq2TL_a(d)^2\sum_jh_j\E\norm{D_j}^2
+8\sum_j\E\norm{J_j^D-h_j\xi_j}^2\\
&\leq\bigl(2TL_a(d)^2+8L_\jump(d)^2\bigr)
\sum_jh_j\E\norm{D_j}^2
+8\sum_jh_j^2\E\norm{\xi_j}^2.
\end{align*}
This proves \eqref{eq:fresh_state_sup_stability} with
\[
S_{\mathrm{fr}}(d)=\bigl(2TL_a(d)^2+8L_\jump(d)^2\bigr)TR_{\mathrm{fr}}(d)
+8TQ_{\mathrm{fr}}(d).
\]
Finally, the definition \eqref{eq:mc_branch_payoff} and the spatial
Lipschitz bounds give
\[
|G^{\mathrm{fresh}}-G^0|^2
\leq2L_g(d)^2\norm{D_{K(s)}}^2
+2TL_b(d)^2\sum_kh_k\norm{D_k}^2.
\]
The bound on individual grid expectations proves
\eqref{eq:fresh_functional_stability} with
\[
P_{\mathrm{fr}}(d)=2(L_g(d)^2+T^2L_b(d)^2)R_{\mathrm{fr}}(d).
\]
\end{proof}

\begin{theorem}
\label{thm:mc_weighted_l2_fresh}
\leavevmode\par\noindent
Assume \ref{U1}--\ref{U4}, \ref{F1}, \ref{A1}--\ref{A3},
\ref{B1}--\ref{B4}, \ref{C1}--\ref{C5}, and \ref{G1}--\ref{G4},
together with the stated logarithmic upper bound on $\lambda(d)$.
For the sample construction above, there are positive constants
$C_u^{\mathrm{fresh}}(d)$, $C_M^{\mathrm{fresh}}(d)$, and $C_z^{\mathrm{fresh}}(d)$,
independent of $h,M,M_z,s,x$, such that
\begin{equation}
\label{eq:mc_weighted_l2_fresh_mean}
\E\|u-\mathrm{MC}^{\mathrm{fresh}}_{h,M,M_z}\|_{L^2(\mu_f)}^2
\leq C_u^{\mathrm{fresh}}(d)h^{2\alpha_Q}
+\frac{C_M^{\mathrm{fresh}}(d)}{M}
+C_z^{\mathrm{fresh}}(d)\frac{h}{M_z}.
\end{equation}
Their dimension dependence is at most polynomial: there exist constants
$C_{\mathrm{dim}},\kappa_{\mathrm{dim}}>0$, independent of
$d,h,M,M_z,s,x$, such that
\[
C_u^{\mathrm{fresh}}(d)+C_M^{\mathrm{fresh}}(d)
+C_z^{\mathrm{fresh}}(d)
\leq C_{\mathrm{dim}}(1+d)^{\kappa_{\mathrm{dim}}}.
\]
The constants $C_{\mathrm{dim}}$ and $\kappa_{\mathrm{dim}}$ may depend on
the fixed horizon $T$ and the dimension-independent constants and
regularity exponents in the assumptions.
Here $\alpha_Q=\min\{\alpha_{\mathrm E},\alpha_b\}$ is the exponent
from Lemma~\ref{lem:running_source_quadrature}. For
$\epsilon,\delta\in(0,1)$, choose admissible $h=T/N_{\mathrm{Euler}}$ and
positive integers $M,M_z$ satisfying
\begin{align*}
h&\leq\left(\frac{\epsilon^2\delta}{3C_u^{\mathrm{fresh}}(d)}
\right)^{1/(2\alpha_Q)},\\
M&\geq\left\lceil\frac{3C_M^{\mathrm{fresh}}(d)}{\epsilon^2\delta}\right\rceil,
\qquad
M_z\geq\max\left\{1,
\left\lceil\frac{3C_z^{\mathrm{fresh}}(d)h}{\epsilon^2\delta}\right\rceil\right\}.
\end{align*}
Then
\[
\mathbb P\left(
\|u-\mathrm{MC}^{\mathrm{fresh}}_{h,M,M_z}\|_{L^2(\mu_f)}<\epsilon
\right)\geq1-\delta.
\]
\end{theorem}

\begin{proof}
Define the exact trajectory functional
\[
\mathcal H(s,x)=g(X_T^{s,x})+\int_s^T b(v,X_v^{s,x})\,\dd v.
\]
By \eqref{eq:u_decompositin}, $u(s,x)=\E\mathcal H(s,x)$.

Compare $\mathcal H$ with the synchronously coupled exact-compensator
Euler functional $G^0$. Jensen's inequality gives
\[
|u(s,x)-\E G^0(s,x)|^2\leq\E|\mathcal H(s,x)-G^0(s,x)|^2.
\]
The fixed-dimension strong estimate proved in
Proposition~\ref{prop:dimension_explicit_euler_constant}
(which does not require \ref{C1} for finiteness of $A_d$),
Lemma~\ref{lem:running_source_quadrature}, and \ref{G3} imply
\begin{align*}
\E|\mathcal H-G^0|^2
&\leq2\left[L_g(d)^2A_d h^{2\alpha_{\mathrm E}}
+C_Q(d)^2h^{2\alpha_Q}\right](1+\norm{x})^2\\
&\leq4\left[L_g(d)^2A_d+C_Q(d)^2\right]
h^{2\alpha_Q}(1+\norm{x}^2).
\end{align*}
Thus, using $F_{2,d}$ from \ref{F1} as defined above,
\begin{align*}
\int|u-\E G^0|^2\,\dd\mu_f
&\leq A_{\mathrm{fr},0}(d)h^{2\alpha_Q},\\
A_{\mathrm{fr},0}(d)&:=4F_{2,d}
\left[L_g(d)^2A_d+C_Q(d)^2\right].
\end{align*}
By \eqref{eq:fresh_functional_stability}, Jensen's inequality and
\ref{F1} also give
\[
\int|\E G^{\mathrm{fresh}}-\E G^0|^2\,\dd\mu_f
\leq F_{2,d}P_{\mathrm{fr}}(d)\frac h{M_z}.
\]
For each fixed $(s,x)$ the branch functionals are iid. Their boundedness
under \ref{B4} and \ref{G4} yields
\[
\operatorname{Var}\!\left(\mathrm{MC}^{\mathrm{fresh}}_{h,M,M_z}(s,x)\right)
=\frac{\operatorname{Var}(G^{\mathrm{fresh}}(s,x))}{M}
\leq\frac{(M_g(d)+TM_b(d))^2}{M}.
\]
The bias--variance identity, Tonelli's theorem, and
$|a+b|^2\leq2|a|^2+2|b|^2$ prove
\eqref{eq:mc_weighted_l2_fresh_mean}. For example, positive constants
are obtained by taking
\begin{align*}
C_u^{\mathrm{fresh}}(d)&=1+2A_{\mathrm{fr},0}(d),\\
C_M^{\mathrm{fresh}}(d)&=1+(M_g(d)+TM_b(d))^2,\\
C_z^{\mathrm{fresh}}(d)&=1+2F_{2,d}P_{\mathrm{fr}}(d).
\end{align*}

We now verify the polynomial dimension bound explicitly. Put
$\ell_d:=1+\log(1+d)$. In the estimates below, $C_T>0$ denotes a
constant independent of $d,h,M,M_z,s,x$, depending only on $T$ and
the dimension-independent assumption constants; its value may increase
from one estimate to the next. By \ref{A1}, \ref{C1}, and the intensity
upper bound,
\begin{align*}
K_a(d)+\widehat K_\jump(d)+\lambda(d)&\leq C_T\ell_d,\\
L_a(d)^2+L_\jump(d)^2&\leq C_T\ell_d.
\end{align*}
Consequently, the constants in the fresh moment recurrence satisfy
\begin{align*}
b_{\mathrm{fr}}(d)&\leq C_T\ell_d,
&c_{\mathrm{fr}}(d)&\leq C_T\ell_d^2,\\
\gamma_{\mathrm{fr}}(d)&\leq C_T(d+\ell_d^3),
&r_{\mathrm{fr}}(d)&\leq C_T\ell_d.
\end{align*}
In particular, the exponential in $B_{\mathrm{fr}}$ is polynomially
bounded, because
\[
e^{r_{\mathrm{fr}}(d)T}
\leq e^{C_TT\ell_d}
=e^{C_TT}(1+d)^{C_TT}.
\]
Thus there are constants $C_T,\kappa_B>0$, independent of $d$, such that
$B_{\mathrm{fr}}(d)\leq C_T(1+d)^{\kappa_B}$.
Furthermore,
\begin{align*}
\Gamma_\jump(d)&\leq C_T\ell_d^2,\\
\beta_{\mathrm{fr}}(d)
&=2L_a(d)+HL_a(d)^2+L_\jump(d)^2
\leq C_T\ell_d,\\
Q_{\mathrm{fr}}(d)
&\leq C_T\ell_d^3(1+B_{\mathrm{fr}}(d)).
\end{align*}
The same exponential estimate therefore makes
$R_{\mathrm{fr}}(d)=TQ_{\mathrm{fr}}(d)e^{\beta_{\mathrm{fr}}(d)T}$
polynomially bounded. By \ref{B3} and \ref{G1},
\[
P_{\mathrm{fr}}(d)
=2\bigl(L_g(d)^2+T^2L_b(d)^2\bigr)R_{\mathrm{fr}}(d)
\leq C_T\ell_d^2R_{\mathrm{fr}}(d),
\]
so $P_{\mathrm{fr}}$ is polynomially bounded as well.
Proposition~\ref{prop:dimension_explicit_euler_constant} gives a
polynomial bound for $A_d$, and the dimension estimate in the proof of
Lemma~\ref{lem:running_source_quadrature} gives one for $C_Q(d)$.
Finally, \ref{F1}, \ref{B4}, and \ref{G1} imply
\[
F_{2,d}+M_b(d)+M_g(d)+L_g(d)\leq C_T\ell_d.
\]
Substituting these bounds into the displayed formulas for
$A_{\mathrm{fr},0}$ and the three error constants proves their claimed
polynomial bound, since finite sums and products of polynomially bounded
quantities are polynomially bounded. No bound on $\lambda(d)^{-1}$ is
needed for this argument.

The stated parameter choices make the right-hand side of
\eqref{eq:mc_weighted_l2_fresh_mean} at most $\epsilon^2\delta$.
Markov's inequality applied to the squared weighted error therefore gives
$\mathbb P(\|u-\mathrm{MC}^{\mathrm{fresh}}_{h,M,M_z}\|_{L^2(\mu_f)}
\geq\epsilon)\leq\delta$.
\end{proof}

Since $h\leq1$, Theorem~\ref{thm:mc_weighted_l2_fresh} also implies the
weaker bound with $C_z^{\mathrm{fresh}}(d)/M_z$ in place of
$C_z^{\mathrm{fresh}}(d)h/M_z$, with the same polynomially bounded
constants.
\endgroup

\begingroup
\allowdisplaybreaks[1]
\renewenvironment{proof}{\par\noindent{\bfseries Proof\ }}%
    {\unskip\nobreak\hfill\BlackBox\par\vspace{2mm}}

\paragraph{Euler scheme with reusable samples for the compensator.}
We now introduce the common-reusable-sample approximation. Sample once and
reuse the compensator marks for all trajectories and all Euler steps:
\[
    Z^{\mathrm{comp}}=(z^1,\dots,z^{M_z})
    \overset{i.i.d.}{\sim}\rho_d^{\otimes M_z}.
\]
For each trajectory index $m$ and shifted time step $k$, use the same
trajectory noises as above and set
\[
    \widehat I_{k}^{s,M_z}(x)=\frac{\lambda(d)}{M_z}
    \sum_{\ell=1}^{M_z}c(\tau_k(s),x,z^\ell).
\]
Starting from $\bar X_s^{h,\mathrm{emp},m;s,x}=x$, the corresponding
empirical Euler recursion is
\begin{equation}
\label{eq:euler_with_shared_Mz}
\begin{aligned}
\bar X_{\tau_{k+1}(s)}^{h,\mathrm{emp},m;s,x}
&=\bar X_{\tau_k(s)}^{h,\mathrm{emp},m;s,x}
 +\Delta\tau_k(s)a(\tau_k(s),\bar X_{\tau_k(s)}^{h,\mathrm{emp},m;s,x})
 +\Delta W_k^{s,(m)}\\
&\quad+\sum_{n=1}^{N_k^{s,(m)}}c\bigl(\tau_k(s),\bar X_{\tau_k(s)}^{h,\mathrm{emp},m;s,x},Y_{k,n}^{s,(m)}\bigr)\\
&\quad-\Delta\tau_k(s)\widehat I_{k}^{s,M_z}
 (\bar X_{\tau_k(s)}^{h,\mathrm{emp},m;s,x}).
\end{aligned}
\end{equation}
Thus the same $\{z^\ell\}_{\ell=1}^{M_z}$ are reused by all Euler
trajectories. The associated piecewise-constant interpolation is denoted by
$\overline X^{h,\mathrm{emp},m;s,x}$; initial data or the branch index
are suppressed only when unambiguous.

In this common-reusable-compensator construction, the branches use
independent Brownian increments, Poisson counts, and jump marks, while one
compensator sample set is shared by all branches. Here $G_m=G_m^{\mathrm{emp}}$
in \eqref{eq:mc_branch_payoff}. Thus, conditional on the
common sample set, $G_1(t,x)$ denotes the first copy in the iid family of
branch functionals, and
\[
\mathrm{MC}^{\mathrm{com}}_{h,M,M_z}(t,x)=\frac1M\sum_{m=1}^M G_m(t,x).
\]

We now prove the technical statements used in this subsection. They are
distinct from the coefficient assumptions \ref{G2}--\ref{G4} in Section~2.
The global pointwise Lipschitz assumption~\ref{Cstar1}, stated in Section~2,
allows us to transfer compensator consistency on independent exact Euler
states to stability of the reusable-sample scheme. Uniform-in-$M_z$
empirical moment bounds are then consequences,
they are established in Lemma~\ref{lem:empirical_euler_moments}.

The following statements use the common compensator construction. In them,
$Z=(z^1,\ldots,z^{M_z})$ is independent of the trajectory noises, and
$\E_{Z,\xi}$ denotes expectation over both sources of randomness, while
$\E_\xi[\,\cdot\mid Z]$ denotes expectation over the trajectory noises with
$Z$ fixed.

\begin{lemma}
\label{lem:empirical_compensator_consistency}
Assume \ref{A1}--\ref{A3} and \ref{C1}--\ref{C5}. For fixed $s\in[0,T]$
and $x\in\R^d$, let $K(s)$ be the number of shifted Euler steps and put
\[
X_k^{\mathrm{exact}}:=\bar X_{\tau_k(s)}^{h;s,x},
\qquad h_k:=\Delta\tau_k(s).
\]
These exact-compensator Euler states are independent of $Z$. Define
\[
\epsilon_k:=\widehat I_k^{s,M_z}(X_k^{\mathrm{exact}})
 -I_k^s(X_k^{\mathrm{exact}}).
\]
There is a constant $C_I(d)<\infty$, independent of $h$, $M_z$, $s$, and
$x$, such that
\begin{equation}
\label{eq:empirical_compensator_consistency}
\E_{Z,\xi}\sum_{k<K(s)}h_k\norm{\epsilon_k}^2
\leq \frac{C_I(d)(1+\norm{x}^2)}{M_z}.
\end{equation}
Moreover, there are constants $C,\kappa_I>0$, independent of $d$, such
that $C_I(d)\leq C(1+d)^{\kappa_I}$.
\end{lemma}

\begin{proof}
By \eqref{jump_int_norm} with $q=2$, we may choose
\begin{align*}
J_d&:=2\max\{L_\jump(d),
 K_\jump(d)+T^{\beta_\jump}L_\jump(d)\}^2,\\
\int_{\Z}\norm{c(t,y,z)}^2\nu(\dd z)
&\leq J_d(1+\norm{y}^2).
\end{align*}
The exact-compensator moment estimate
\eqref{eq:euler_interpolation_moment} does not use the density assumption
\ref{F1}. With $M_d$ defined in
\eqref{eq:euler_moment_constant_definition}, it gives
\[
\max_{k\leq K(s)}\E_\xi\norm{X_k^{\mathrm{exact}}}^2
\leq B_0(d)(1+\norm{x}^2),
\qquad B_0(d):=M_d.
\]
In particular, $B_0(d)$ grows at most polynomially in $d$: only
$1+L_a(d)^2+L_\jump(d)^2=O(1+\log(1+d))$ occurs in its
exponent, while the additive Brownian dimension term $d$ remains in the
prefactor.
Let $\mathcal H^{\mathrm{exact}}$ be the sigma-field generated by the
entire exact-compensator Euler trajectory. It is independent of $Z$.
Conditional on this sigma-field, each state in $\epsilon_k$ is fixed,
and the compensator marks remain iid with law $\rho_d$. Since
$\nu(\dd z)=\lambda(d)\rho_d(\dd z)$, put
\begin{align*}
\bar c_k&:=\int_{\Z}c(\tau_k(s),X_k^{\mathrm{exact}},z)\rho_d(\dd z),\\
U_{k,\ell}&:=c(\tau_k(s),X_k^{\mathrm{exact}},z^\ell)-\bar c_k,
\qquad \ell=1,\ldots,M_z.
\end{align*}
Then
\begin{align*}
\epsilon_k
&=\frac{\lambda(d)}{M_z}\sum_{\ell=1}^{M_z}
 c(\tau_k(s),X_k^{\mathrm{exact}},z^\ell)
 -\int_{\Z}c(\tau_k(s),X_k^{\mathrm{exact}},z)\nu(\dd z)\\
&=\frac{\lambda(d)}{M_z}\sum_{\ell=1}^{M_z}
 c(\tau_k(s),X_k^{\mathrm{exact}},z^\ell)-\lambda(d)\bar c_k\\
&=\frac{\lambda(d)}{M_z}\sum_{\ell=1}^{M_z}U_{k,\ell}.
\end{align*}
Conditional independence and centering give, for $\ell\ne j$,
\begin{align*}
\E_Z[U_{k,\ell}\mid\mathcal H^{\mathrm{exact}}]&=0,\\
\E_Z[\langle U_{k,\ell},U_{k,j}\rangle
 \mid\mathcal H^{\mathrm{exact}}]
&=\left\langle
 \E_Z[U_{k,\ell}\mid\mathcal H^{\mathrm{exact}}],
 \E_Z[U_{k,j}\mid\mathcal H^{\mathrm{exact}}]
 \right\rangle=0.
\end{align*}
Consequently,
\begin{align*}
\E_Z[\norm{\epsilon_k}^2\mid\mathcal H^{\mathrm{exact}}]
&=\frac{\lambda(d)^2}{M_z^2}
 \E_Z\left[\left\|\sum_{\ell=1}^{M_z}U_{k,\ell}\right\|^2
 \mathrel{\big|}\mathcal H^{\mathrm{exact}}\right]\\
&=\frac{\lambda(d)^2}{M_z^2}
 \Bigg(\sum_{\ell=1}^{M_z}
 \E_Z[\norm{U_{k,\ell}}^2\mid\mathcal H^{\mathrm{exact}}]\\
 &\qquad+2\sum_{\ell<j}
 \E_Z[\langle U_{k,\ell},U_{k,j}\rangle
 \mid\mathcal H^{\mathrm{exact}}]\Bigg)\\
&=\frac{\lambda(d)^2}{M_z}
 \int_{\Z}\norm{c(\tau_k(s),X_k^{\mathrm{exact}},z)-\bar c_k}^2
 \rho_d(\dd z)\\
&=\frac{\lambda(d)^2}{M_z}
 \operatorname{Var}_{\rho_d}
 \bigl(c(\tau_k(s),X_k^{\mathrm{exact}},Z_{\mathrm{mark}})\bigr).
\end{align*}
Here $Z_{\mathrm{mark}}\sim\rho_d$ is an independent mark.
Expanding the centered square gives
\begin{align*}
\E_Z[\norm{\epsilon_k}^2\mid\mathcal H^{\mathrm{exact}}]
&=\frac{\lambda(d)^2}{M_z}
 \left(\int_{\Z}\norm{c(\tau_k(s),X_k^{\mathrm{exact}},z)}^2
 \rho_d(\dd z)-\norm{\bar c_k}^2\right)\\
&\leq\frac{\lambda(d)^2}{M_z}
 \int_{\Z}\norm{c(\tau_k(s),X_k^{\mathrm{exact}},z)}^2
 \rho_d(\dd z)\\
&=\frac{\lambda(d)}{M_z}
 \int_{\Z}\norm{c(\tau_k(s),X_k^{\mathrm{exact}},z)}^2\nu(\dd z).
\end{align*}
Conditioning on the entire trajectory is only a convenience: for each
$k$, conditioning on $\sigma(X_k^{\mathrm{exact}})$ suffices.
Taking expectations and summing with weights $h_k$, whose sum
is at most $T$, proves \eqref{eq:empirical_compensator_consistency} with
\[
C_I(d):=T\lambda(d)J_d(1+B_0(d)).
\]
By \ref{C1}, $J_d=O((1+\log(1+d))^2)$, and the assumed intensity
bound gives $\lambda(d)=O(1+\log(1+d))$. Together with the polynomial
bound on $B_0(d)$, this proves the asserted dimension dependence of $C_I(d)$.
No independence between the discrepancies at different grid indices is
needed. For $s=T$, the sum is empty and the assertion is immediate.
\end{proof}

Here $G^{\mathrm{emp}}(s,x)$ and $G^0(s,x)$ denote the corresponding branch
functionals computed with the empirical and exact compensators, respectively.

For comparison, let $Z_1,\ldots,Z_M$ be independent copies of the
compensator sample set $Z$, independent also of the trajectory noises, and let
$G_m^{\mathrm{ind}}(s,x)$ denote the branch functional obtained when
$Z_m$ is used in place of the common set $Z$. Let
$G_m^{\mathrm{com}}(s,x)$ denote the branch functional obtained when the
same $Z$ is used in every branch.

In the common-compensator construction the branch functionals are generally
correlated through $Z$, so the usual $M^{-1}$ variance reduction cannot be
inferred from branchwise independence. Moreover, at a random state depending on $Z$, the
empirical compensator need not be conditionally unbiased.
The appropriate conditional variance decomposition is given below.

More precisely, writing
\[
\mu_Z(s,x)=\E_\xi\left[G_1^{\mathrm{com}}(s,x)\mid Z\right],
\qquad
\sigma_Z^2(s,x)=\E_\xi
\left[|G_1^{\mathrm{com}}(s,x)-\mu_Z(s,x)|^2\mid Z\right],
\]
the conditional variance formula gives
\begin{align}
\operatorname{Var}\left(\frac1M\sum_{m=1}^M G_m^{\mathrm{com}}(s,x)\right)
&=\E\left[\frac{\sigma_Z^2(s,x)}{M}\right]
    +\operatorname{Var}\left(\mu_Z(s,x)\right),
\label{eq:common_compensator_variance}\\
\operatorname{Var}\left(\frac1M\sum_{m=1}^M G_m^{\mathrm{ind}}(s,x)\right)
&=\frac{1}{M}\operatorname{Var}\left(G_1^{\mathrm{ind}}(s,x)\right).
\label{eq:independent_compensator_variance}
\end{align}
The common and independent constructions have the same single-branch law.
Consequently, the total variance formula shows that
\eqref{eq:independent_compensator_variance} equals
\[
\operatorname{Var}\left(G_1^{\mathrm{com}}(s,x)\right)
= \operatorname{Var}\left(G_1^{\mathrm{ind}}(s,x)\right)
=
\frac{\E[\sigma_Z^2(s,x)]}{M}
 +\frac{\operatorname{Var}(\mu_Z(s,x))}{M}.
\]
The variance due to the random conditional mean therefore does not
disappear with independent compensator sample sets; it is reduced by
the factor $M^{-1}$. By \eqref{eq:common_compensator_variance}, the
increase caused by sharing one compensator sample set is exactly
\[
\begin{aligned}
&\operatorname{Var}\left(\frac1M\sum_{m=1}^M G_m^{\mathrm{com}}(s,x)\right)
 -\operatorname{Var}\left(\frac1M\sum_{m=1}^M G_m^{\mathrm{ind}}(s,x)\right)\\
&\qquad=\left(1-\frac1M\right)\operatorname{Var}(\mu_Z(s,x)).
\end{aligned}
\]

\begin{lemma}
\label{lem:common_compensator_stability}
Assume \ref{A1}--\ref{A3}, \ref{C1}--\ref{C5}, and \ref{Cstar1}.
If the empirical and exact-compensator Euler schemes are driven by the
same trajectory noises, then
there is a constant $C_{\mathrm{stab}}(d)$, independent of $h$, $M_z$, $s$,
and $x$, such that
\begin{equation}
\label{eq:common_compensator_stability}
\E_{Z,\xi}\left[\sup_{s\leq v\leq T}
\norm{\overline X_v^{h,\mathrm{emp};s,x}-
\overline X_v^{h;s,x}}^2\right]
\leq \frac{C_{\mathrm{stab}}(d)}{M_z}(1+\norm{x}^2).
\end{equation}
If, in addition, \ref{B3} and \ref{G3} hold, there is a constant
$C_{\mathrm{pay}}(d)$, independent of the same parameters, such that
\begin{equation}
\label{eq:common_compensator_payoff_stability}
\E_{Z,\xi}\left|G^{\mathrm{emp}}(s,x)-G^0(s,x)\right|^2
\leq \frac{C_{\mathrm{pay}}(d)}{M_z}(1+\norm{x}^2).
\end{equation}
The constants can be chosen to grow at most polynomially in $d$:
there are $C,\kappa_{\mathrm{st}}>0$, independent of $d$, such that
$C_{\mathrm{stab}}(d)\leq C(1+d)^{\kappa_{\mathrm{st}}}$ and,
under \ref{B3}, \ref{G1}, and \ref{G3},
$C_{\mathrm{pay}}(d)\leq C(1+d)^{\kappa_{\mathrm{st}}}$.
\end{lemma}

\begin{proof}
Fix $s$, $x$, and use the same trajectory noises in both schemes. Write
\begin{align*}
X_k^e&=\bar X_{\tau_k(s)}^{h,\mathrm{emp};s,x},
&X_k^{\mathrm{exact}}&=\bar X_{\tau_k(s)}^{h;s,x},\\
D_k&=X_k^e-X_k^{\mathrm{exact}},
&h_k&=\Delta\tau_k(s).
\end{align*}
Let $\mathcal F_k$ be generated by $Z$ and both coupled trajectories up
to time $\tau_k(s)$, before the next trajectory increments are sampled.
Define the $\mathcal F_k$-measurable quantities
\begin{align*}
\epsilon_k&=\widehat I_k^{s,M_z}(X_k^{\mathrm{exact}})
 -I_k^s(X_k^{\mathrm{exact}}),\\
\Delta a_k&=a(\tau_k(s),X_k^e)-a(\tau_k(s),X_k^{\mathrm{exact}}),\\
\Delta I_k&=I_k^s(X_k^e)-I_k^s(X_k^{\mathrm{exact}}),\\
\Delta c_k(z)&=c(\tau_k(s),X_k^e,z)
 -c(\tau_k(s),X_k^{\mathrm{exact}},z),\\
r_k&=\widehat I_k^{s,M_z}(X_k^e)
 -\widehat I_k^{s,M_z}(X_k^{\mathrm{exact}})-\Delta I_k.
\end{align*}

From \ref{Cstar1}, all sampled compensator marks belong to this set on an event of
probability one. On that event, \ref{Cstar1} holds simultaneously for
all times and all pairs of states. We may therefore substitute
$t=\tau_k(s)$, $x=X_k^e$, and $y=X_k^{\mathrm{exact}}$, even though
$X_k^e$ depends on $Z$; no independence is needed for this pathwise
substitution. Consequently, for every $\ell=1,\ldots,M_z$,
\[
\norm{c(\tau_k(s),X_k^e,z^\ell)
 -c(\tau_k(s),X_k^{\mathrm{exact}},z^\ell)}
\leq\ell_c(d)\norm{X_k^e-X_k^{\mathrm{exact}}}.
\]
Using the same marks in both empirical compensators, their
definition and the triangle inequality give, pathwise,
\begin{align*}
&\norm{\widehat I_k^{s,M_z}(X_k^e)
 -\widehat I_k^{s,M_z}(X_k^{\mathrm{exact}})}\\
&\quad=\frac{\lambda(d)}{M_z}
 \norm{\sum_{\ell=1}^{M_z}\bigl[
 c(\tau_k(s),X_k^e,z^\ell)
 -c(\tau_k(s),X_k^{\mathrm{exact}},z^\ell)\bigr]}\\
&\quad\leq\frac{\lambda(d)}{M_z}\sum_{\ell=1}^{M_z}
 \norm{c(\tau_k(s),X_k^e,z^\ell)
 -c(\tau_k(s),X_k^{\mathrm{exact}},z^\ell)}\\
&\quad\leq\frac{\lambda(d)}{M_z}\sum_{\ell=1}^{M_z}
 \ell_c(d)\norm{D_k}
 =\lambda(d)\ell_c(d)\norm{D_k}.
\end{align*}
Integrating the pointwise Lipschitz bound in \ref{Cstar1} gives
\begin{align*}
\norm{\Delta I_k}
&\leq\int_{\Z}\norm{\Delta c_k(z)}\nu(\dd z)\\
&\leq\lambda(d)\ell_c(d)\norm{D_k}.
\end{align*}
Consequently, with deterministic constants
\begin{align*}
L_r(d)&:=2\lambda(d)\ell_c(d),\\
B_{\mathrm{st}}(d)&:=L_a(d)+L_r(d),
\end{align*}
we have $\norm{r_k}\leq L_r(d)\norm{D_k}$ and
$\norm{\Delta a_k-r_k}\leq B_{\mathrm{st}}(d)\norm{D_k}$.
The Brownian increments cancel. The remaining jump difference is not
zero, but its centered part
\[
\Delta\mathcal N_k=\sum_{n=1}^{N_k^s}\Delta c_k(Y_{k,n}^s)-h_k\Delta I_k
\]
satisfies, by the conditional Poisson isometry and \ref{C4},
\[
\E[\Delta\mathcal N_k\mid\mathcal F_k]=0,
\qquad
\E[\norm{\Delta\mathcal N_k}^2\mid\mathcal F_k]
\leq h_kL_\jump(d)^2\norm{D_k}^2.
\]
Subtracting the recursions, the mean jump difference cancels the exact
compensator difference, leaving
\[
D_{k+1}=D_k+h_k(\Delta a_k-r_k-\epsilon_k)+\Delta\mathcal N_k,
\qquad D_0=0.
\]
The forcing $\epsilon_k$ is $\mathcal F_k$-measurable and is not generally
centered conditional on $\mathcal F_k$. Its unconditional second moment
is controlled by Lemma~\ref{lem:empirical_compensator_consistency}, using
independence of the entire exact Euler trajectory from $Z$, not independence
from the empirical state. To estimate the conditional second moment,
first use the recursion to write
\begin{align*}
\E[\norm{D_{k+1}}^2\mid\mathcal F_k]
&=\E\bigl[\norm{D_k+h_k(\Delta a_k-r_k-\epsilon_k)
 +\Delta\mathcal N_k}^2\mid\mathcal F_k\bigr]\\
&=\norm{D_k+h_k(\Delta a_k-r_k-\epsilon_k)}^2
 +\E[\norm{\Delta\mathcal N_k}^2\mid\mathcal F_k].
\end{align*}
Here the mixed term vanishes because
$D_k+h_k(\Delta a_k-r_k-\epsilon_k)$ is $\mathcal F_k$-measurable
and $\E[\Delta\mathcal N_k\mid\mathcal F_k]=0$.
Expanding the remaining squared norm gives
\begin{align*}
\E[\norm{D_{k+1}}^2\mid\mathcal F_k]
&=\norm{D_k}^2
 +2h_k\langle D_k,\Delta a_k-r_k-\epsilon_k\rangle\\
&\quad+h_k^2\norm{\Delta a_k-r_k-\epsilon_k}^2
 +\E[\norm{\Delta\mathcal N_k}^2\mid\mathcal F_k].
\end{align*}
Cauchy--Schwarz, Young's inequality, and
$\norm{\Delta a_k-r_k}\leq B_{\mathrm{st}}(d)\norm{D_k}$ yield
\begin{align*}
2\langle D_k,\Delta a_k-r_k-\epsilon_k\rangle
&\leq(2B_{\mathrm{st}}(d)+1)\norm{D_k}^2+\norm{\epsilon_k}^2,\\
\norm{\Delta a_k-r_k-\epsilon_k}^2
&\leq2B_{\mathrm{st}}(d)^2\norm{D_k}^2+2\norm{\epsilon_k}^2.
\end{align*}
Substituting these bounds and the conditional Poisson estimate above,
and using $h_k^2\leq T h_k$, gives
\begin{align}
\E[\norm{D_{k+1}}^2\mid\mathcal F_k]
&\leq(1+\Gamma_{\mathrm{st}}(d)h_k)\norm{D_k}^2\nonumber\\
&\quad+(1+2T)h_k\norm{\epsilon_k}^2,
\label{eq:common_compensator_one_step}
\end{align}
where
\[
\Gamma_{\mathrm{st}}(d)
=2B_{\mathrm{st}}(d)+1+2TB_{\mathrm{st}}(d)^2+L_\jump(d)^2.
\]
These calculations are justified inductively: $D_0=0$ is square
integrable, each $\epsilon_k$ has a finite unconditional second moment by
the consistency proof, and the conditional Poisson isometry makes
$\Delta\mathcal N_k$ square integrable whenever $D_k$ is. The deterministic
Lipschitz bound then makes $D_{k+1}$ square integrable.
Taking expectations in \eqref{eq:common_compensator_one_step} and applying
the discrete Grönwall inequality gives the grid second-moment estimate
\[
\max_k\E_{Z,\xi}\norm{D_k}^2
\leq (1+2T)e^{\Gamma_{\mathrm{st}}(d)T}
\E_{Z,\xi}\sum_k h_k\norm{\epsilon_k}^2.
\]
For the expected maximum, put
$\mathcal N_n=\sum_{k<n}\Delta\mathcal N_k$ and
$F_n=\E_{Z,\xi}\max_{j\leq n}\norm{D_j}^2$.
The identity
$D_j=\sum_{k<j}h_k(\Delta a_k-r_k-\epsilon_k)+\mathcal N_j$,
Cauchy--Schwarz, and Doob's inequality yield
\begin{align*}
F_n
&\leq4TB_{\mathrm{st}}(d)^2\sum_{k<n}h_kF_k
 +4T\E_{Z,\xi}\sum_{k<n}h_k\norm{\epsilon_k}^2\\
&\quad+8\E_{Z,\xi}\norm{\mathcal N_n}^2\\
&\leq\bigl(4TB_{\mathrm{st}}(d)^2+8L_\jump(d)^2\bigr)
 \sum_{k<n}h_kF_k\\
&\quad+4T\E_{Z,\xi}\sum_{k<n}h_k\norm{\epsilon_k}^2.
\end{align*}
The induction above and the finite number of grid steps ensure that
$F_n$ is finite and that $\mathcal N_n$ is a square-integrable martingale,
so Doob's inequality applies.
Another discrete Grönwall estimate gives
\[
F_{K(s)}\leq
4T e^{(4TB_{\mathrm{st}}(d)^2+8L_\jump(d)^2)T}
\E_{Z,\xi}\sum_k h_k\norm{\epsilon_k}^2.
\]
Lemma~\ref{lem:empirical_compensator_consistency} bounds the last
expectation with the required weight $h_k$.
Both interpolations are piecewise constant, including at $T$; their
pathwise supremum difference equals the maximum over the grid.
Thus \eqref{eq:common_compensator_stability} holds with
\[
C_{\mathrm{stab}}(d)=4T
e^{(4TB_{\mathrm{st}}(d)^2+8L_\jump(d)^2)T}C_I(d).
\]
For $s=T$, both schemes equal $x$ and their difference is zero.

For the branch functional, the Lipschitz bounds in \ref{G3} and \ref{B3},
Cauchy--Schwarz, and $\sum_k h_k\leq T$ give
\begin{align*}
\left|G^{\mathrm{emp}}(s,x)-G^0(s,x)\right|^2
&\leq 2L_g(d)^2\norm{D_{K(s)}}^2
 +2T L_b(d)^2\sum_k h_k\norm{D_k}^2.
\end{align*}
Taking expectations and using the preceding estimate proves
\eqref{eq:common_compensator_payoff_stability} with
\[
C_{\mathrm{pay}}(d)
=2\bigl(L_g(d)^2+T^2L_b(d)^2\bigr)C_{\mathrm{stab}}(d),
\]
which is independent of $h$ and $M_z$.
To verify the dimension bounds, put $\eta_d:=1+\log(1+d)$.
By \ref{A1}, \ref{C1}, and the product bound in \ref{Cstar1}, there
are constants $b_a,b_c,b_\ell>0$, independent of $d$, such that
\[
L_a(d)^2\leq b_a\eta_d,\qquad
L_\jump(d)^2\leq b_c\eta_d,\qquad
(\lambda(d)\ell_c(d))^2\leq b_\ell\eta_d.
\]
Thus $B_{\mathrm{st}}(d)^2\leq(2b_a+8b_\ell)\eta_d$. In particular,
with $q_{\mathrm{st}}:=4T^2(2b_a+8b_\ell)+8Tb_c$, we have
\begin{align*}
\Gamma_{\mathrm{st}}(d)&=O(\eta_d),\\
C_{\mathrm{stab}}(d)
&\leq4T e^{q_{\mathrm{st}}}(1+d)^{q_{\mathrm{st}}}C_I(d).
\end{align*}
The consistency lemma proves a polynomial bound on $C_I(d)$.
Finally, \ref{B3} and \ref{G1} bound $L_b(d)$ and $L_g(d)$ by
$O(\eta_d)$, so the displayed formula for $C_{\mathrm{pay}}(d)$
proves its polynomial growth as well.
\end{proof}

\begin{lemma}
\label{lem:empirical_euler_moments}
Assume \ref{A1}--\ref{A3}, \ref{C1}--\ref{C5}, and \ref{Cstar1}.
There is a constant $C_{\mathrm{emp}}(d)$, independent of $h$, $M_z$, $s$,
and $x$, such that, for $M_z\geq1$,
\begin{equation}
\label{eq:empirical_euler_moments}
\E_{Z,\xi}\left[\sup_{s\leq v\leq T}
\norm{\overline X_v^{h,\mathrm{emp};s,x}}^2\right]
\leq C_{\mathrm{emp}}(d)(1+\norm{x}^2).
\end{equation}
Moreover, $C_{\mathrm{emp}}(d)\leq C(1+d)^{\kappa_{\mathrm{emp}}}$
for some $C,\kappa_{\mathrm{emp}}>0$ independent of $d$.
\end{lemma}

\begin{proof}
Use the exact grid states, $J_d$, and $B_0(d)$ from the proof of
Lemma~\ref{lem:empirical_compensator_consistency}. Their compensated
recursion is
\[
X_n^{\mathrm{exact}}=x+\sum_{k<n}h_ka_k+\mathcal M_n,
\qquad a_k=a(\tau_k(s),X_k^{\mathrm{exact}}),
\]
where $\mathcal M_n$ is the square-integrable martingale formed by the
Brownian and compensated-jump increments. More explicitly, let
$\mathcal F_k^{\mathrm{exact}}$ be the sigma-field generated by the
trajectory noises up to $\tau_k(s)$. We put
\begin{align*}
\Delta\mathcal J_k
&:=\sum_{n=1}^{N_k^s}
 c(\tau_k(s),X_k^{\mathrm{exact}},Y_{k,n}^s)
 -h_k I_k^s(X_k^{\mathrm{exact}}),\\
\Delta\mathcal M_k&:=\Delta W_k^s+\Delta\mathcal J_k,
\qquad \mathcal M_n=\sum_{k<n}\Delta\mathcal M_k.
\end{align*}
Conditional on $\mathcal F_k^{\mathrm{exact}}$, the state
$X_k^{\mathrm{exact}}$ is fixed, and the new Brownian increment is
independent of the new Poisson count and jump marks. Both
$\Delta W_k^s$ and $\Delta\mathcal J_k$ have conditional mean zero.
Consequently,
\[
\E_\xi[\langle\Delta W_k^s,\Delta\mathcal J_k\rangle
 \mid\mathcal F_k^{\mathrm{exact}}]=0.
\]
The Brownian second-moment identity and the conditional Poisson
isometry give
\[
\E_\xi[\norm{\Delta\mathcal M_k}^2
 \mid\mathcal F_k^{\mathrm{exact}}]
=h_k d+h_k\int_{\Z}
 \norm{c(\tau_k(s),X_k^{\mathrm{exact}},z)}^2\nu(\dd z).
\]
Moreover, increments at different steps are orthogonal in $L^2$:
for $j<k$, $\Delta\mathcal M_j$ is
$\mathcal F_k^{\mathrm{exact}}$-measurable, whereas
$\E_\xi[\Delta\mathcal M_k\mid\mathcal F_k^{\mathrm{exact}}]=0$, so
\[
\E_\xi\langle\Delta\mathcal M_j,\Delta\mathcal M_k\rangle
=\E_\xi\left\langle\Delta\mathcal M_j,
 \E_\xi[\Delta\mathcal M_k\mid\mathcal F_k^{\mathrm{exact}}]
 \right\rangle=0.
\]
Using \eqref{jump_int_norm} and taking expectations give
\[
\E_\xi\int_{\Z}
 \norm{c(\tau_k(s),X_k^{\mathrm{exact}},z)}^2\nu(\dd z)
\leq J_d(1+\E_\xi\norm{X_k^{\mathrm{exact}}}^2).
\]
Hence
\begin{align*}
\E_\xi\norm{\mathcal M_{K(s)}}^2
&=\sum_{k<K(s)}\E_\xi\norm{\Delta\mathcal M_k}^2\\
&=\sum_{k<K(s)}h_k\left[d+\E_\xi\int_{\Z}
 \norm{c(\tau_k(s),X_k^{\mathrm{exact}},z)}^2\nu(\dd z)\right]\\
&\leq\sum_{k<K(s)}h_k\bigl[d+J_d(1+
 \E_\xi\norm{X_k^{\mathrm{exact}}}^2)\bigr].
\end{align*}
The pointwise grid moment bound, \ref{A2}, Cauchy--Schwarz, and Doob's
inequality therefore give
\begin{align*}
\E_\xi\max_{n\leq K(s)}\norm{X_n^{\mathrm{exact}}}^2
&\leq3\norm{x}^2+3T\sum_k h_k\E_\xi\norm{a_k}^2\\
&\quad+12\E_\xi\norm{\mathcal M_{K(s)}}^2\\
&\leq B_{\max}(d)(1+\norm{x}^2),
\end{align*}
where
\begin{align*}
B_{\max}(d)&=3+6T^2K_a(d)^2(1+B_0(d))\\
&\quad+12T\bigl[d+J_d(1+B_0(d))\bigr].
\end{align*}
For the piecewise-constant interpolations, including their terminal grid
values, the supremum equals the grid maximum. Recall that
\[
X_k^e:=\bar X_{\tau_k(s)}^{h,\mathrm{emp};s,x},
\qquad D_k:=X_k^e-X_k^{\mathrm{exact}}.
\]
The inequality
$\norm{X_k^e}^2\leq2\norm{X_k^{\mathrm{exact}}}^2+2\norm{D_k}^2$
and the state estimate \eqref{eq:common_compensator_stability} yield
\eqref{eq:empirical_euler_moments} with
\[
C_{\mathrm{emp}}(d):=2B_{\max}(d)+2C_{\mathrm{stab}}(d),
\]
using $M_z\geq1$. This derivation needs neither \ref{B3} nor \ref{G3}.
The constants $K_a(d)^2$, $J_d$, and $B_0(d)$ have polynomial
bounds, hence so does $B_{\max}(d)$. The polynomial stability bound
then gives the asserted dimension dependence of $C_{\mathrm{emp}}(d)$.
For $s=T$, the process equals its deterministic initial value.
\end{proof}

\begin{lemma}
\label{G_ass_3}
Assume \ref{U1}--\ref{U4} and \ref{F1}.
The coefficients satisfy \ref{A1}--\ref{A3}, \ref{B1}--\ref{B4},
\ref{C1}--\ref{C5}, and \ref{G1}--\ref{G4}.
Assume also \ref{Cstar1} for the common reusable compensator construction.
Let $h$ be admissible with $0<h\leq\min\{1,T\}$.
There are constants $C_u,C_M,C_z>0$, independent of $h$, $M$, and $M_z$,
such that
\[
\E\int_0^T\int_{\R^d}
\left|u(t,x)-\mathrm{MC}^{\mathrm{com}}_{h,M,M_z}(t,x)\right|^2
 f_d(t,x)\,\dd x\,\dd t
\le C_u h^{2\alpha_Q}+\frac{C_M}{M}+\frac{C_z}{M_z}.
\]
Moreover, there are constants
$C,\kappa_{\mathrm{com}}>0$, independent of $d$, $h$, $M$, and $M_z$,
such that
\[
C_u(d)+C_M(d)+C_z(d)\leq C(1+d)^{\kappa_{\mathrm{com}}}.
\]
\end{lemma}

\begin{proof}
For fixed $(s,x)$, couple the exact trajectory and the exact-compensator
Euler trajectory using the same Brownian and Poisson noises. Define
\[
\mathcal H(s,x)=g(X_T^{s,x})+\int_s^T b(v,X_v^{s,x})\,\dd v.
\]
By \eqref{eq:u_decompositin}, $u(s,x)=\E\mathcal H(s,x)$.
Jensen's inequality, \ref{G3},
Proposition~\ref{prop:dimension_explicit_euler_constant}, and
Lemma~\ref{lem:running_source_quadrature} give
\begin{align*}
|u(s,x)-\E G^0(s,x)|^2
&\leq\E|\mathcal H(s,x)-G^0(s,x)|^2\\
&\leq 2\bigl(L_g(d)^2A_d+C_Q(d)^2\bigr)
h^{2\alpha_Q}(1+\norm{x})^2.
\end{align*}
Here the continuous Euler interpolation equals the grid value at $T$,
and $h^{2\alpha_{\mathrm E}}\leq h^{2\alpha_Q}$ because $h\leq1$.
The weighted integral is finite by \ref{F1} and
$(1+\norm{x})^2\leq2(1+\norm{x}^2)$.

For every branch and every realization, \ref{B4}, \ref{G4}, and
$\sum_k\Delta\tau_k(s)=T-s$ imply
\[
|G_m^{\mathrm{com}}(s,x)|\leq M_g(d)+TM_b(d).
\]
Since $\mu_f$ is a probability measure by \ref{F1}, Tonelli's theorem gives
\[
\int\E|G_1^{\mathrm{com}}(s,x)|^2\,\mu_f(\dd s,\dd x)
\leq\bigl(M_g(d)+TM_b(d)\bigr)^2.
\]
Let $\mu_Z(s,x)=\E_\xi[G_1^{\mathrm{com}}(s,x)\mid Z]$.
Conditional on the common compensator sample set $Z$,
the branch noises are independent and identically distributed, so
\begin{align*}
\E_{Z,\xi}\left|M^{-1}\sum_{m=1}^M G_m^{\mathrm{com}}-\mu_Z\right|^2
&=M^{-1}\E_{Z,\xi}\left|G_1^{\mathrm{com}}-\mu_Z\right|^2\\
&\leq M^{-1}\E_{Z,\xi}|G_1^{\mathrm{com}}|^2.
\end{align*}
The boundedness estimate above controls this term without a state-moment
lemma. Couple $G^0$ with $G_1^{\mathrm{com}}$ using the same trajectory
noises, as in Lemma~\ref{lem:common_compensator_stability}.
Since the exact-compensator branch $G^0$ is independent of $Z$,
conditional Jensen's inequality yields
\[
|\mu_Z-\E G^0|^2
\leq \E_\xi[|G_1^{\mathrm{com}}-G^0|^2\mid Z].
\]
After expectation over $Z$, Lemma~\ref{lem:common_compensator_stability}
and \ref{F1} give the $M_z^{-1}$ contribution. Finally, use
\begin{align*}
u-\mathrm{MC}^{\mathrm{com}}_{h,M,M_z}
&=(u-\E G^0)+(\E G^0-\mu_Z)\\
&\quad +(\mu_Z-\mathrm{MC}^{\mathrm{com}}_{h,M,M_z})
\end{align*}
and $(a+b+c)^2\leq3(a^2+b^2+c^2)$ to obtain the asserted bound with
\[
\begin{aligned}
    C_u(d)&:=1+6\bigl(L_g(d)^2A_d+C_Q(d)^2\bigr)
    \int(1+\norm{x})^2\,\mu_f(\dd s,\dd x),\\
    C_M(d)&:=1+3\bigl(M_g(d)+TM_b(d)\bigr)^2,\\
    C_z(d)&:=1+3C_{\mathrm{pay}}(d)
    \int(1+\norm{x}^2)\,\mu_f(\dd s,\dd x).
\end{aligned}
\]
Here $A_d$ is the strong-error constant from
Proposition~\ref{prop:dimension_explicit_euler_constant}, $C_Q(d)$ is
defined in Lemma~\ref{lem:running_source_quadrature}, and
$C_{\mathrm{pay}}(d)$ is the branch-functional stability constant from
Lemma~\ref{lem:common_compensator_stability}.
By \ref{F1} and the dimension bounds in those Euler and quadrature
results, $C_u(d)$ and $C_M(d)$ grow at most polynomially in $d$.
The polynomial bound on $C_{\mathrm{pay}}(d)$ in
Lemma~\ref{lem:common_compensator_stability}, together with \ref{F1},
gives the same conclusion for $C_z(d)$, proving the asserted bound on
$C_u(d)+C_M(d)+C_z(d)$.
This argument uses independence only conditional on $Z$.
\end{proof}

\begin{theorem}
\label{thm:mc_weighted_l2}
Assume \ref{U1}--\ref{U4}, \ref{F1}, \ref{A1}--\ref{A3}, \ref{B1}--\ref{B4},
\ref{C1}--\ref{C5}, \ref{G1}--\ref{G4}, and \ref{Cstar1}.
The common compensator uses one
sample set $Z^{\mathrm{comp}}$ shared by all branches. The jump marks used
to simulate the trajectories are generated from independent sets
$Z_m^{\mathrm{jump}}$, independent of $Z^{\mathrm{comp}}$ and independent
across $m$. Let $\mathrm{MC}^{\mathrm{com}}_{h,M,M_z}$ denote the resulting
measurable estimator. By \ref{B4}, \ref{G4}, and \ref{F1}, it belongs to
$L^2(\mu_f)$ almost surely.
Then, for every $\epsilon,\delta\in(0,1)$, there exist admissible
$h$, $M$, and $M_z$ such that
\[
\Pra\left(\left\|u-\mathrm{MC}^{\mathrm{com}}_{h,M,M_z}\right\|_{L^2(\mu_f)}
<\epsilon\right)\geq 1-\delta.
\]
Moreover, $N_{\mathrm{Euler}}=T/h$, $M$, and $M_z$ can be chosen
with at most polynomial dependence on $d$, $\epsilon^{-1}$, and
$\delta^{-1}$.
\end{theorem}

\begin{proof}
Fix $\epsilon,\delta\in(0,1)$ and define
\[
R:=\left\|u-\mathrm{MC}^{\mathrm{com}}_{h,M,M_z}\right\|_{L^2(\mu_f)}^2.
\]
This is a nonnegative measurable random variable. By \ref{B4} and
\ref{G4} and \eqref{eq:u_decompositin}, both $u$ and every branch functional
are bounded in absolute value by
$M_g(d)+TM_b(d)$; since $\mu_f$ is a probability measure by \ref{F1},
$R$ is finite almost surely. Lemma~\ref{G_ass_3} gives
\[
\E R\leq C_u h^{2\alpha_Q}+\frac{C_M}{M}+\frac{C_z}{M_z},
\]
with the constants $C_u,C_M,C_z$ of that lemma. Choose admissible
parameters such that
\[
\begin{aligned}
h&\leq\min\left\{1,T,
\left(\frac{\epsilon^2\delta}{3C_u}\right)^{1/(2\alpha_Q)}\right\},\\
M&\geq\frac{3C_M}{\epsilon^2\delta},
\qquad M_z\geq\frac{3C_z}{\epsilon^2\delta}.
\end{aligned}
\]
The bound on $h$ can be met by setting $h=T/N_{\mathrm{Euler}}$ with
$N_{\mathrm{Euler}}$ sufficiently large, and the sample sizes can be
chosen as integers. More explicitly, put
\[
h_*:=\min\left\{1,T,
\left(\frac{\epsilon^2\delta}{3C_u(d)}\right)^{1/(2\alpha_Q)}\right\}.
\]
Taking $N_{\mathrm{Euler}}=\lceil T/h_*\rceil$ and taking the ceilings
of the displayed lower bounds for $M$ and $M_z$ gives, by
Lemma~\ref{G_ass_3},
\begin{align*}
N_{\mathrm{Euler}}
&\leq C\left[1+(1+d)^{\kappa_{\mathrm{com}}/(2\alpha_Q)}
 \epsilon^{-1/\alpha_Q}\delta^{-1/(2\alpha_Q)}\right],\\
M+M_z&\leq C(1+d)^{\kappa_{\mathrm{com}}}\epsilon^{-2}\delta^{-1}.
\end{align*}
Here $C$ is independent of $d$, $\epsilon$, and $\delta$, and
$\alpha_Q>0$ is independent of $d$. Then $\E R\leq\epsilon^2\delta$, so Markov's
inequality yields
\[
\Pra(R\geq\epsilon^2)\leq\frac{\E R}{\epsilon^2}\leq\delta.
\]
Taking the complementary event gives
\[
\Pra\left(\left\|u-\mathrm{MC}^{\mathrm{com}}_{h,M,M_z}\right\|_{L^2(\mu_f)}
<\epsilon\right)\geq1-\delta,
\]
as asserted.
\end{proof}

\endgroup

\subsection{Activated neural-network architecture induced by Monte Carlo}
\label{subsubsec:mc_network_architecture}

Theorem~\ref{thm:mc_weighted_l2} motivates an Euler--Monte Carlo architecture. Replacing all random variables in the Monte Carlo estimator by their sampled values yields a deterministic computational graph. In this form, the sampled parameters can be optimized by gradient-based methods, as in physical neural networks.
The precise training procedure is described later. We first show that a modified version of the Monte Carlo estimator, incorporating a 
$\tanh$ activation function, has a bounded $L^2$ error.

Set $\sigma_d(v)_i=\tanh(v_i)$ for $i=1,\ldots,d$, choose
$h=T/N_{Euler}$, and put $t_j=jh$. For $t<T$, define
\[
 j(t)=\lfloor t/h\rfloor+1,\qquad K(t)=N_{Euler}-j(t)+1,
 \qquad \tau_0(t)=t,
\]
\[
 \tau_k(t)=t_{j(t)+k-1}\quad(1\leq k\leq K(t)),\qquad
 \Delta\tau_k(t)=\tau_{k+1}(t)-\tau_k(t)\quad(0\leq k<K(t)).
\]
Thus only the first duration may be shorter than $h$; set $K(T)=0$.

For each branch $m\leq M$ and potential cell $k<N_{Euler}$, sample
$G_{m,k}\sim\mathcal N(0,I_d)$ and use
$\Delta W_k^{t,(m)}=\sqrt{\Delta\tau_k(t)}G_{m,k}$, with law
$\mathcal N(0,\Delta\tau_k(t)I_d)$. Independently sample full counts
$C_{m,k}\sim\operatorname{Poisson}(\lambda(d)h)$ and, conditional on them,
independent timestamps $U_{m,k,n}\sim\operatorname{Uniform}(0,h)$ and marks
$Y_{m,k,n}\sim\rho_d:=\nu_d/\lambda(d)$. Define
\[
 J(\omega)=\sum_{m=1}^M\sum_{k=0}^{N_{Euler}-1}C_{m,k}(\omega),\qquad
 N_k^{t,(m)}=\sum_{n=1}^{C_{m,k}}
 \mathbf 1_{\{U_{m,k,n}\leq\Delta\tau_k(t)\}}.
\]
Then $N_k^{t,(m)}\sim\operatorname{Poisson}(\lambda(d)\Delta\tau_k(t))$;
write $Y_{k,n}^{t,(m)}$ for the retained marks in stored slot order.
Sample $z^1,\ldots,z^{M_z}\sim\rho_d$ independently once and reuse them
in every cell and branch.

Initialize $X_0^{\sigma,m}=x$, $A_0^{\sigma,m}=0$ and, for $k<K(t)$, set
\[
\begin{aligned}
 V_k^{\sigma,m}
 :={}&\Delta\tau_k(t)a(\tau_k(t),X_k^{\sigma,m})
 +\Delta W_k^{t,(m)},\\
 X_{k+1}^{\sigma,m}
 ={}&X_k^{\sigma,m}+\sigma_d(V_k^{\sigma,m})\\
 &+\sum_{n=1}^{N_k^{t,(m)}}
 c(\tau_k(t),X_k^{\sigma,m},Y_{k,n}^{t,(m)})\\
 &-\frac{\Delta\tau_k(t)\lambda(d)}{M_z}\sum_{\ell=1}^{M_z}
 c(\tau_k(t),X_k^{\sigma,m},z^\ell),\\
 A_{k+1}^{\sigma,m}
 ={}&A_k^{\sigma,m}+\Delta\tau_k(t)b(\tau_k(t),X_k^{\sigma,m}).
\end{aligned}
\]
The readout, including $\Phi_{h,M,M_z}^\sigma(T,x;\omega)=g(x)$, is
\[
 G_m^\sigma(t,x;\omega)=g(X_{K(t)}^{\sigma,m})+A_{K(t)}^{\sigma,m},
 \qquad
 \Phi_{h,M,M_z}^\sigma(t,x;\omega)=\frac1M\sum_{m=1}^M G_m^\sigma(t,x;\omega).
\]

Fixed coefficient evaluations $a,b,c,g$, input-dependent grid selection,
square roots, arithmetic, and masks are allowed operations.
The state $(X_k^{\sigma,m},A_k^{\sigma,m})$, its residual update, and the
ensemble readout give a recurrent interpretation.

Count every stored scalar: Gaussian coordinates, frozen full-cell counters,
marks and their timestamps, common
compensator coordinates, and the $M$ averaging weights. Coefficient-oracle
costs and fixed grid constants are excluded, as are runtime indices and
accumulators (computational state). This convention gives
\begin{equation}
\begin{aligned}
 P_{\mathrm{all}}(\omega)
 ={}&M N_{Euler}d+M N_{Euler}
 +(d+1)J(\omega)+dM_z+M.
\end{aligned}
\label{eq:mc_network_parameter_count_exact}
\end{equation}
In the training variant, Gaussian vectors, jump marks, and compensator marks
are trainable, while counters and timestamps
are frozen. Hence
\[
 P_{\mathrm{train}}(\omega)=M N_{Euler}d+d\bigl(J(\omega)+M_z\bigr).
\]
Under this storage convention, the following theorem gives existence and a
polynomial parameter bound for the above recurrence.

\begingroup
\renewenvironment{proof}{\par\noindent{\bfseries Proof\ }}%
	{\unskip\nobreak\hfill\BlackBox\par\vspace{2mm}}
In the lemma and theorem below, $C$ may change between estimates and
depends only on the dimension-independent primary constants and $T$,
not on $d,h,M,M_z,s,x$. Put
\[
 \Lambda_d:=1+\log(1+d),\qquad
 L_2:=L_\jump(d),\quad L_6:=L_{c,6}(d),\quad
 B:=L_a(d)+2\lambda(d)\ell_c(d).
\]
Write $\lambda=\lambda(d)$ and $\nu=\nu_d$. Under the primary
dimension bounds,
\[
 B^2+L_2^2+L_6^6\leq C\Lambda_d,
 \qquad \lambda\leq C\Lambda_d,\qquad
 h_0:=\min\{1,(1+\lambda)^{-1},(1+B^2)^{-1}\},
\]
so $h_0^{-1}\leq C\Lambda_d$. Use the arrays already specified above
and the common compensator sample $Z=(z^1,\ldots,z^{M_z})$, and set
\[
 I(t,y):=\int_{\Z}c(t,y,z)\nu(\dd z),\qquad
 \widehat I_Z(t,y):=\frac{\lambda}{M_z}
 \sum_{\ell=1}^{M_z}c(t,y,z^\ell).
\]
For each fixed input $(s,x)$ suppress only the fixed $s$ in
$K=K(s)$, $\tau_k=\tau_k(s)$ and
$\Delta\tau_k=\Delta\tau_k(s)$. Thus
$0<\Delta\tau_k\leq h$ and $\sum_{k<K}\Delta\tau_k=T-s$;
at $s=T$ there are no steps. For a single branch, suppress its index
and couple $X^\sigma$ to the unactivated empirical Euler state $X^e$
by the same Gaussian and marked Poisson samples and the same $Z$:
\begin{align*}
    X_0^e&=x,\\
 X_{k+1}^e&=X_k^e+\Delta\tau_k a(\tau_k,X_k^e)
 +\sqrt{\Delta\tau_k}G_k
 +\sum_{n=1}^{N_k}c(\tau_k,X_k^e,Y_{k,n})
 -\Delta\tau_k\widehat I_Z(\tau_k,X_k^e).
\end{align*}
 
Let $\mathcal F_k$ be generated by $Z$ and this branch's trajectory
samples strictly before step $k$. Both states at step $k$ are
$\mathcal F_k$-measurable; conditionally, $G_k$ and the next marked
Poisson sample are independent and have their original laws.
On the probability-one event that the finitely many sampled marks
belong to the common set in \ref{Cstar1}, simultaneously for all $t,y,y'$,
\[
 \norm{I(t,y)-I(t,y')}\leq\lambda\ell_c\norm{y-y'},\qquad
 \norm{\widehat I_Z(t,y)-\widehat I_Z(t,y')}
 \leq\lambda\ell_c\norm{y-y'}.
\]

\begin{lemma}
\label{lem:tanh_recurrent_sixth_grid}
Assume \ref{U1}--\ref{U4}, \ref{F1}, \ref{A1}--\ref{A3},
\ref{B1}--\ref{B4}, \ref{C1}--\ref{C5}, \ref{G1}--\ref{G4},
\ref{Cstar1}, and \ref{Cstar2}, with the dimension-dependent intensity
bounds stated in the preliminaries and $T$ independent of $d$.
There is a polynomial $P_6$ independent of $s,h,M,M_z,x$ such that,
for every initial time $s\in[0,T]$, state $x\in\R^d$ and
sample size $M_z\geq1$, with
\[
 h=T/N_{Euler}\leq\min\{T,h_0\},
\]
the following estimate holds:
\[
 \max_{k\leq K}\E\norm{X_k^\sigma}^6
 +\max_{k\leq K}\E\norm{X_k^e}^6
 \leq P_6(d)(1+\norm{x}^6).
\]
\end{lemma}

\begin{proof}
For $j=2,\ldots,6$ define the intercept moments and spatial constants
\begin{align*}
 H_j(t)&:=\int_{\Z}\norm{c(t,0,z)}^j\nu(\dd z),\\
 \mathfrak a_j&:=(L_2^2)^{(6-j)/4}(L_6^6)^{(j-2)/4}.
\end{align*}
At the endpoints set $\mathfrak a_2=L_2^2$ and
$\mathfrak a_6=L_6^6$.
Minkowski's inequality and \ref{C3}, \ref{C5} give
$H_j(t)\leq(K_\jump+T^{\beta_\jump}L_\jump)^j$ for $j=2,6$.
For intermediate $j$, H\"older's inequality gives
\[
 H_j(t)\leq H_2(t)^{(6-j)/4}H_6(t)^{(j-2)/4}.
\]
The same interpolation, now using \ref{C4} for $j=2$ and
\ref{Cstar2} for $j=6$, proves
\begin{align*}
 \int_{\Z}\norm{c(t,x,z)-c(t,0,z)}^j\nu(\dd z)
 &\leq\mathfrak a_j\norm{x}^j,\qquad \mathfrak a_j\leq C\Lambda_d,\\
 A_j(t,x):=\int_{\Z}\norm{c(t,x,z)}^j\nu(\dd z)
 &\leq2^{j-1}\bigl(\mathfrak a_j\norm{x}^j+H_j(t)\bigr).
\end{align*}
Thus all $H_j$ have uniform polynomial bounds.

Fix an active step $k<K$. We estimate the increment directly from
\begin{align*}
 X_{k+1}^\sigma-X_k^\sigma
 ={}&\sigma_d\bigl(\Delta\tau_k a(\tau_k,X_k^\sigma)
				  +\sqrt{\Delta\tau_k}G_k\bigr)\\
 &+\sum_{n=1}^{N_k}c(\tau_k,X_k^\sigma,Y_{k,n})
 -\Delta\tau_k\widehat I_Z(\tau_k,X_k^\sigma).
\end{align*}
All conditional estimates below hold almost surely.
Since $|\tanh'(r)|\leq1$, the coordinatewise map $\sigma_d$ satisfies
\[
 \norm{\sigma_d(x)-\sigma_d(y)}\leq\norm{x-y},
 \qquad x,y\in\R^d.
\]
Moreover, $G_k$ is independent of $\mathcal F_k$ and has the same law
as $-G_k$. The oddness of $\sigma_d$ therefore gives
\[
 \E[\sigma_d(\sqrt{\Delta\tau_k}G_k)\mid\mathcal F_k]=0.
\]
Subtracting this zero expectation and using the Lipschitz bound yield
\begin{align*}
 &\norm{\E\left[
 \sigma_d\bigl(\Delta\tau_k a(\tau_k,X_k^\sigma)
				   +\sqrt{\Delta\tau_k}G_k\bigr)
 \mid\mathcal F_k\right]}\\
 &\quad\leq\E\left[\norm{
 \sigma_d\bigl(\Delta\tau_k a(\tau_k,X_k^\sigma)
				   +\sqrt{\Delta\tau_k}G_k\bigr)
 -\sigma_d(\sqrt{\Delta\tau_k}G_k)
 }\mid\mathcal F_k\right]\\
 &\quad\leq\Delta\tau_k\norm{a(\tau_k,X_k^\sigma)}
 \leq\Delta\tau_k\bigl(K_a+L_a\norm{X_k^\sigma}\bigr).
\end{align*}
For higher moments, no centering is needed. Since
$\sigma_d(0)=0$, we have $\norm{\sigma_d(x)}\leq\norm{x}$.
Also, H\"older's inequality gives
\[
 \E\norm{G_k}^j
 \leq\bigl(\E\norm{G_k}^6\bigr)^{j/6}
 =\bigl[d(d+2)(d+4)\bigr]^{j/6}
 \leq C d^{j/2},\qquad 2\leq j\leq6.
\]
Thus $\norm{p+q}^j\leq2^{j-1}(\norm{p}^j+\norm{q}^j)$ implies
\begin{align*}
 &\E\left[\norm{\sigma_d\bigl(\Delta\tau_k a(\tau_k,X_k^\sigma)
				   +\sqrt{\Delta\tau_k}G_k\bigr)}^j
				   \mid\mathcal F_k\right]\\
 &\quad\leq C(\Delta\tau_k)^j
			  \bigl(K_a^j+L_a^j\norm{X_k^\sigma}^j\bigr)
			  +C(\Delta\tau_k)^{j/2}d^{j/2}.
\end{align*}

We next estimate the jump sum. Put $\theta=\lambda\Delta\tau_k$.
The step-size restriction gives
$0\leq\theta\leq\lambda h_0\leq1$, so $\theta^n\leq\theta$
for every $n\geq1$. Thus, directly from the Poisson probabilities,
\begin{align*}
 \E N_k^j
 &=e^{-\theta}\sum_{n=1}^{\infty}n^j\frac{\theta^n}{n!}\\
 &\leq\theta\sum_{n=1}^{\infty}\frac{n^j}{n!}
 \leq\theta\sum_{n=1}^{\infty}\frac{n^6}{n!}
 \leq C\lambda\Delta\tau_k,\qquad 2\leq j\leq6.
\end{align*}
The constant is finite because $\sum_{n\geq1}n^6/n!$ converges.
Conditionally on $\mathcal F_k$ and $N_k=n\geq1$, the retained
marks have law $\rho_d=\nu/\lambda$. By convexity,
\begin{align*}
 &\E\left[\norm{\sum_{i=1}^{N_k}c(\tau_k,X_k^\sigma,Y_{k,i})}^j
			\mid\mathcal F_k,N_k=n\right]\\
 &\quad\leq n^{j-1}\sum_{i=1}^n
	 \int_{\Z}\norm{c(\tau_k,X_k^\sigma,z)}^j\rho_d(\dd z)
 =\frac{n^j}{\lambda}A_j(\tau_k,X_k^\sigma).
\end{align*}
For $n=0$ the sum vanishes. Averaging over $N_k$, which is independent
of $\mathcal F_k$, therefore gives
\begin{align*}
 &\E\left[\norm{\sum_{n=1}^{N_k}c(\tau_k,X_k^\sigma,Y_{k,n})}^j
			\mid\mathcal F_k\right]\\
 &\quad\leq\frac{\E N_k^j}{\lambda}A_j(\tau_k,X_k^\sigma)
 \leq C\Delta\tau_k
	 \bigl(\mathfrak a_j\norm{X_k^\sigma}^j+H_j(\tau_k)\bigr).
\end{align*}
The same conditioning argument for the first moment gives
\[
 \E\left[\sum_{n=1}^{N_k}c(\tau_k,X_k^\sigma,Y_{k,n})
			\mid\mathcal F_k\right]
 =\Delta\tau_k I(\tau_k,X_k^\sigma).
\]

The compensator is $\mathcal F_k$-measurable. Its spatial Lipschitz
bound gives
\[
 \norm{\widehat I_Z(\tau_k,X_k^\sigma)}
 \leq\lambda\ell_c\norm{X_k^\sigma}
	 +\norm{\widehat I_Z(\tau_k,0)}.
\]
Combining the first-moment estimates for the activation and jump sum
with the spatial bounds for $I$ and $\widehat I_Z$ yields
\begin{align*}
 &\norm{\E[X_{k+1}^\sigma-X_k^\sigma\mid\mathcal F_k]}\\
 &\quad\leq\Delta\tau_k\bigl[
 (L_a+2\lambda\ell_c)\norm{X_k^\sigma}
 +K_a+\norm{I(\tau_k,0)}+\norm{\widehat I_Z(\tau_k,0)}\bigr].
\end{align*}
For $j=2,\ldots,6$, apply
$\norm{p+q+r}^j\leq3^{j-1}(\norm{p}^j+\norm{q}^j+\norm{r}^j)$
to the three terms in the original increment. The estimates above give
\begin{align*}
 &\E[\norm{X_{k+1}^\sigma-X_k^\sigma}^j\mid\mathcal F_k]\\
 &\quad\leq C\bigl[
 (\Delta\tau_k)^j B^j\norm{X_k^\sigma}^j
 +(\Delta\tau_k)^j
	 \bigl(K_a^j+\norm{\widehat I_Z(\tau_k,0)}^j\bigr)\\
 &\hspace{35mm}+(\Delta\tau_k)^{j/2}d^{j/2}
 +\Delta\tau_k\mathfrak a_j\norm{X_k^\sigma}^j
 +\Delta\tau_k H_j(\tau_k)\bigr],
\end{align*}
where we used $L_a\leq B$ and $\lambda\ell_c\leq B$.

We now estimate the sixth moment. Expanding
$(\norm{x}^2+2\langle x,y\rangle+\norm{y}^2)^3$ and using
Cauchy--Schwarz gives, for all $x,y\in\R^d$,
\[
 \norm{x+y}^6\leq\norm{x}^6
 +6\norm{x}^4\langle x,y\rangle
 +C\sum_{j=2}^6\norm{x}^{6-j}\norm{y}^j.
\]
Apply this inequality with $x=X_k^\sigma$ and
$y=X_{k+1}^\sigma-X_k^\sigma$, and condition on $\mathcal F_k$:
\begin{align*}
 \E[\norm{X_{k+1}^\sigma}^6\mid\mathcal F_k]
 \leq{}&\norm{X_k^\sigma}^6\\
 &+6\norm{X_k^\sigma}^4
 \left\langle X_k^\sigma,
	 \E[X_{k+1}^\sigma-X_k^\sigma\mid\mathcal F_k]\right\rangle\\
 &+C\sum_{j=2}^6\norm{X_k^\sigma}^{6-j}
	 \E[\norm{X_{k+1}^\sigma-X_k^\sigma}^j\mid\mathcal F_k].
\end{align*}
The first-moment bound controls the inner-product term by
\[
 6\Delta\tau_k\norm{X_k^\sigma}^5
 \bigl[B\norm{X_k^\sigma}+K_a
	 +\norm{I(\tau_k,0)}+\norm{\widehat I_Z(\tau_k,0)}\bigr].
\]
For the remaining terms, the step-size restriction ensures
\[
 (\Delta\tau_k)^{j-1}B^j\leq1,\qquad
 (\Delta\tau_k)^{j/2}\leq\Delta\tau_k,
 \qquad j=2,\ldots,6.
\]
If $B\leq1$ use $\Delta\tau_k\leq1$; if $B>1$ use
$\Delta\tau_k\leq B^{-2}$, giving
$(\Delta\tau_k)^{j-1}B^j\leq B^{2-j}\leq1$.
Young's inequality gives, for $1\leq j\leq6$ and $r\geq0$,
\[
 \norm{X_k^\sigma}^{6-j}r^j
 \leq\norm{X_k^\sigma}^6+r^6.
\]
Use this with $r=K_a$, $\norm{I(\tau_k,0)}$,
$\norm{\widehat I_Z(\tau_k,0)}$, $\sqrt d$, and
$H_j(\tau_k)^{1/j}$ as appropriate. At $j=6$ the bound is direct.
Since $B\leq C\Lambda_d$ and $\mathfrak a_j\leq C\Lambda_d$,
substitution into the sixth-moment estimate yields
\begin{align*}
 \E[\norm{X_{k+1}^\sigma}^6\mid\mathcal F_k]
 \leq{}&(1+C\Lambda_d\Delta\tau_k)\norm{X_k^\sigma}^6\\
 &+C\Delta\tau_k\biggl[K_a^6
	 +\norm{I(\tau_k,0)}^6
	 +\norm{\widehat I_Z(\tau_k,0)}^6\\
 &\hspace{30mm}+d^3+\sum_{j=2}^6H_j(\tau_k)^{6/j}\biggr].
\end{align*}
It remains to bound the terms at the spatial origin. Since
$\nu=\lambda\rho_d$, Jensen's inequality gives
\begin{align*}
 \norm{I(t,0)}^6
 &=\lambda^6\norm{\int_{\Z}c(t,0,z)\rho_d(\dd z)}^6
 \leq\lambda^5 H_6(t),\\
 \E\norm{\widehat I_Z(t,0)}^6
 &=\lambda^6\E\norm{\frac1{M_z}
			     \sum_{\ell=1}^{M_z}c(t,0,z^\ell)}^6\\
 &\leq\frac{\lambda^6}{M_z}\sum_{\ell=1}^{M_z}
				 \E\norm{c(t,0,z^\ell)}^6
 =\lambda^5 H_6(t).
\end{align*}
These bounds are uniform in $M_z$. Together with the polynomial bounds
on $H_j$, $K_a$, and $\lambda$, they show inductively that the sixth
moments are finite and, for some dimension-independent $C,\gamma>0$,
\[
 \E\norm{X_{k+1}^\sigma}^6
 \leq(1+C\Lambda_d\Delta\tau_k)\E\norm{X_k^\sigma}^6
	 +C\Delta\tau_k(1+d)^\gamma.
\]
Discrete Gronwall and $\sum_{k<K}\Delta\tau_k\leq T$ now give
\[
 \max_{k\leq K}\E\norm{X_k^\sigma}^6
 \leq e^{C\Lambda_dT}\left(\norm{x}^6+C T(1+d)^\gamma\right)
 \leq C(1+d)^{\gamma+CT}(1+\norm{x}^6).
\]
Here $e^{C\Lambda_dT}=e^{CT}(1+d)^{CT}$.
For $X^e$, replace the activation term throughout by
$\Delta\tau_k a(\tau_k,X_k^e)+\sqrt{\Delta\tau_k}G_k$.
Its conditional mean is exactly $\Delta\tau_k a(\tau_k,X_k^e)$,
and its moments satisfy the same bounds obtained above for the
activation term. All jump and compensator estimates remain unchanged,
so the same iteration proves the bound for $X^e$.
Enlarging to a polynomial of integer degree and adding the two bounds
gives $P_6$ as stated. If $s=T$, there are no steps and both initial
states equal $x$, so the conclusion also holds.
\end{proof}

Finally, we are ready to prove the main theorem.
\begin{theorem}
\label{thm:tanh_recurrent_existence}
Assume \ref{U1}--\ref{U4}, \ref{F1}, \ref{A1}--\ref{A3},
\ref{B1}--\ref{B4}, \ref{C1}--\ref{C5}, \ref{G1}--\ref{G4},
\ref{Cstar1}, and \ref{Cstar2}, with the dimension-dependent intensity
bounds stated in the preliminaries. The time horizon $T$ and the exponent
$\alpha_Q>0$ defined earlier are independent of $d$.
There are constants $C,r,q>0$, independent of $d$ and $\varepsilon$,
such that, for every $d\geq1$ and $\varepsilon\in(0,1)$, there exists a
deterministic finite realization of the above recurrence,
$\Phi_{d,\varepsilon}:[0,T]\times\R^d\to\R$, with coordinatewise
activation $\tanh$, whose stored parameters serve all inputs and satisfy
\[
 \norm{u-\Phi_{d,\varepsilon}}_{L^2(\mu_f)}<\varepsilon,
 \qquad
 P_{\mathrm{train}}\leq P_{\mathrm{all}}
 \leq C(1+d)^r\varepsilon^{-q}.
\]
\end{theorem}

\begin{proof}
Fix $(s,x)$ and construct $X^\sigma$ and $X^e$ using the same
Gaussian vectors, jump counts, jump marks, and compensator samples.
Use the filtration $\mathcal F_k$ defined above, which contains the
compensator samples and the trajectory samples before step $k$.
For each active step, put
\[
 V_k:=\Delta\tau_k a(\tau_k,X_k^\sigma)
	 +\sqrt{\Delta\tau_k}G_k.
\]
The difference between the activated and unactivated terms is
$\sigma_d(V_k)-V_k$. We estimate its conditional mean and second
moment directly.

The function $r\mapsto\tanh(r)-r$ is odd, and its derivative is
$-\tanh^2(r)$, whose absolute value is at most $r^2$.
For a scalar standard Gaussian $G$, symmetry gives
$\E[\tanh(\sqrt{\Delta\tau_k}G)-\sqrt{\Delta\tau_k}G]=0$.
Integrating the derivative from zero to the deterministic shift
$\Delta\tau_k a_i$ therefore gives
\begin{align*}
 &\left|\E\left[\tanh(\Delta\tau_k a_i+\sqrt{\Delta\tau_k}G)
			  -\Delta\tau_k a_i-\sqrt{\Delta\tau_k}G\right]\right|\\
 &\quad\leq\int_0^{|\Delta\tau_k a_i|}
			\E(s+\sqrt{\Delta\tau_k}G)^2\,\dd s\\
 &\quad=\int_0^{|\Delta\tau_k a_i|}(s^2+\Delta\tau_k)\,\dd s\\
 &\quad=(\Delta\tau_k)^2|a_i|
   +\tfrac13(\Delta\tau_k)^3|a_i|^3.
\end{align*}
Conditionally on $\mathcal F_k$, apply this to each coordinate of
$a(\tau_k,X_k^\sigma)$. Taking Euclidean norms and using
$\sum_i|a_i|^6\leq\norm{a}^6$ yields
\begin{align*}
 &\norm{\E[\sigma_d(V_k)-V_k\mid\mathcal F_k]}\\
 &\quad\leq(\Delta\tau_k)^2\norm{a(\tau_k,X_k^\sigma)}
 +\tfrac13(\Delta\tau_k)^3\norm{a(\tau_k,X_k^\sigma)}^3.
\end{align*}
Also, integrating $\tanh^2(r)\leq r^2$ gives
$|\tanh(r)-r|\leq|r|^3/3$. Consequently,
\begin{align*}
 \E[\norm{\sigma_d(V_k)-V_k}^2\mid\mathcal F_k]
 &\leq\tfrac19\E[\norm{V_k}^6\mid\mathcal F_k]\\
 &\leq C(\Delta\tau_k)^3\left[d^3+
 (\Delta\tau_k)^3\norm{a(\tau_k,X_k^\sigma)}^6\right].
\end{align*}
Here we used $\norm{p+q}^6\leq32(\norm{p}^6+\norm{q}^6)$ and
$\E\norm{G_k}^6=d(d+2)(d+4)\leq Cd^3$.
Subtracting the conditional mean can only decrease the conditional
second moment, so
\begin{align*}
 &\E\left[\norm{\sigma_d(V_k)-V_k
	 -\E[\sigma_d(V_k)-V_k\mid\mathcal F_k]}^2
	 \mid\mathcal F_k\right]\\
 &\quad\leq\E[\norm{\sigma_d(V_k)-V_k}^2\mid\mathcal F_k].
\end{align*}
By Lemma~\ref{lem:tanh_recurrent_sixth_grid},
$\norm{a(t,y)}\leq K_a+L_a\norm{y}$,
$\Delta\tau_k\leq h\leq1$, and $\sum_k\Delta\tau_k\leq T$,
the preceding estimates imply
\begin{align*}
 &\sum_{k<K}\E\norm{\sigma_d(V_k)-V_k
	 -\E[\sigma_d(V_k)-V_k\mid\mathcal F_k]}^2\\
 &\quad+\sum_{k<K}\frac{
	 \E\norm{\E[\sigma_d(V_k)-V_k\mid\mathcal F_k]}^2}
	 {\Delta\tau_k}
 \leq P_{\mathrm{def}}(d)h^2(1+\norm{x}^6).
\end{align*}
Indeed, each summand is bounded by a polynomial in $d$ times
$(\Delta\tau_k)^3(1+\norm{x}^6)$, and
$\sum_{k<K}(\Delta\tau_k)^3\leq Th^2$.
The bound is uniform in $s$ and $M_z$.

We next compare the two trajectories. Subtracting their updates and
adding and subtracting the conditional means gives the exact identity
\begin{align*}
 X_{k+1}^\sigma-X_{k+1}^e
 ={}&X_k^\sigma-X_k^e\\
 &+\Delta\tau_k\bigl[a(\tau_k,X_k^\sigma)-a(\tau_k,X_k^e)
			+I(\tau_k,X_k^\sigma)-I(\tau_k,X_k^e)\\
 &\hspace{24mm}-\widehat I_Z(\tau_k,X_k^\sigma)
			+\widehat I_Z(\tau_k,X_k^e)\bigr]\\
 &+\E[\sigma_d(V_k)-V_k\mid\mathcal F_k]\\
 &+\sigma_d(V_k)-V_k
	 -\E[\sigma_d(V_k)-V_k\mid\mathcal F_k]\\
 &+\sum_{n=1}^{N_k}\bigl[c(\tau_k,X_k^\sigma,Y_{k,n})
			     -c(\tau_k,X_k^e,Y_{k,n})\bigr]\\
 &-\Delta\tau_k\bigl[I(\tau_k,X_k^\sigma)-I(\tau_k,X_k^e)\bigr].
\end{align*}
The norm of the coefficient in square brackets multiplying
$\Delta\tau_k$ is at most $B\norm{X_k^\sigma-X_k^e}$.
The last two lines together have conditional mean zero, and the
compound Poisson variance identity and \ref{C4} give
\begin{align*}
 &\E\Biggl[\Biggl\Vert
	 \sum_{n=1}^{N_k}\bigl[c(\tau_k,X_k^\sigma,Y_{k,n})
				   -c(\tau_k,X_k^e,Y_{k,n})\bigr]\\
 &\hspace{25mm}-\Delta\tau_k
		 \bigl[I(\tau_k,X_k^\sigma)-I(\tau_k,X_k^e)\bigr]
	 \Biggr\Vert^2\,\Biggm|\,\mathcal F_k\Biggr]\\
 &\quad=\Delta\tau_k\int_{\Z}
	 \norm{c(\tau_k,X_k^\sigma,z)-c(\tau_k,X_k^e,z)}^2\nu(\dd z)\\
 &\quad\leq\Delta\tau_k L_2^2\norm{X_k^\sigma-X_k^e}^2.
\end{align*}
To see the variance identity, condition on $N_k$: for iid vector
marks $v$ with mean $\bar v$, the variance of their random sum is
\[
 \E N_k\E\norm{v-\bar v}^2
 +\operatorname{Var}(N_k)\norm{\bar v}^2
 =\lambda\Delta\tau_k\E\norm{v}^2.
\]
Conditional independence of the Gaussian and Poisson samples makes
the centered activation difference and centered jump sum orthogonal
conditionally on $\mathcal F_k$. Their sum has conditional mean zero
and its conditional second moment is the sum of the two bounds above.

For readability, denote only the expected maximum squared trajectory
error by
\[
 F_n:=\E\max_{0\leq l\leq n}\norm{X_l^\sigma-X_l^e}^2.
\]
It is finite by Lemma~\ref{lem:tanh_recurrent_sixth_grid}.
Both schemes start at $x$, so their initial difference is zero.
For every $0\leq l\leq n\leq K$, sum the difference of the updates
over $k=0,\ldots,l-1$. The left-hand side telescopes, giving
\begin{align*}
 X_l^\sigma-X_l^e
 ={}&\sum_{k<l}\Delta\tau_k\bigl[
	 a(\tau_k,X_k^\sigma)-a(\tau_k,X_k^e)\\
 &\hspace{22mm}+I(\tau_k,X_k^\sigma)-I(\tau_k,X_k^e)
	 -\widehat I_Z(\tau_k,X_k^\sigma)
	 +\widehat I_Z(\tau_k,X_k^e)\bigr]\\
 &+\sum_{k<l}\E[\sigma_d(V_k)-V_k\mid\mathcal F_k]\\
 &+\sum_{k<l}\Biggl[
	 \sigma_d(V_k)-V_k
	 -\E[\sigma_d(V_k)-V_k\mid\mathcal F_k]\\
 &\hspace{18mm}+\sum_{i=1}^{N_k}
	 \bigl[c(\tau_k,X_k^\sigma,Y_{k,i})
		 -c(\tau_k,X_k^e,Y_{k,i})\bigr]\\
 &\hspace{18mm}-\Delta\tau_k
	 \bigl[I(\tau_k,X_k^\sigma)-I(\tau_k,X_k^e)\bigr]\Biggr].
\end{align*}
Empty sums at $l=0$ are zero. The three outer sums are, respectively,
the accumulated coefficient differences, the accumulated conditional
means of the activation error, and the accumulated centered terms.
For each $l$, apply
$\norm{p+q+r}^2\leq3(\norm{p}^2+\norm{q}^2+\norm{r}^2)$
to these three sums. Taking the maximum over $0\leq l\leq n$ then
bounds the squared trajectory error by three times the sum of their
individual squared maxima.

For the first sum, the Lipschitz bounds give, for every $l\leq n$,
\begin{align*}
 &\Biggl\Vert\sum_{k<l}\Delta\tau_k\bigl[
	 a(\tau_k,X_k^\sigma)-a(\tau_k,X_k^e)
	 +I(\tau_k,X_k^\sigma)-I(\tau_k,X_k^e)\\
 &\hspace{35mm}-\widehat I_Z(\tau_k,X_k^\sigma)
	 +\widehat I_Z(\tau_k,X_k^e)\bigr]\Biggr\Vert^2\\
 &\quad\leq B^2\left(\sum_{k<l}
			  \Delta\tau_k\norm{X_k^\sigma-X_k^e}\right)^2\\
 &\quad=B^2\left(\sum_{k<l}\sqrt{\Delta\tau_k}
			  \sqrt{\Delta\tau_k}\norm{X_k^\sigma-X_k^e}\right)^2\\
 &\quad\leq B^2\left(\sum_{k<l}\Delta\tau_k\right)
			  \sum_{k<l}\Delta\tau_k\norm{X_k^\sigma-X_k^e}^2\\
 &\quad\leq TB^2\sum_{k<n}\Delta\tau_k\norm{X_k^\sigma-X_k^e}^2.
\end{align*}
The penultimate step is Cauchy--Schwarz. The last step uses
$\sum_{k<l}\Delta\tau_k\leq T$ and extends a sum of nonnegative
terms from $k<l$ to $k<n$. The resulting bound no longer depends on
$l$, so it also bounds the squared maximum of the first sum.

For the second sum, the same argument gives
\begin{align*}
 &\max_{0\leq l\leq n}
	 \norm{\sum_{k<l}\E[\sigma_d(V_k)-V_k\mid\mathcal F_k]}^2\\
 &\quad\leq\max_{0\leq l\leq n}\left(
	 \sum_{k<l}\sqrt{\Delta\tau_k}\,
	 \frac{\norm{\E[\sigma_d(V_k)-V_k\mid\mathcal F_k]}}
		{\sqrt{\Delta\tau_k}}\right)^2\\
 &\quad\leq T\sum_{k<n}
	 \frac{\norm{\E[\sigma_d(V_k)-V_k\mid\mathcal F_k]}^2}
		{\Delta\tau_k}.
\end{align*}
All denominators are positive because only active steps are included.
For the third group, the partial sums form a square-integrable
martingale. The $L^2$ Doob inequality contributes a factor four;
orthogonality of its increments allows their second moments to be
added. Taking expectations consequently gives
\begin{align*}
 F_n\leq{}&(3TB^2+12L_2^2)\sum_{k<n}\Delta\tau_k F_k\\
 &+3T\sum_{k<n}\frac{
	 \E\norm{\E[\sigma_d(V_k)-V_k\mid\mathcal F_k]}^2}
	 {\Delta\tau_k}\\
 &+12\sum_{k<n}\E\norm{\sigma_d(V_k)-V_k
	 -\E[\sigma_d(V_k)-V_k\mid\mathcal F_k]}^2.
\end{align*}
Insert the summed activation estimates proved above and use
$3TB^2+12L_2^2\leq C\Lambda_d$. Discrete Gronwall then proves
\[
 \E\max_{k\leq K}\norm{X_k^\sigma-X_k^e}^2
 \leq P_{\mathrm{err}}(d)h^2(1+\norm{x}^6).
\]
The Gronwall factor $e^{C\Lambda_dT}$ is polynomially bounded, so
$P_{\mathrm{err}}$ is as well, uniformly in $s,M_z$.

We now show that the same finite set of sampled data defines the network
for all inputs $(s,x)$, while giving the required increment distributions
for each fixed input. We also verify joint measurability, which allows
us to integrate the error over the inputs and subsequently select one
deterministic realization.

Sample the arrays $G_{m,k},C_{m,k},U_{m,k,n},Y_{m,k,n},Z$ specified
before the lemma once, independently of the input. For a fixed $s<T$
and an active step $k$, retain a stored jump precisely when
$U_{m,k,n}\leq\Delta\tau_k(s)$. Each jump is retained independently
with probability $p=\Delta\tau_k(s)/h$. Thus, conditionally on
$C_{m,k}=c$, the retained count $N_k^{s,(m)}$ has distribution
$\operatorname{Binomial}(c,p)$. For every integer $r\geq0$,
\begin{align*}
 \Pra(N_k^{s,(m)}=r)
 &=\sum_{c=r}^{\infty}\binom cr p^r(1-p)^{c-r}
	 e^{-\lambda h}\frac{(\lambda h)^c}{c!}\\
 &=e^{-\lambda h}\frac{(\lambda hp)^r}{r!}
	 \sum_{q=0}^{\infty}\frac{[\lambda h(1-p)]^q}{q!}\\
 &=e^{-\lambda\Delta\tau_k(s)}
	 \frac{[\lambda\Delta\tau_k(s)]^r}{r!}.
\end{align*}
Consequently, $N_k^{s,(m)}\sim
\operatorname{Poisson}(\lambda\Delta\tau_k(s))$.
The retention rule depends on the timestamps, not on the marks.
Hence the retained marks remain iid with law $\rho_d$, independently
of the retained count. Samples from different cells remain independent.
Likewise, the stored vector $G_{m,k}$ gives the increment
$\sqrt{\Delta\tau_k(s)}G_{m,k}$ with distribution
$\mathcal N(0,\Delta\tau_k(s)I_d)$.
Thus, for every fixed $(s,x)$, the two recursions have exactly the
increment distributions used above, and $X^e$ has the law of the
common-compensator scheme in Lemma~\ref{G_ass_3}.
Conditional on the shared sample $Z$, the $M$ branches are iid.

For each fixed sample, implement the recurrence using $N_{Euler}$
cells: perform the update in active cells, leave the state unchanged
in inactive cells, and include a stored jump only when its timestamp
satisfies the retention rule. At $s=T$, return $g(x)$.
The functions $K(s)$, $\tau_k(s)$, and $\Delta\tau_k(s)$ are Borel
on their active domains. For each fixed array of counts, the output
is therefore a finite composition of Borel coefficient evaluations,
arithmetic operations, square roots of nonnegative durations, and
indicator functions. There are only countably many possible count
arrays, so the output is jointly measurable in $(\omega,s,x)$.
This is the measurability needed to apply Tonelli's theorem below.

Finally, the total number of stored jumps is a sum of independent
Poisson counts, so
\[
 J\sim\operatorname{Poisson}(M\lambda T),\qquad \E J=M\lambda T.
\]
In particular, all stored arrays are finite almost surely.
On a probability-one event, all sampled marks also belong to the
common full-measure set in \ref{Cstar1}; this event does not depend
on the input. Freezing any sample in this event gives one finite
deterministic graph defined for every $(s,x)$.
Its stored scalars are those counted in
\eqref{eq:mc_network_parameter_count_exact}, and freezing counters
and timestamps gives $P_{\mathrm{train}}\leq P_{\mathrm{all}}$.

We now integrate the approximation error over the inputs and choose
$h$, $M$, and $M_z$ to make its expected square sufficiently small.
We then show that at least one realization has both error below
$\varepsilon$ and a polynomially bounded number of stored parameters.

Let $G_m^e$ be the output of branch $m$ without activation, namely
\[
\begin{aligned}
G_m^e(s,x) &= g(X_{K(s)}^{e,m})
    +\sum_{k<K(s)}\Delta\tau_k(s)b(\tau_k(s),X_k^{e,m}),\\
\mathrm{MC}^{\mathrm{com}}(s,x)
    &= \frac1M\sum_{m=1}^M G_m^e(s,x).
\end{aligned}
\]
This is the common-compensator estimator from Lemma~\ref{G_ass_3},
constructed from the same stored samples as the activated estimator.
The spatial Lipschitz bounds \ref{G3} and \ref{B3} give
\begin{align*}
 |G_m^\sigma(s,x)-G_m^e(s,x)|
 \leq{}&L_g(d)\norm{X_{K(s)}^{\sigma,m}-X_{K(s)}^{e,m}}\\
 &+L_b(d)\sum_{k<K(s)}\Delta\tau_k(s)
				 \norm{X_k^{\sigma,m}-X_k^{e,m}}\\
 \leq{}&\bigl(L_g(d)+TL_b(d)\bigr)
	 \max_{k\leq K(s)}\norm{X_k^{\sigma,m}-X_k^{e,m}}.
\end{align*}
Squaring and using $(a+b)^2\leq2a^2+2b^2$ therefore yields
\[
 |G_m^\sigma(s,x)-G_m^e(s,x)|^2
 \leq2(L_g(d)^2+T^2L_b(d)^2)
	 \max_{k\leq K(s)}\norm{X_k^{\sigma,m}-X_k^{e,m}}^2.
\]
For the branch average, Jensen's inequality gives pointwise
\[
 |\Phi_{h,M,M_z}^\sigma(s,x)-\mathrm{MC}^{\mathrm{com}}(s,x)|^2
 \leq\frac1M\sum_{m=1}^M|G_m^\sigma(s,x)-G_m^e(s,x)|^2.
\]
Integrate this inequality against $\mu_f$ and take expectations.
Tonelli's theorem allows us to exchange the expectation and input
integral, and the trajectory estimate proved above then gives
\begin{align*}
 &\E\norm{\Phi_{h,M,M_z}^\sigma-
		     \mathrm{MC}^{\mathrm{com}}}_{L^2(\mu_f)}^2\\
 &\quad\leq2\bigl(L_g(d)^2+T^2L_b(d)^2\bigr)
	 P_{\mathrm{err}}(d)h^2
	 \int(1+\norm{x}^6)\,\mu_f(\dd s,\dd x).
\end{align*}
In particular,
\begin{align*}
 \left(\E\norm{\Phi_{h,M,M_z}^\sigma-
 \mathrm{MC}^{\mathrm{com}}}_{L^2(\mu_f)}^2\right)^{1/2}
 &\leq A_\sigma(d)h,\\
 A_\sigma(d)^2&:=2(L_g(d)^2+T^2L_b(d)^2)P_{\mathrm{err}}(d)\\
 &\quad\times\int(1+\norm{x}^6)\,\mu_f(\dd s,\dd x).
\end{align*}
The sixth-moment bound on $\mu_f$ in \ref{F1} and the primary
dimension bounds show that $A_\sigma(d)$ grows at most polynomially
in $d$. Moreover, \ref{B4} and \ref{G4} imply
$|G_m^\sigma|,|G_m^e|\leq M_g(d)+TM_b(d)$.
Since $\mu_f$ is finite by \ref{F1}, these outputs belong to
$L^2(\mu_f)$ for every finite valid sample.

To estimate the total error, write
$u-\Phi_{h,M,M_z}^\sigma=(u-\mathrm{MC}^{\mathrm{com}})
+(\mathrm{MC}^{\mathrm{com}}-\Phi_{h,M,M_z}^\sigma)$.
Apply the triangle inequality in $L^2(\Omega\times\mu_f)$ and
Lemma~\ref{G_ass_3}, using
$\sqrt{a+b+c}\leq\sqrt a+\sqrt b+\sqrt c$ for nonnegative terms.
Set $\alpha_*:=\min\{\alpha_Q,1\}$. Since
$h\leq h_0\leq1$, both $h^{\alpha_Q}$ and $h$ are at most
$h^{\alpha_*}$. Consequently,
\begin{align*}
 \left(\E\norm{u-\Phi_{h,M,M_z}^\sigma}_{L^2(\mu_f)}^2\right)^{1/2}
 &\leq\sqrt{C_u(d)}h^{\alpha_Q}+A_\sigma(d)h\\
 &\quad+\frac{\sqrt{C_M(d)}}{\sqrt M}
	+\frac{\sqrt{C_z(d)}}{\sqrt{M_z}}\\
 &\leq K(d)\bigl(h^{\alpha_*}+M^{-1/2}+M_z^{-1/2}\bigr),
\end{align*}
where, for some dimension-independent $C_0,\kappa>0$,
\[
 K(d):=\max\{1,\sqrt{C_u(d)}+A_\sigma(d),
		   \sqrt{C_M(d)},\sqrt{C_z(d)}\}
 \leq C_0(1+d)^\kappa.
\]
Choose the step size and sample sizes as follows:
\begin{align*}
 h_*&:=\min\left\{T,h_0,
	 \left(\frac{\varepsilon}{6K(d)}\right)^{1/\alpha_*}\right\},\\
 N_{Euler}&:=\lceil T/h_*\rceil,\qquad h:=T/N_{Euler},\\
 M&:=\left\lceil36K(d)^2\varepsilon^{-2}\right\rceil,
 \qquad M_z:=M.
\end{align*}
The ceiling ensures $h=T/N_{Euler}\leq h_*$, so this step size is
admissible. Each of the three terms in the preceding error bound is
at most $\varepsilon/6$:
\[
 K(d)h^{\alpha_*}\leq\frac\varepsilon6,\qquad
 \frac{K(d)}{\sqrt M}\leq\frac\varepsilon6,\qquad
 \frac{K(d)}{\sqrt{M_z}}\leq\frac\varepsilon6.
\]
Thus the measurable nonnegative random variable
$R:=\norm{u-\Phi_{h,M,M_z}^\sigma}_{L^2(\mu_f)}^2$
satisfies $\E R\leq\varepsilon^2/4$.
We also have $\E J=M\lambda(d)T$. Markov's inequality therefore gives
\[
 \Pra(R\geq\varepsilon^2)\leq\tfrac14,\qquad
 \Pra(J>4M\lambda(d)T)\leq\tfrac14.
\]
The probability that either bound fails is at most
$1/4+1/4=1/2$. Hence
\[
 \Pra\bigl(R<\varepsilon^2,\ J\leq4M\lambda(d)T\bigr)\geq\tfrac12.
\]
Intersecting this event with the probability-one set of finite valid
samples does not change its probability. It is therefore nonempty;
choose $\omega_*$ in this intersection and put
$\Phi_{d,\varepsilon}:=\Phi_{h,M,M_z}^\sigma(\cdot;\omega_*)$.
All sampled data are now fixed, independently of the input, and
$R(\omega_*)<\varepsilon^2$ gives
\[
 \norm{u-\Phi_{d,\varepsilon}}_{L^2(\mu_f)}<\varepsilon.
\]

It remains to bound the number of parameters for this same realization.
Insert $J(\omega_*)\leq4M\lambda(d)T$ into
\eqref{eq:mc_network_parameter_count_exact} to obtain
\[
 P_{\mathrm{all}}(\omega_*)\leq MN_{Euler}(d+1)
	 +4(d+1)M\lambda(d)T+dM_z+M.
\]
The elementary bound $\lceil a\rceil\leq a+1$ gives
\[
 N_{Euler}\leq1+T\max\left\{T^{-1},h_0^{-1},
			   \left(\frac{6K(d)}\varepsilon\right)^{1/\alpha_*}
						 \right\}.
\]
Use $h_0^{-1}\leq C\Lambda_d\leq C(1+d)$,
$K(d)\leq C_0(1+d)^\kappa$, and $0<\varepsilon<1$.
Together with the analogous ceiling bound for $M=M_z$, this yields
\begin{align*}
 N_{Euler}&\leq C(1+d)^{\kappa/\alpha_*+1}
		     \varepsilon^{-1/\alpha_*},\\
 M+M_z&\leq C(1+d)^{2\kappa}\varepsilon^{-2}.
\end{align*}
Using also $\lambda(d)\leq C(1+d)$ in the storage bound proves
\[
 P_{\mathrm{train}}\leq P_{\mathrm{all}}(\omega_*)
 \leq C(1+d)^{2+2\kappa+\kappa/\alpha_*}
		\varepsilon^{-(2+1/\alpha_*)}.
\]
Thus $r=2+2\kappa+\kappa/\alpha_*$ and $q=2+1/\alpha_*$ are
admissible and dimension-independent. This is existence of a deterministic
fixed-operation graph.
\end{proof}
\endgroup

\begin{figure}[htbp]
\centering
\resizebox{0.98\textwidth}{!}{%
\begin{tikzpicture}[
>=Latex,
node distance=5mm and 11mm,
every node/.style={draw,rounded corners,align=center,font=\small},
input/.style={minimum width=25mm,minimum height=9mm,fill=blue!8},
euler/.style={minimum width=67mm,minimum height=22mm,fill=green!8},
accum/.style={minimum width=50mm,minimum height=17mm,fill=yellow!14},
output/.style={minimum width=67mm,minimum height=11mm,fill=orange!10},
weights/.style={minimum width=38mm,minimum height=23mm,fill=red!12},
repeat/.style={draw=none,align=left,font=\small\itshape}
]
\node[input] (input) {input\\$t,x$};

\node[euler,below=16mm of input] (stepone) {first Euler step, $\Delta\tau_0$\\
$V_0^{m}=\Delta\tau_0 a(\tau_0,X_0^{m})+\Delta W_0^{m}$\\
$X_1^{m}=X_0^{m}+\sigma_d(V_0^{m})$\\
$\qquad+\Delta J_0^{m}-\Delta\tau_0\widehat I_0^{M_z}(X_0^{m})$};
\node[euler,below=of stepone] (steptwo) {next Euler step, $\Delta\tau_1$\\
$V_1^{m}=\Delta\tau_1 a(\tau_1,X_1^{m})+\Delta W_1^{m}$\\
$X_2^{m}=X_1^{m}+\sigma_d(V_1^{m})$\\
$\qquad+\Delta J_1^{m}-\Delta\tau_1\widehat I_1^{M_z}(X_1^{m})$};
\node[draw=none,below=of steptwo,minimum height=7mm] (dots) {$\vdots$};
\node[euler,below=of dots] (stepk) {last Euler step, $\Delta\tau_{K-1}$\\
$V_{K-1}^{m}=\Delta\tau_{K-1}a(\tau_{K-1},X_{K-1}^{m})+\Delta W_{K-1}^{m}$\\
$X_K^{m}=X_{K-1}^{m}+\sigma_d(V_{K-1}^{m})$\\
$\qquad+\Delta J_{K-1}^{m}-\Delta\tau_{K-1}\widehat I_{K-1}^{M_z}(X_{K-1}^{m})$};

\node[accum,left=of stepone,anchor=east] (accumone) {first accumulator step\\
$A_0^{m}=0$\\
$A_1^{m}=A_0^{m}+\Delta\tau_0 b(\tau_0,X_0^{m})$};
\node[accum,below=of accumone] (accumtwo) {next accumulator step\\
$A_2^{m}=A_1^{m}+\Delta\tau_1 b(\tau_1,X_1^{m})$};
\node[draw=none,below=of accumtwo,minimum height=7mm] (accumdots) {$\vdots$};
\node[accum,below=of accumdots] (accumk) {last accumulator step\\
$A_K^{m}=A_{K-1}^{m}+\Delta\tau_{K-1}b(\tau_{K-1},X_{K-1}^{m})$};

\node[output,below=12mm of stepk,xshift=-35mm] (readout) {merged branch output\\
$G_m=g(X_K^{m})+A_K^{m}$};
\node[weights,right=of steptwo] (weights) {shared compensator\\weights\\
$z_1,\ldots,z_{M_z}$\\
$\widehat I_k^{M_z}=M_z^{-1}\sum_{r=1}^{M_z}c(\tau_k,X_k,z_r)$};
\begin{scope}[on background layer]
\node[draw,rounded corners,thick,inner sep=8mm,fit=(accumone)(accumtwo)(accumdots)(accumk)(stepone)(steptwo)(dots)(stepk)(readout)] (pathbox) {};
\end{scope}

\node[repeat,anchor=south east] at (pathbox.south east) {$\text{one displayed path}$\\
$\text{is repeated for }m=1,\ldots,M$};

\node[output,below=10mm of pathbox] (average) {Monte Carlo output\\
$M^{-1}\sum_{m=1}^{M}G_m$};

\draw[->,thick] (input) -- (stepone);
\draw[->,thick] (stepone) -- (steptwo);
\draw[->,thick] (steptwo) -- (dots);
\draw[->,thick] (dots) -- (stepk);
\draw[->,thick] (stepone.west) -- (accumone.east);
\draw[->,thick] (steptwo.west) -- (accumtwo.east);
\draw[->,thick] (stepk.west) -- (accumk.east);
\draw[->,thick] (accumone) -- (accumtwo);
\draw[->,thick] (accumtwo) -- (accumdots);
\draw[->,thick] (accumdots) -- (accumk);
\draw[->,thick] (stepk.south) to[out=-90,in=100] (readout.north east);
\draw[->,thick] (accumk.south) to[out=-90,in=140] (readout.north west);
\draw[->,thick] (readout.south) to[out=-90,in=90,looseness=1.6] (average.north);
\draw[->,dashed] (weights.west) to[out=180,in=0] (stepone.east);
\draw[->,dashed] (weights.west) -- (steptwo.east);
\draw[->,dashed] (weights.west) to[out=180,in=0] (stepk.east);
\end{tikzpicture}
}
\caption{An example of the neural-network architecture constructed in
Theorem~\ref{thm:tanh_recurrent_existence}.}
\label{fig:activated_network_architecture}
\end{figure}

        \let\XX\undefined
        \let\X\undefined


    \section{Numerical experiments}

We study three manufactured jump-diffusion terminal-value problems on
\([0,1]\times\mathbb{R}^d\), each designed so that the exact solution is known and the source
term is reconstructed directly from the PIDE. In all cases the benchmark is built from a fixed
common set of \(10{,}000\) independent random points, and the same metrics are reported for each
method: RMSE, MAE, maximum absolute error.

The Euler--Monte Carlo network used in the notebook is configured with
\(N_{\mathrm{Euler}}=32\) time steps, \(M_z=32\) jump-compensator samples, and
\(M=1000\) trajectories.
The activation is chosen as \texttt{tanh} and identity.
The corresponding Picard solver uses the same time discretization \(32\) points.
The number of spatial points is set to \(1000\).
We use \(256\) random Fourier features and set iterations limit to 
\(100\) or reaching a relative update smaller than \(10^{-4}\). 
Relaxation parameter \(\omega=0.25\), ridge penalty \(\alpha=10^{-5}\).

We report the results of the network immediately after initialization, without training. 
Although training decreased the PINN loss, the error metrics with respect to the reference solution became worse.
We presume this is caused by approximation errors in the training objective: the loss requires the gradient, the Laplacian and the time derivative
of the network, which are only approximated, so a lower loss does not imply a more accurate solution.

\subsection{First experiment: three-dimensional Merton basket}
The first study is the basket jump-diffusion model
\begin{align*}
-\partial_t u(t,x)&=\frac{1}{2}\Delta u(t,x)\\
&\qquad+\lambda\,\mathbb{E}_{Z}\bigl[u(t,x+c(t,x,Z))-u(t,x)-c(t,x,Z)\cdot\nabla u(t,x)\bigr]\\
&\qquad+\mu\cdot\nabla u(t,x),
\end{align*}
with terminal condition \(u(T,x)=g(x)\) and time horizon \(T=1\). The implementation uses
\(d=3\), mark distribution \(Z\sim\mathcal N(-0.4,0.6^2)\), jump intensity
\(\lambda=3\), and a common jump direction determined by a vector of weights
\(w=(0.5,1.0,1.5)\). The drift is constant,
\(\mu=(0.0374,0.0011,-0.0570)\), and the exact solution is the geometric basket
\(u(t,x)=\exp\bigl(\beta (T-t)+\theta\cdot x\bigr)\) with
\(\theta=(1/d,\ldots,1/d)\). The Picard iteration reaches the stopping criterion after
35 iterations. At the final iteration, the relative update is approximately
\(9\times10^{-5}\), while the recorded RMSE is about \(6.43\times10^{-1}\) and the MAE is about
\(1.86\times10^{-1}\). This is the best-behaved case among the three benchmark examples.
\begin{figure}[ht]
\centering
\includegraphics[width=\linewidth]{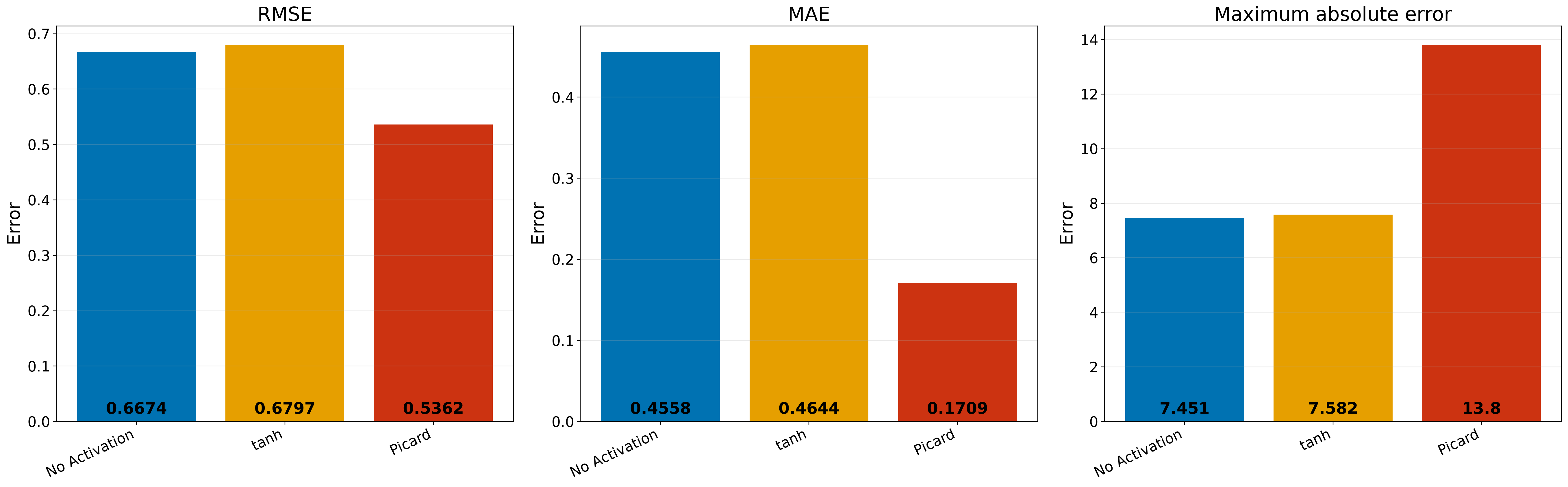}
\caption{Errors for the Merton basket example.}
\label{fig:basket}
\end{figure}

The Figure \ref{fig:basket} shows the metrics for Picard method and the initialized networks. We can observe that the Picard method is the best one, beating networks approach roughly by \(20\%\) in terms of RMSE or MAE.

\begin{figure}[ht]
\centering
\begin{subfigure}[t]{0.97\textwidth}
\centering
\includegraphics[width=\linewidth]{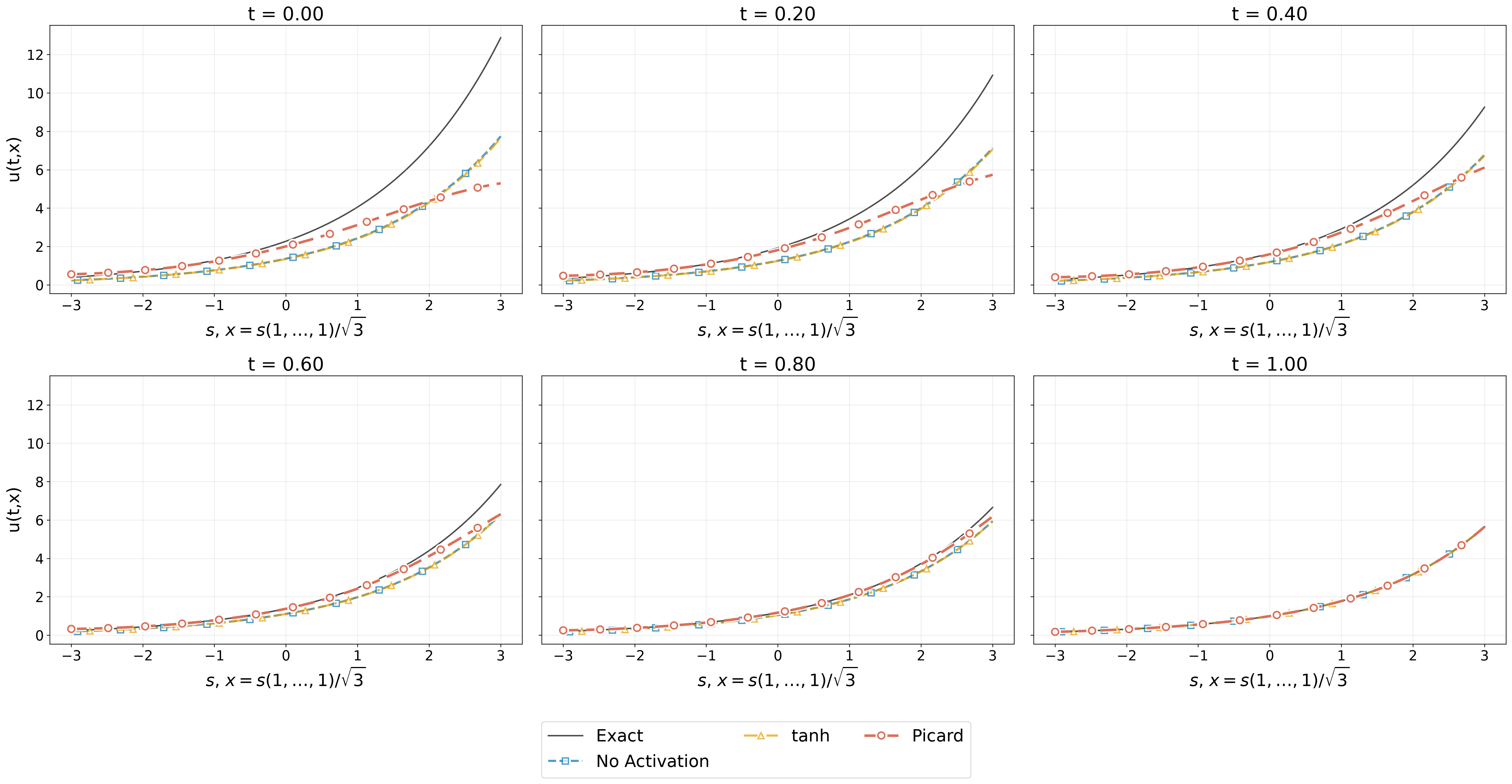}
\caption{Spatial slices for the Merton basket example.}
\end{subfigure}

\vspace{0.75em}

\begin{subfigure}[t]{0.97\textwidth}
\centering
\includegraphics[width=\linewidth]{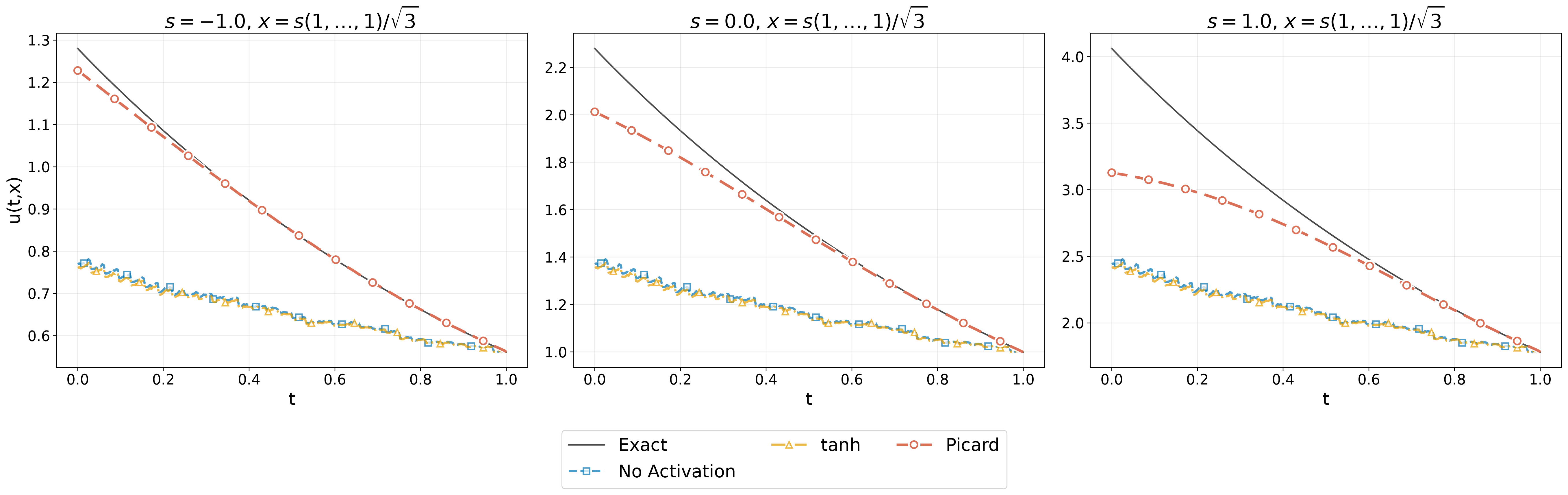}
\caption{Time slices for the Merton basket example.}
\end{subfigure}
\caption{Spatial and time slices for the Merton basket example.}
\label{fig:basket_slices}
\end{figure}

The Figure \ref{fig:basket_slices} shows the spatial and time slices of the exact solution, Picard approximation and the initialized networks predictions.
We can observe that the Picard method and networks perform similarly in spatial slices and the Picard method performs significantly better than the networks in time slices.

\subsection{Second experiment: ten-dimensional multiplicative inventory model}
The second experiment uses the state-dependent jump model
\begin{align*}    
-\partial_t u(t,x)&=\frac{1}{2}\Delta u(t,x)\\
&\quad+\lambda\,\mathbb{E}_{Z}\bigl[u(t,x+\kappa Zx)-u(t,x)-\kappa Zx\cdot\nabla u(t,x)\bigr]\\
&\quad-\gamma x\cdot\nabla u(t,x)+b(t,x),
\end{align*}

with \(d=10\), \(Z\sim\mathcal N(-0.05,0.15^2)\), \(\lambda=2\), \(\gamma=3\), and
\(\kappa=1\). The source term $b(t,x)$ is chosen so satisfy the PIDE with
the exact solution
\begin{equation}
\begin{aligned}
 u(t,x) &= \tfrac12 A(t)\|x\|^2 + C(t)\bigl(h\cdot x\bigr) + D(t), \\
 A(t) &= 1+\tfrac12(T-t),\quad C(t)=0.3+0.2(T-t),\quad D(t)=0.1(T-t).
\end{aligned}
\end{equation}
where \(h=(0.5, \dots, 1.5)/\sqrt{10}\). 
The Picard solver reports convergence after 54 iterations.

\begin{figure}[ht]
\centering
\includegraphics[width=\linewidth]{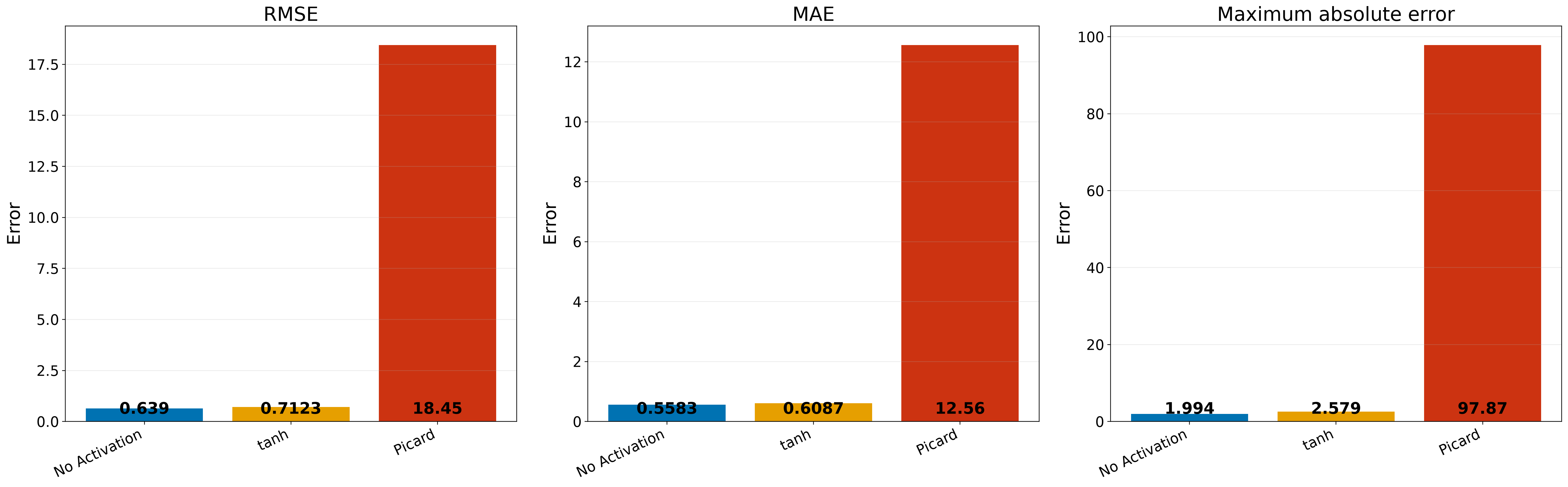}
\caption{Errors for the inventory ten dimensional model.}
\label{fig:inventory}
\end{figure}
The Figure \ref{fig:inventory} shows the errors for the inventory ten dimensional model.
We can observe that the Picard method is the worse one, The networks approach are better roughly by \(95\%\).
To confirm that the implementation is correct, we perform a sanity check by runing the same eaxmaple but in three dimensions space.
\begin{figure}[ht]
\centering
\includegraphics[width=\linewidth]{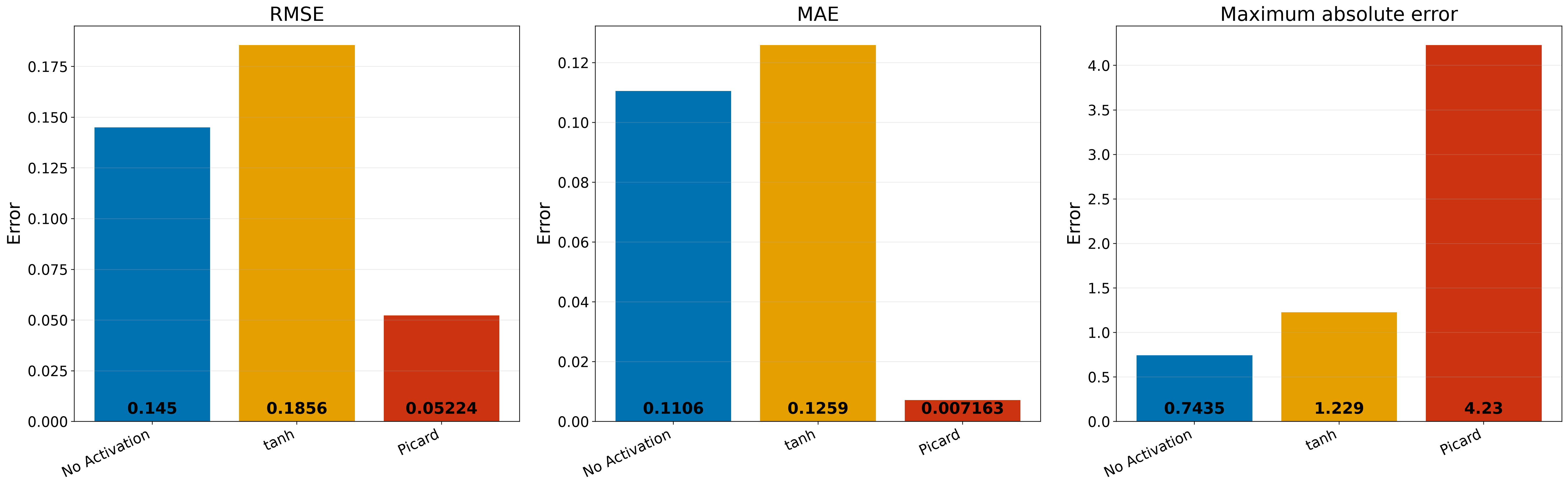}
\caption{Errors for the inventory three dimensional model.}
\label{fig:inventory_3d}
\end{figure}
The Figure \ref{fig:inventory_3d} shows the errors for the inventory three dimensional model, and turns out that the Picard method is the best one in three dimensional space.
This confirms that the implementation is correct and the Picard method is not stable enough in ten dimensions.
For $d=3$ the Picard scheme performs very well, but for $d=10$ the error with respect to the exact solution increases although the relative update decreases.
We attribute it to the dimension dependent approximation error of the regression

\subsection{Third experiment: thirty-two-dimensional oscillatory ridge problem}
The third experiment is a high-dimensional oscillatory ridge model
\begin{equation}
\begin{aligned}
-\partial_t u(t,x) &= \frac{1}{2}\Delta u(t,x)
+\lambda\,\mathbb{E}_{Z}\bigl[u(t,x+Zw)-u(t,x)-Zw\cdot\nabla u(t,x)\bigr] \\
&\quad +\bigl(-\kappa x+\mu\bigr)\cdot\nabla u(t,x)+b(t,x).
\end{aligned}
\end{equation}
with dimension \(d=32\), jump law \(Z\sim\mathcal N(1.5,1.0^2)\), jump intensity
\(\lambda=2\), and parameters \(\kappa=7\), \(\phi=0.3\). The exact solution is
\begin{align*}
 u(t,x) &= A(t)\sin\bigl(\theta\cdot x+\phi\bigr), \\
 A(t) &= 1+\tfrac12(T-t),\\
 \theta_i &= \frac{2}{\sqrt{32}}(0.5,\dots,1.5),\\
 w&=(1/\sqrt{32},\dots,1/\sqrt{32}).
\end{align*}
In this example the
Picard solver does not satisfy the tolerance within \(100\) iterations. The last reported
relative update is approximately \(2.64\times10^{-4}\).

\begin{figure}[ht]
\centering
\includegraphics[width=\linewidth]{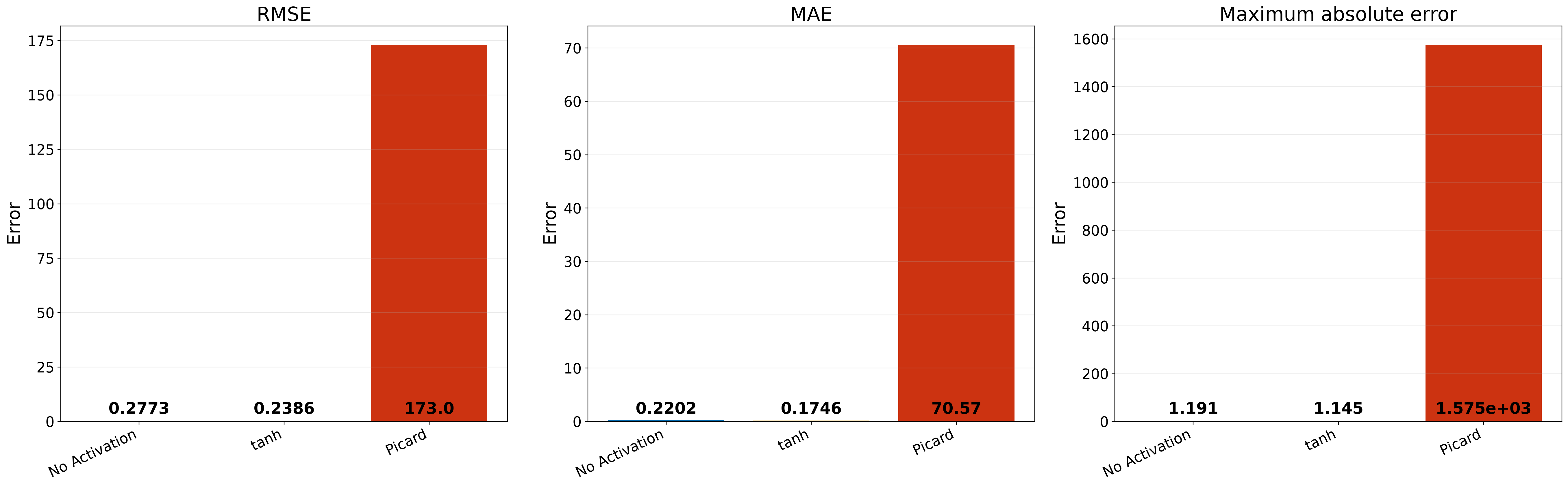}
\caption{Errors for the oscillatory ridge model.}
\label{fig:ridge}
\end{figure}

The Figure \ref{fig:ridge} shows the errors for the oscillatory ridge model. The results are similar to the second experiment, the Picard method is the worst by almost two orders of magnitude.

    \section{Conclusions}

In this paper we constructed an explicit neural network architecture, 
more precisely a computational graph, for approximating solutions of high-dimensional partial integro-differential equations.
Both the number of parameters of the construction and the corresponding approximation error bound depend at most polynomially on the dimension,
so the approach overcomes the curse of dimensionality. This makes it a promising tool for high-dimensional PIDEs.

The numerical experiments support this conclusion.
In low dimensions, the Picard method performs better than the proposed approach.
As the dimension increases, its accuracy significantly drops, while the error of our method remains controlled and does not blow up. 
The proposed approach is therefore particularly suitable for the high-dimensional regime.

    \bibliographystyle{plainnat}
\bibliography{bibliography}
\newpage

\appendix



\end{document}